\documentclass[letterpaper,11pt]{article}
\usepackage{amsmath, amsthm, amssymb}
\usepackage{graphicx}
\usepackage[colorlinks=true, urlcolor=blue, citecolor=black, linkcolor=black]{hyperref}
\usepackage[margin=1in]{geometry}
\usepackage{subcaption}
\usepackage{tikz}
\usepackage{multirow}
\usepackage{cleveref}
\usepackage{array}
\newcolumntype{P}[1]{>{\raggedright\arraybackslash}p{#1}}

\newtheorem{theorem}{Theorem}
\newtheorem{proposition}{Proposition}
\newtheorem{corollary}{Corollary}
\theoremstyle{definition}
\newtheorem{example}{Example}
\newtheorem{definition}{Definition}
\theoremstyle{remark}
\newtheorem{note}{Note}
\newtheorem*{oeis*}{OEIS data}

\newcommand{\id}{\operatorname{id}}

\newcommand{\tid}[1]{t_\mathrm{id}^{#1}}
\newcommand{\trr}[1]{t_{r^2}^{#1}}
\newcommand{\tr}[1]{t_{r}^{#1}}
\newcommand{\tf}[1]{t_{\!f}^{#1}}
\newcommand{\trf}[1]{t_{rf}^{#1}}

\newcommand{\upTri}[3]{
    \fill[xshift=#1 cm, yshift=#2*1.732cm, blue!50!white, rotate around={#3:(1,1.732/3)}]
        (0,0) -- (2,0) -- (1,{sqrt(3)})--cycle;
    \fill[xshift=#1 cm, yshift=#2*1.732cm, white, rotate around={#3:(1,1.732/3)}]
        (0,0) -- (2,0) -- (1,{sqrt(3)/3})--cycle;
    \draw[xshift=#1 cm, yshift=#2*1.732cm, rotate around={#3:(1,1.732/3)}, ultra thick, black, rounded corners=0.001mm]
        (0,0) -- (2,0) -- (1,{sqrt(3)})--cycle;
}

\newcommand{\oeis}[1]{\href{https://oeis.org/#1}{#1}}
\DeclareMathOperator{\lcm}{lcm}

\title{Escher's Cubes: Tiling the Faces of Polyhedra}
\date{}
\author{Peter Kagey\textsuperscript{1} and William Keehn\textsuperscript{2}
\vspace{10pt}\\
\textsuperscript{1}Department of Mathematics and Statistics, Cal Poly Pomona; pkagey@cpp.edu\\
\textsuperscript{2}Prison Mathematics Project}

\begin{document}
\maketitle
\begin{abstract}
  Isohedral or face-transitive polyhedra are polyhedra with identical faces, such that any face can be mapped onto any other via a symmetry of the polyhedron. Examples include the Platonic solids, Catalan solids, bipyramids, and trapezohedra.
  We use the orbit-counting theorem to count the number of ways of tiling the faces of such polyhedra up to their isometries given any arbitrary set of tile designs.
  This approach also counts tilings fixed under each isometry of a given polyhedron and provides explicit enumeration formulas.
  We use this framework to recover sixteen sequences from the On-Line Encyclopedia of Integer Sequences and fill in the gaps by contributing twelve new sequences.
\end{abstract}

\section{Introduction}

In May 1942, the artist M.C. Escher recorded visual experiments in which he took four identical square stamps, each with a rotationally asymmetric pattern (illustrated in Figure \ref{fig:escherTile_1}), placed them together to make a  larger \(2 \times 2\) stamp (perhaps with some of the smaller stamps rotated), and used this stamp to fill up the plane.
As described by Schattschneider \cite{Schattschneider1990, Schattschneider}, Escher systematically cataloged all \(23\) essentially different tilings of the plane up to rigid motion that could be generated by making stamps in this way---equivalently, counting the tilings of \(2 \times 2\) toroidal grid.

\begin{figure}[h!tbp]
  \centering
  \begin{minipage}[b]{0.3\textwidth}
      \centering
      \includegraphics[width=\linewidth]{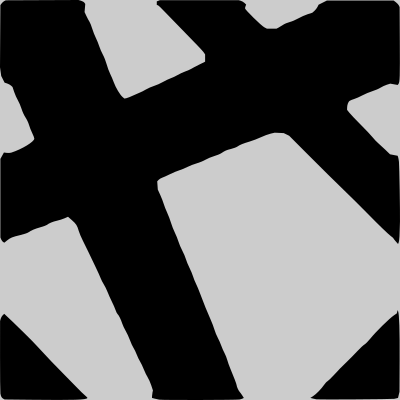}
      \subcaption{Escher-inspired tile design} \label{fig:escherTile_1}
  \end{minipage}
  \hfill
  \begin{minipage}[b]{0.3\textwidth}
    \centering
    \includegraphics[width=\linewidth]{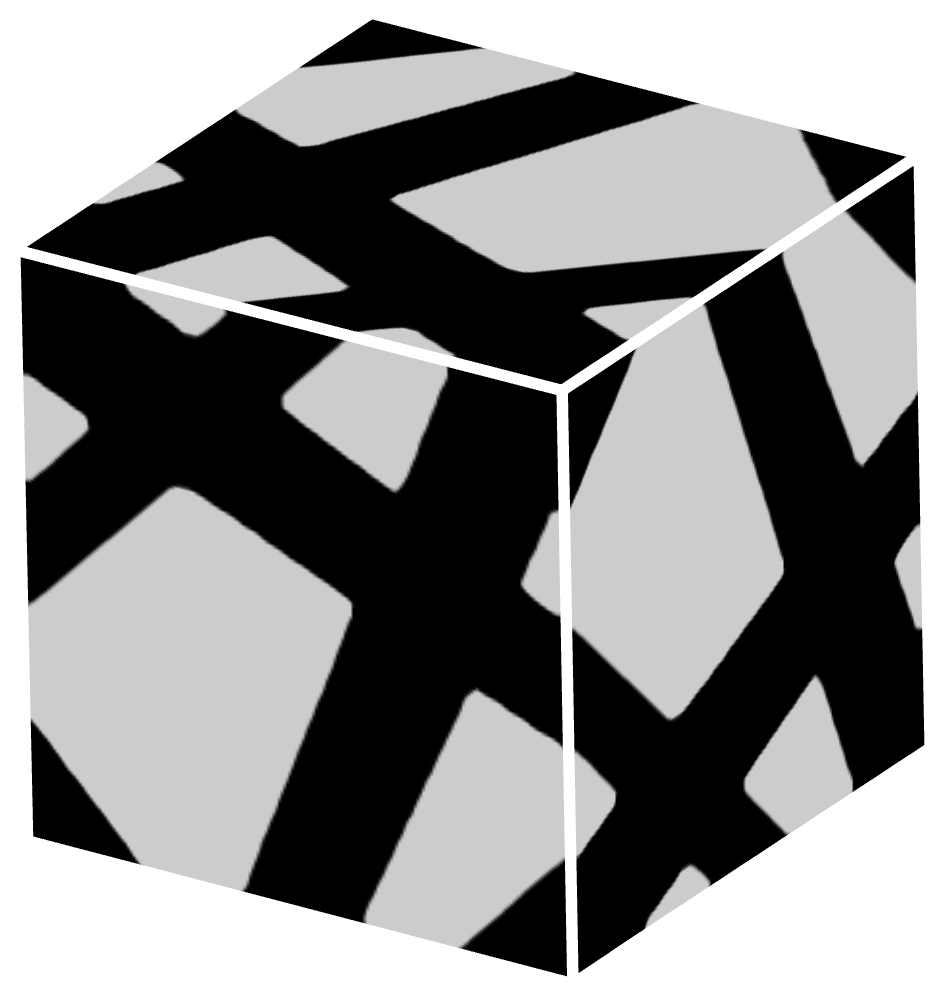}
    \subcaption{Tiling of the faces of a cube} \label{fig:escherTile_2}
  \end{minipage}
  \hfill
  \begin{minipage}[b]{0.25\textwidth}
    \centering
    \includegraphics[width=\linewidth]{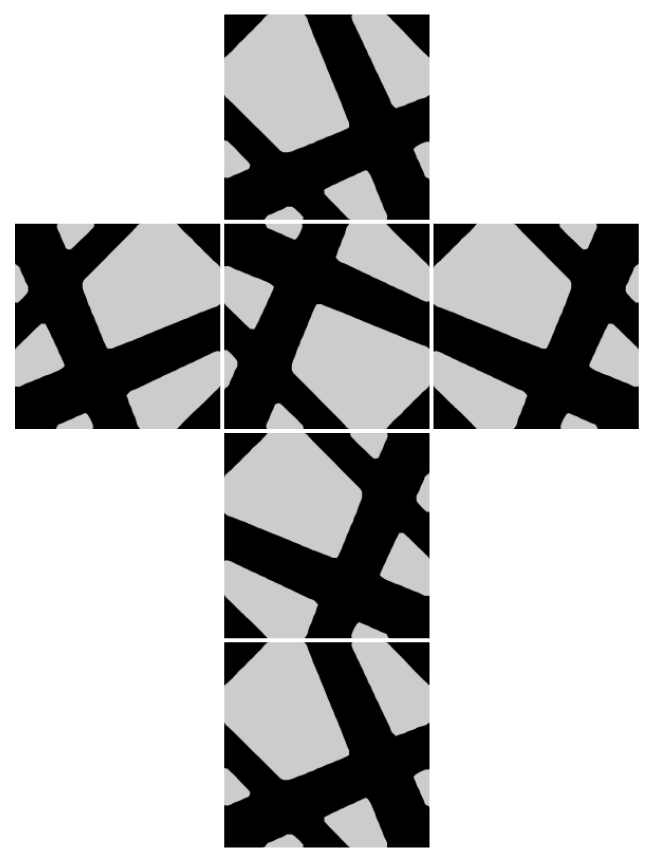}
    \subcaption{Net of a cube tiling} \label{fig:escherTile_3}
  \end{minipage}
  \caption{An example of a tile and a tiling of a cube. There are 5548 distinct tilings by the above tile up to rotation and reflection of the cube.}
  \label{fig:escherTile}
\end{figure}

In our previous paper \cite{KageyKeehn}, we counted tilings of the \(n \times m\) grid, cylindrical grid, and toroidal grid. We now extend these techniques to count the number of face tilings of various families of isohedral (face-transitive) polyhedra, including Platonic solids, Catalan solids, bipyramids, and trapezohedra. Given a polyhedron and a set of \textit{tile designs} with which to decorate the faces, we enumerate the tilings of the faces of the polyhedron up to symmetry of the polyhedron. When the tile designs are \emph{maximally symmetric}, this corresponds to face colorings, and we can use the same technique to enumerate the vertex colorings of each polyhedron's dual. We also use results from the previous paper to enumerate tilings of bipyramids and trapezohedra in bijection with tilings of cylindrical grids.

We are also interested in counting the number of tilings up to \textit{subgroups} of a polyhedron's symmetry group. In some cases, this may be directly of interest: for instance, we may have a physical object with both rotational and reflectional symmetries, but we are only interested in the rotational symmetries which we can perform in 3D space. Other times, we can see that the tilings of one object up to a subgroup of its symmetry group are equal to the number of tilings of another object up to its entire symmetry group, as in the case of the dodecahedron and pyritohedron or in the case of the cube and the \(3\)-trapezohedron.

\subsection{Technique and definitions}
In order to count the tilings up to symmetry, we must specify the polyhedron, its isometry group (or a subgroup thereof), and combinatorial data describing the allowable \emph{tile designs}, which we describe in terms of their symmetries.

Informally, tile designs are the patterns that can be placed on each face, along with their rotations and reflections. In \Cref{fig:tilings}, we show two tilings of the tetrahedron, where each face is decorated with a rotation/reflection of the triangular design
\begin{tikzpicture}[scale=0.6,baseline=3]
  \fill[blue!50!white] (0,0)--(1/2,{sqrt(3)/2})--(1,0)--(1/2,{sqrt(3)/6})--cycle;
  \draw[thick] (0,0)--(1/2,{sqrt(3)/2})--(1,0)--cycle;
\end{tikzpicture}, which generates the set of tile designs resulting from its rotations and reflections:
\(
  \left\{
    \begin{tikzpicture}[scale=0.6,baseline=3]
      \fill[blue!50!white] (0,0)--(1/2,{sqrt(3)/2})--(1,0)--(1/2,{sqrt(3)/6})--cycle;
      \draw[thick] (0,0)--(1/2,{sqrt(3)/2})--(1,0)--cycle;
  \end{tikzpicture},
  \begin{tikzpicture}[scale=0.6,baseline=3]
    \fill[blue!50!white] (0,0)--(1/2,{sqrt(3)/2})--(1/2,{sqrt(3)/6})--(1,0)--cycle;
    \draw[thick] (0,0)--(1/2,{sqrt(3)/2})--(1,0)--cycle;
  \end{tikzpicture},
  \begin{tikzpicture}[scale=0.6,baseline=3]
    \fill[blue!50!white] (0,0)--(1/2,{sqrt(3)/6})--(1/2,{sqrt(3)/2})--(1,0)--cycle;
    \draw[thick] (0,0)--(1/2,{sqrt(3)/2})--(1,0)--cycle;
\end{tikzpicture}
  \right\}\!.
\)

For the purposes of counting face tilings of the polyhedron, the only information we need is how many of these tiles are fixed under each face symmetry in \(g \in D_3 = \langle r, f \mid r^3 = f^2 = (rf)^2 = e \rangle\), which we denote \(t_g\). All of these tiles are fixed under the identity \(e \in D_3\), one of them is fixed under horizontal reflection \(f \in D_3\), and none of them are fixed under \(r \in D_3\). Thus \(\tid{} = 3\), \(\tf{} = 1\), and \(\tr{} = 0\). In \Cref{fig:tilings}, we give an example of two tilings that are the same up to a rotation of the tetrahedron. In \Cref{sec:tetrahedral}, we enumerate the six distinct tilings of the tetrahedron that can arise from this set of tile designs.

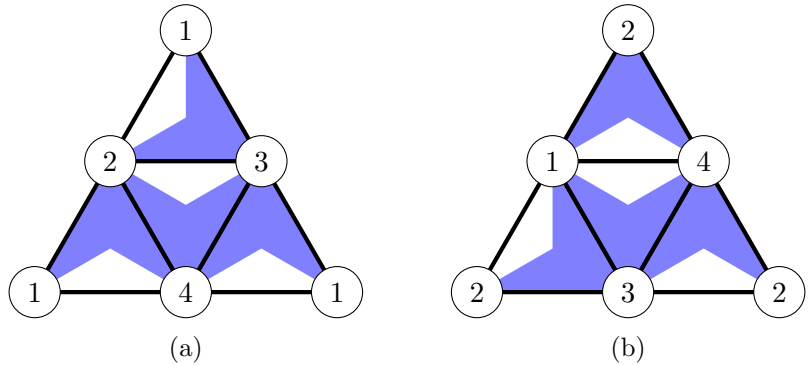
\begin{figure}[h!tbp]
    \centering
    \begin{minipage}[b]{0.3\textwidth}
        \centering
        \begin{tikzpicture}
            \upTri{0}{0}{0}
            \upTri{1}{1}{240}
            \upTri{2}{0}{0}
            \upTri{1}{1/3}{180}
            \node[fill=white, circle, draw=black] at (0,0) {1};
            \node[fill=white, circle, draw=black] at (4,0) {1};
            \node[fill=white, circle, draw=black] at (2,{2*sqrt(3)}) {1};
            \node[fill=white, circle, draw=black] at (1,{1*sqrt(3)}) {2};
            \node[fill=white, circle, draw=black] at (2,0) {4};
            \node[fill=white, circle, draw=black] at (3,{1*sqrt(3)}) {3};
        \end{tikzpicture}
        \subcaption{}\label{fig:tiling_1}
    \end{minipage}
    \qquad
    \begin{minipage}[b]{0.3\textwidth}
        \centering
        \begin{tikzpicture}
            \upTri{0}{0}{240}
            \upTri{1}{1}{0}
            \upTri{2}{0}{0}
            \upTri{1}{1/3}{180}
            \node[fill=white, circle, draw=black] at (0,0) {2};
            \node[fill=white, circle, draw=black] at (4,0) {2};
            \node[fill=white, circle, draw=black] at (2,{2*sqrt(3)}) {2};
            \node[fill=white, circle, draw=black] at (1,{1*sqrt(3)}) {1};
            \node[fill=white, circle, draw=black] at (3,{1*sqrt(3)}) {4};
            \node[fill=white, circle, draw=black] at (2,0) {3};
        \end{tikzpicture}
        \subcaption{}\label{fig:tiling_2}
    \end{minipage}
    \caption{Two tetrahedral nets corresponding to the same face-tiling.}
    \label{fig:tilings}
\end{figure}

We begin by formally describing how the isometries of the polyhedron can induce isometries of the faces.
\begin{definition}
  Let \(\mathcal A\) be a (subgroup of) the isometry group of a polyhedron with the natural group action on the faces of the polyhedron. If \(F\) is a face and \(\mathcal{A}_F\) is the stabilizer subgroup of \(\mathcal{A}\) with respect to that face, we say that the \textbf{induced symmetry group} \(G\) on the face \(F\) is the group of isometries of the face that can be realized by the action of \(\mathcal{A}_F\).
  We say that the \textbf{induced symmetry} on \(F\) by \(A \in \mathcal{A}_F\) is the isometry of \(F\) that results from the group action of \(A\) on the vertices of \(F\).
\end{definition}
For example, the isometries of a truncated icosahedron (the classic soccer ball shape) can induce a \(120^\circ\) rotation about the center of a hexagonal face, but not a \(60^\circ\) rotation. The induced symmetry group on the hexagonal faces by the polyhedral symmetry group is isomorphic to the dihedral group of the triangle, \(D_3\).

\begin{definition}
  Given a polyhedron with a set of faces \(\mathcal F\) and a symmetry group \(\mathcal A\), a \textbf{set of tile designs} is a set \(\mathcal T\) together with a group action of \(\mathcal A\) on \(\mathcal T \times \mathcal F\).
  A \textbf{tiling} of a polyhedron is a map \(\mathcal F \to \mathcal T\) which assigns a tile design to every face.
\end{definition}
In this paper, all of the polyhedra are face-transitive, so we can place any tile design on any face. In non-face-transitive contexts such as Archimedean solids, we must also specify a compatibility condition on the mappings. (We see an example of this in \Cref{subsec:Archimedean})

When enumerating the face tilings of polyhedra, we do not need to describe the set of tile designs themselves, but only the symmetries of the tiles under their induced symmetry group.
\begin{definition}
  Let \(\mathcal T\) be a set of tile designs, and let \(G\) be the induced symmetry of a face \(F\). For each \(g \in G\), there are \(t_g\) elements \(T \in \mathcal T\) such that \(g \cdot (T,F) = (T,F)\).
\end{definition}

Since the polyhedra are face transitive, we can conjugate to see that \(t_g\) does not depend on the choice of \(F\).

\subsection{Computation}
We use the orbit-counting theorem (Burnside's lemma) to count the number of tilings up to isometry by enumerating the tilings that are fixed under each isometry of each polyhedron.

\begin{theorem}[Orbit-counting theorem]
  \label{thm:orbitCountingTheorem}
  Let \(\mathcal A\) be a group, and let \(X\) be a finite set with a \(\mathcal A\)-action. Then the cardinality of \(X\) up to the action of \(\mathcal A\) is \[
    |X/\mathcal A| = \frac{1}{|\mathcal A|}\sum_{A \in \mathcal A} \left|X^A\right|,
  \] where \(|X^A|\) is the number of elements of \(X\) that are fixed under \(\mathcal A\).
\end{theorem}

In the case of Platonic and Catalan solids, we describe an explicit embedding of the vertices into \(\mathbb{R}^3\) along with a compatible matrix representation of the underlying symmetry group of the polyhedron. We use this to compute data about both the induced symmetry group on the faces, along with the cycle structure of the permutations of faces resulting from the action of the polyhedral isometry group. In the case of bipyramids and trapezohedra, we describe explicit bijections with tilings of square cylindrical grids.

The face tilings we enumerate fall into four groupings:
in \Cref{sec:tetrahedral}, we count polyhedra with tetrahedral symmetry;
in \Cref{sec:octahedral}, we count polyhedra with octahedral symmetry;
in \Cref{sec:icosahedral}, we count polyhedra with icosahedral symmetry;
and in \Cref{sec:dihedral}, we count bipyramids and trapezohedra, which have three-dimensional dihedral symmetry (\(D_{nh}\) and \(D_{nd}\) respectively).

Because all of the Platonic and Catalan solids have tetrahedral, octahedral, or icosahedral symmetry, we describe matrix representations of the corresponding Coxeter groups. For each polyhedron, we give an explicit set of points \(\vec{v}_i \in \mathbb{R}^3\), which generate the set of vertices of each polygon when acted on by the symmetry group. This matrix representation and the corresponding polyhedral construction are used to determine how each of the polyhedron's symmetries permute the vertices, and in turn how the symmetries permute, rotate, or reflect the faces.

When the cardinality \(n\) of the face orbit is less than the order of the cyclic subgroup \(\langle A \rangle\), the face lands back on top of itself under the action of \(A^n\), transformed by a (non-identity) induced symmetry of the face.

Thus, if a tiling is fixed under \(A\), the tilings of the faces in each orbit are entirely determined by the tile design of any one face in the orbit, and moreover, that tile design must be fixed under the face symmetry induced by \(A^n\).

This data is recorded in a table where each row corresponds to a conjugacy class of the polyhedral isometry group, and the columns give data about the permutations of the faces and the symmetry of the face induced by \(A^n\).

\begin{note}
  So that the reader can reproduce the table, we choose the representative of each conjugacy class to be the one that is lexicographically earliest by value when read by rows as a matrix.
  Also, the vertices and faces of the polyhedron are given canonical names. Each orbit of vertices is generated by  \(v\), \(v'\), \(v''\), and so on. Within each orbit, the vertices are indexed in lexicographic order by value.
  Each face is described as a list of vertices in positively oriented order (i.e., counterclockwise from the outside) where the list is cyclically shifted so that the lexicographically minimal vertex occurs first. The faces are similarly indexed in lexicographic order.
\end{note}

\begin{example}
  Here we explicitly work through an example of the permutations and induced symmetries of the faces of the icosahedron under its isometry group, which has ten conjugacy classes (enumerated in Table \ref{tabl:IcosahedralGroupConjugacyClasses}).
  Of these ten, seven conjugacy classes have the property that for all \(A\) in the conjugacy class, the cyclic subgroup \(\langle A \rangle\) acts freely on the faces.

  For each of the three remaining conjugacy classes, there is some \(n\) such that \(A^n\) fixes a face and thus induces a (non-trivial) symmetry of that face. We record the relevant data for these three conjugacy classes in the following table, which is a copy of Table \ref{tabl:IcosahedronConjugacy}:
  \[
    \begin{tabular}{P{1.7cm}P{1.6cm}lP{3.4cm}P{1.6cm}}
      conjugacy class  & cycle structure & non-maximal faces & vertex permutation generator & induced subgroup
      \\ \hline & & & & \\[-6pt]
      \multirow{2}{*}{\(\mathcal C^I_3\)} &
      \multirow{2}{*}{\((3^6,1^2)\)} &
      \(F_7 = (v_2,v_6,v_8)\) &
      \((v_2\ v_6\ v_8)\) &
      \(\langle r \rangle\)
      \\
      & &
      \(F_{13} = (v_5,v_7,v_{11})\) &
      \((v_5\ v_{11}\ v_7)\) &
      \(\langle r \rangle\)
      \\[3pt] \hline & & & & \\[-9pt]
      \multirow{2}{*}{\(\mathcal C^I_7\)} &
      \multirow{2}{*}{\((6^3,2)\)} &
      \(F_2 = (v_1,v_3,v_2)\) &
      \((v_1\ v_3\ v_2)\) &
      \(\langle r \rangle\)
      \\
      & &
      \(F_{20} = (v_{10},v_{12},v_{11})\) &
      \((v_{10}\ v_{11}\ v_{12})\) &
      \(\langle r \rangle\)
      \\[3pt] \hline & & & & \\[-9pt]
      \multirow{4}{*}{\(\mathcal C^I_8\)} &
      \multirow{4}{*}{\((2^8,1^4)\)} &
      \(F_9 = (v_3, v_5, v_9)\) &
      \((v_3\ v_9)(v_5)\) &
      \(\langle f \rangle\)
      \\
      & &
      \(F_{10} = (v_3, v_9, v_6)\) &
      \((v_3\ v_9)(v_6)\) &
      \(\langle f \rangle\)
      \\
      & &
      \(F_{11} = (v_4, v_8, v_{10})\) &
      \((v_4\ v_{10})(v_8)\) &
      \(\langle f \rangle\)
      \\
      & &
      \(F_{12} = (v_4, v_{10}, v_7)\) &
      \((v_4\ v_{10})(v_7)\) &
      \(\langle f \rangle\)
    \end{tabular}
  \]

  We will work through using the second row of the table to compute the number of tilings fixed by an element \(A \in \mathcal{C}_7^I\), which is (according to \Cref{tabl:IcosahedralGroupConjugacyClasses}) a \(60^\circ\) rotation of one of the triangular faces of the icosahedron followed by a reflection over the equator. Because the cycle structure indicated in the second column is \((6^3, 2)\), this means that there is one orbit of order \(2\), and thus \(A^2\) fixes (and induces a non-identity symmetry of) both faces in the orbit. By convention, we take \(A\) to be the lexicographically minimal matrix representative, and we use a computer to compute that \(A^2\) fixes the faces \(F_2\) and \(F_{20}\), shown in the third column. By tracking the vertices corresponding to these faces, we see \(A^2\) acts by sending \(v_1 \mapsto v_3\), \(v_3 \mapsto v_2\), and \(v_2 \mapsto v_1\), shown in the fourth column. We see that this induced symmetry is \(r \in D_3\), namely a \(120^\circ\) rotation of the equilateral triangular face, which is recorded in the final column.

  This gives us all of the information needed to compute the number of tilings of the faces of the icosahedron that are fixed under \(A\); in particular, there is one orbit of size \(2\) which is determined by a tile design that is fixed under \(120^\circ\) rotation, and the other \(18\) faces come in three orbits of size \(6\). Therefore, the number of tilings fixed by \(A\) is given by \(|X^A| = \tid3 \tr{}\), where we have three free choices of tile design for the three orbits of size \(|A| = 6\), and because the orbit of size \(2\) must be specified with a rotationally symmetric tile design, we have \(t_r\) choices for that orbit.

  By continuing this with all of the conjugacy classes and computing a weighted average with respect to the size of the conjugacy classes, we determine the number of face tilings up to symmetry by the orbit-counting theorem.
\end{example}

\subsection{On-Line Encyclopedia of Integer Sequences}
Ultimately, we contribute twelve new sequences and recover sixteen existing sequences from the On-Line Encyclopedia of Integer Sequences (OEIS) related to \(n\)-colorings of the faces of various face-transitive polyhedra (and equivalently the colorings of the vertices of their polyhedral duals).

The twelve sequences that the authors added to the OEIS are
\oeis{A378473} (tetrakis hexahedron, \(O_h\)),
\oeis{A378474} (disdyakis dodecahedron, \(O_h\)),
\oeis{A378475} (deltoidal icositetrahedron, \(O\); pentagonal icositetrahedron, \(O\); tetrakis hexahedron, \(O\); tetrakis hexahedron, \(T_d\); triakis octahedron, \(O\)),
\oeis{A378476} (triakis icosahedron, \(I_h\); pentakis dodecahedron \(I_h\); deltoidal hexecontahedron, \(I_h\)),
\oeis{A378477} (disdyakis triacontahedron, \(I_h\)),
\oeis{A378478} (pentagonal hexecontahedron, \(I\); triakis icosahedron, \(I\); pentakis dodecahedron \(I\); deltoidal hexecontahedron, \(I\)),
\oeis{A395240} (bipyramids, \(D_{nh}\)),
\oeis{A396858} (pyritohedron, \(T_h\)),
\oeis{A396861} (truncated icosahedron, \(I\)),
\oeis{A396913} (trapezohedra, \(D_{nd}\)),
\oeis{A396986} (disdyakis dodecahedron, \(O\)), and
\oeis{A396987} (triakis tetrahedron, \(T\); pyritohedron, \(T\)).

The sixteen sequences that are recovered from the OEIS are
\oeis{A000332} (tetrahedron, \(T_d\)),
\oeis{A000543} (octahedron, \(O\)),
\oeis{A000545} (dodecahedron, \(I\)),
\oeis{A006008} (tetrahedron, \(T\)),
\oeis{A047780} (cube, \(O\)),
\oeis{A054472} (icosahedron, \(I\)),
\oeis{A060530} (rhombic dodecahedron, \(O\); triakis tetrahedron, \(T_d\)),
\oeis{A128766} (octahedron, \(O_h\)),
\oeis{A198833} (cube, \(O_h\)),
\oeis{A199406} (rhombic dodecahedron, \(O_h\)),
\oeis{A252704} (icosahedron, \(I_h\)),
\oeis{A252705} (dodecahedron, \(I_h\)),
\oeis{A274900} (rhombicuboctahedron, \(O_h\)),
\oeis{A282670} (rhombic triacontahedron, \(I\)),
\oeis{A316093} (truncated octahedron, \(O\)), and
\oeis{A337963} (rhombic triacontahedron, \(I_h\)).

\section{Polyhedra with tetrahedral symmetry}
\label{sec:tetrahedral}
There is one Platonic solid (the tetrahedron) and one Catalan solid (the triakis tetrahedron) with tetrahedral symmetry, both of which are described by the full tetrahedral group \(T_d\). While the Platonic and Catalan solids are the main focus of this paper, there are other face-transitive polyhedra with tetrahedral symmetry, such as the pyritohedron, whose symmetry group is the pyritohedral group \(T_h\). Additionally, the polyhedra with octahedral symmetry in \Cref{sec:octahedral} and with icosahedral symmetry in \Cref{sec:icosahedral} each have a symmetry group with a tetrahedral subgroup.

The full tetrahedral group \(T_d\) of order \(24\) is the Coxeter group with generators \(R_0\), \(R_1\), and \(R_2\) and group presentation \[
  T_d = \langle
    R_0, R_1, R_2 \mid R_0^2 = R_1^2 = R_2^2 = (R_0R_1)^3 = (R_1R_2)^3 = (R_0R_2)^2 = 1
  \rangle.
\]

We use the following matrix representation:
\[
  R_0 = \begin{bmatrix}0&1&0 \\ 1&0&0 \\ 0&0&1 \end{bmatrix}\!\!,
  \ R_1 = \begin{bmatrix} 1&0&0 \\ 0&0&1 \\ 0&1&0 \end{bmatrix}\!\!,
  \text{ and}
  \ R_2 = \begin{bmatrix} 0&-1&0 \\ -1&0&0 \\ 0&0&1 \end{bmatrix}\!\!,
\]
and we note that the rotational tetrahedral group \(T = \{A \in T_d \mid \det(A) = 1\}\) is the order \(12\) subgroup of elements with positive determinant. The conjugacy classes of both \(T\) and \(T_d\) are described in Table \ref{tabl:TetrahedralGroupConjugacyClasses}.

\begin{table}
\[
\begin{tabular}{lllll}
  name & description & representative & size & order \\ \hline
  \(C_1^T\) & Identity & \(\mathrm{I}\) & 1 & 1 \\
  \(C_2^T\) & Rotation of a face by \(120^\circ\) & \(R_0R_1\) & \(8\) & \(3\) \\
  \(C_3^T\) & Rotation of an edge by \(180^\circ\) & \(R_0R_2\) & \(3\) & \(2\) \\
  \hline
  \(C_4^T\) & Reflection across a face & \(R_0\) & \(6\) & \(2\) \\
  \(C_5^T\) & Rotoreflection of an edge by \(90^\circ\) & \(R_0R_1R_2\) & \(6\) & \(4\)
\end{tabular}
\]
\caption{
  The five conjugacy classes of the full tetrahedral group \(T_d\), with a description of how an element in the conjugacy class acts on a tetrahedron. The first three rows are also the conjugacy classes of the rotational tetrahedral group \(T\).
}
\label{tabl:TetrahedralGroupConjugacyClasses}
\end{table}

\subsection{Tetrahedron}
The tetrahedron, illustrated in \Cref{fig:tetrahedron}, is a Platonic solid with \(4\) equilateral triangular faces. Its \(4\) vertices are generated by the orbit of \(v = (1,1,1)\) under its isometry group, the full tetrahedral group \(T_d\) of order \(24\).

\par\medskip\noindent
\begin{minipage}{\textwidth}\captionsetup{type=figure}
  \centering
    \begin{subfigure}[b]{0.35\textwidth}
      \centering
      \includegraphics[scale=0.5]{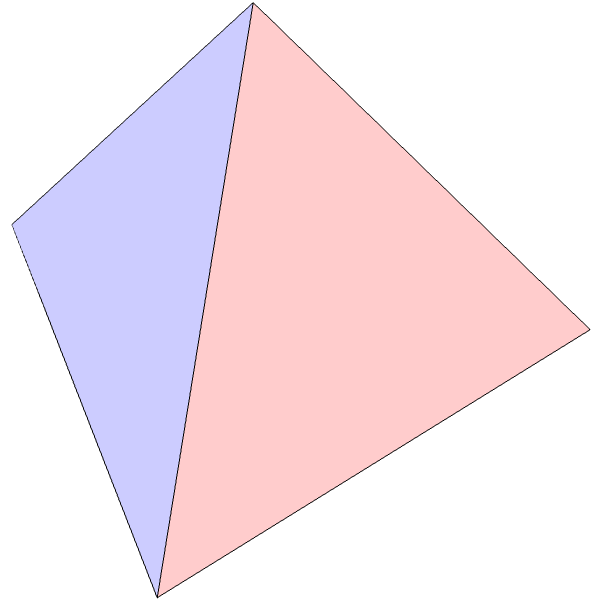}
      \caption{}\label{subfig:tetrahedron}
    \end{subfigure}
    \begin{subfigure}[b]{0.35\textwidth}
      \centering
      \includegraphics[scale=0.5]{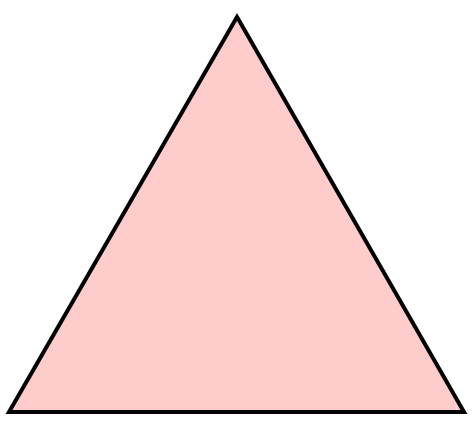}
      \caption{}\label{subfig:tetrahedronFace}
    \end{subfigure}
    \caption[Illustrations of the tetrahedron and its faces.]{
      Illustrations of (a) the tetrahedron with one face highlighted and (b) a face of the tetrahedron.
    }
    \label{fig:tetrahedron}
\end{minipage}
\par\medskip

The induced symmetry group is the dihedral group of the triangle,  \(D_3 = \langle r, f \mid r^3 = f^2 = (rf)^2 = e \rangle\).

\begin{theorem}
  The number of ways of tiling the tetrahedron up to full tetrahedral symmetry \((T_d)\) with \(\tid{}\) tiles, \(\tf{}\) of which are fixed under horizontal reflection and \(\tr{}\) of which are fixed under \(120^\circ\) rotation, is
  \[
    \frac{1}{24}
    \left(
      \tid4 +
      8\tid{}\tr{} +
      3\tid2 +
      6\tid{}\tf2 +
      6\tid{}
    \right).
  \]
\end{theorem}
\begin{proof}
  We proceed by cases, following \Cref{tabl:TetrahedronConjugacy}.
  \begin{description}
    \item[Case 1.] If \(A \in C_{2}^T\), which is the conjugacy class consisting of \(120^\circ\) rotations, then there is \(1\) face that is fixed (rotated) under the action of \(A\). The remaining \(3\) faces appear in the same orbit under \(\langle A \rangle\), the cyclic subgroup generated by \(A\). Thus, there are \(\tid{}\tr{}\) face tilings fixed under \(A \in C_{2}^T\).
    \item[Case 2.] If \(A \in C_{4}^T\), which is the conjugacy class consisting of reflections, then there are \(2\) faces that are fixed (reflected) under the action of \(A\), and the remaining faces appear in the same orbit under \(\langle A \rangle\). Thus, there are \(\tid{}\tf{2}\) face tilings fixed under \(A \in C_{4}^T\).
    \item[Case 3.] For the remaining elements, \(A \in T_d \setminus \left(C_{2}^T \cup C_{4}^T\right)\), which are elements of conjugacy classes that do not appear in the table, the action of \(\langle A \rangle\) on the faces is free. Thus, each remaining group element partitions the \(4\) faces into parts of size \(|A|\), the order of the isometry. Thus, there are \(t_{\id}^{4/|A|}\) tilings fixed under \(A\).
  \end{description}
  We then weight the average by the size of each conjugacy class, as shown in \Cref{tabl:TetrahedralGroupConjugacyClasses}.
\end{proof}
\begin{table}
  \[
  \begin{tabular}{P{1.7cm}P{1.6cm}lP{3.4cm}P{1.6cm}}
    conjugacy class  & cycle structure & non-maximal faces & vertex permutation generator & induced subgroup
    \\ \hline & & & & \\[-6pt]
    \(\mathcal C^T_2\) &
    \((3,1)\) &
    \(F_1 = (v_1, v_2, v_3)\) &
    \((v_1\ v_2\ v_3)\) &
    \(\langle r \rangle\)
    \\[3pt] \hline & & & & \\[-9pt]
    \multirow{2}{*}{\(\mathcal C^T_4\)} &
    \multirow{2}{*}{\((2,1^2)\)} &
    \(F_1 = (v_1, v_2, v_3)\) &
    \((v_1)(v_2\ v_3)\) &
    \(\langle f \rangle\)
    \\
    & &
    \(F_4 = (v_2, v_4, v_3)\) &
    \((v_2\ v_3)(v_4)\) &
    \(\langle f \rangle\)
    \\[3pt] \hline & & & & \\[-9pt]
  \end{tabular}
  \]
  \caption{The cycle structure of the permutations of the vertices of a tetrahedron under each conjugacy class of \(T_d\).}
  \label{tabl:TetrahedronConjugacy}
\end{table}

\begin{corollary}
  The number of ways of tiling the tetrahedron up to rotational tetrahedral symmetry \((T)\) with \(\tid{}\) tiles, \(\tr{}\) of which are fixed under \(120^\circ\) rotation, is
  \[
    \frac{1}{12}
    \left(
      \tid4 +
      8\tid{}\tr{} +
      3\tid2
    \right).
  \]
\end{corollary}

\begin{example}
  If the set of tile designs is \(
    \left\{
      \begin{tikzpicture}[scale=0.7,baseline=5]
        \fill[blue!50!white] (0,0)--(1/2,{sqrt(3)/2})--(1,0)--(1/2,{sqrt(3)/6})--cycle;
        \draw[thick] (0,0)--(1/2,{sqrt(3)/2})--(1,0)--cycle;
    \end{tikzpicture},
    \begin{tikzpicture}[scale=0.7,baseline=5]
      \fill[blue!50!white] (0,0)--(1/2,{sqrt(3)/2})--(1/2,{sqrt(3)/6})--(1,0)--cycle;
      \draw[thick] (0,0)--(1/2,{sqrt(3)/2})--(1,0)--cycle;
    \end{tikzpicture},
    \begin{tikzpicture}[scale=0.7,baseline=5]
      \fill[blue!50!white] (0,0)--(1/2,{sqrt(3)/6})--(1/2,{sqrt(3)/2})--(1,0)--cycle;
      \draw[thick] (0,0)--(1/2,{sqrt(3)/2})--(1,0)--cycle;
  \end{tikzpicture}
    \right\}\!,
  \) and thus \(\tid{} = 3\), \(\tf{} = 1\) and \(\tr{} = 0\), we compute that there are \(
    \frac{1}{24}\left(3^4 + 8(3)(0) + 3(3^2) + 6(3) + 6(3)\right) = 6
  \) tilings of the tetrahedron up to \(T_d\), which are illustrated in Figure \ref{fig:t6}.
\end{example}
\begin{figure}[h!tbp]
    \centering
    \begin{minipage}[b]{0.16\textwidth}
        \centering
        \begin{tikzpicture}[scale=0.6]
            \upTri{1}{1/3}{60}
            \upTri{2}{0}{240}
            \upTri{0}{0}{240}
            \upTri{1}{1}{240}
        \end{tikzpicture}
        \subcaption{}\label{fig:t6_1}
    \end{minipage}
    \hfill
    \begin{minipage}[b]{0.16\textwidth}
        \centering
        \begin{tikzpicture}[scale=0.6]
            \upTri{1}{1/3}{60}
            \upTri{2}{0}{0}
            \upTri{0}{0}{240}
            \upTri{1}{1}{0}
        \end{tikzpicture}
        \subcaption{}\label{fig:t6_2}
    \end{minipage}
    \hfill
    \begin{minipage}[b]{0.16\textwidth}
        \centering
        \begin{tikzpicture}[scale=0.6]
            \upTri{1}{1/3}{60}
            \upTri{2}{0}{240}
            \upTri{0}{0}{120}
            \upTri{1}{1}{0}
        \end{tikzpicture}
        \subcaption{}\label{fig:t6_3}
    \end{minipage}
    \hfill
    \begin{minipage}[b]{0.16\textwidth}
        \centering
        \begin{tikzpicture}[scale=0.6]
            \upTri{1}{1/3}{60}
            \upTri{2}{0}{240}
            \upTri{0}{0}{120}
            \upTri{1}{1}{120}
        \end{tikzpicture}
        \subcaption{}\label{fig:t6_4}
    \end{minipage}
    \hfill
    \begin{minipage}[b]{0.16\textwidth}
        \centering
        \begin{tikzpicture}[scale=0.6]
            \upTri{1}{1/3}{60}
            \upTri{2}{0}{240}
            \upTri{0}{0}{120}
            \upTri{1}{1}{240}
        \end{tikzpicture}
        \subcaption{}\label{fig:t6_5}
    \end{minipage}
    \hfill
    \begin{minipage}[b]{0.16\textwidth}
        \centering
        \begin{tikzpicture}[scale=0.6]
            \upTri{1}{1/3}{60}
            \upTri{2}{0}{0}
            \upTri{0}{0}{120}
            \upTri{1}{1}{0}
        \end{tikzpicture}
        \subcaption{}\label{fig:t6_6}
    \end{minipage}
    \caption{Nets of the six essentially different face-tilings of the tetrahedron up to rotation and reflection by equilateral triangular tiles with reflectional but not rotational symmetry. The orbits under the full tetrahedral group \(T_d\) have size \(3\), \(6\), \(12\), \(12\), \(24\), and \(24\) respectively.}
    \label{fig:t6}
\end{figure}
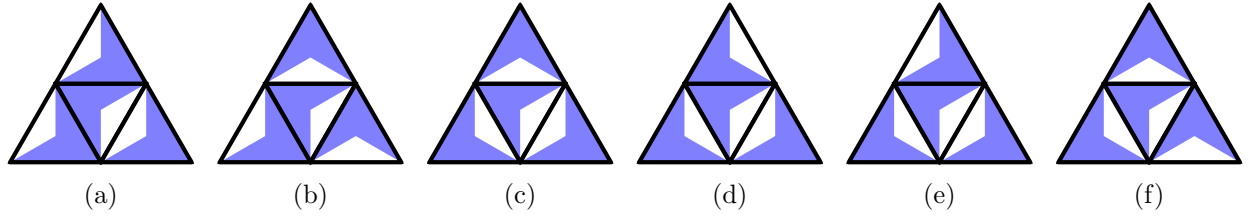
\begin{oeis*}
  The number of \(n\)-colorings of the faces of the tetrahedron up to the \(24\) symmetries of the tetrahedral group \(T_d\) appears in the OEIS as \(\oeis{A000332}(n+3) = \binom{n+3}{4}\): \[
    1, 5, 15, 35, 70, 126, 210, 330, 495, 715, 1001, 1365, 1820, 2380, 3060, \ldots,
  \]
  and the number of \(n\)-colorings up to rotational tetrahedral symmetry \(T\) appears in the OEIS as \oeis{A006008} (\(0\)-indexed): \[
    0, 1, 5, 15, 36, 75, 141, 245, 400, 621, 925, 1331, 1860, 2535, 3381, 4425, \ldots.
  \]
\end{oeis*}

\subsection{Triakis tetrahedron}
The triakis tetrahedron, illustrated in \Cref{fig:triakisTetrahedron}, is a Catalan solid with \(12\) isosceles triangular faces and is the polyhedral dual of the truncated tetrahedron. Its eight vertices are generated by the orbits of the points \(v = (1,1,1)\) and \(v' = (-3/5,-3/5,-3/5)\) under the action of its isometry group, the full tetrahedral group \(T_d\) of order \(24\), and both orbits have size \(4\).

\par\medskip\noindent
\begin{minipage}{\textwidth}\captionsetup{type=figure}
  \centering
    \begin{subfigure}[b]{0.35\textwidth}
      \centering
      \includegraphics[scale=0.5]{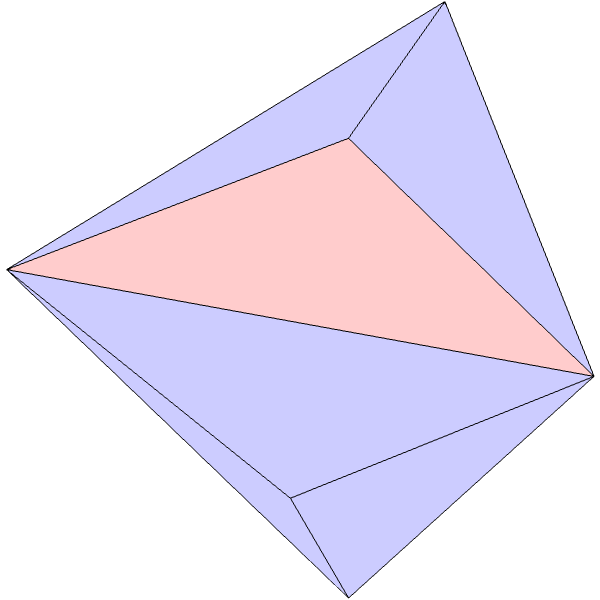}
      \caption{}
    \end{subfigure}
    \begin{subfigure}[b]{0.45\textwidth}
      \centering
      \includegraphics[scale=0.5]{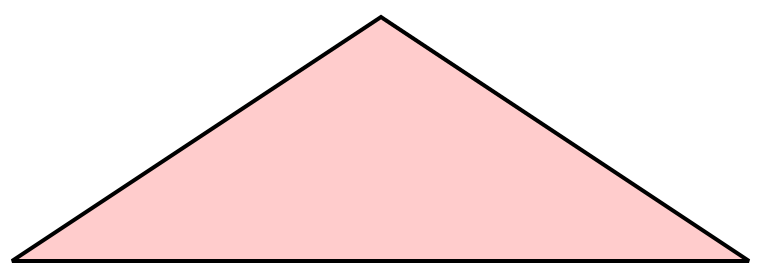}
      \caption{}
    \end{subfigure}
    \caption[Illustrations of the triakis tetrahedron and its faces.]{
      Illustrations of (a) the triakis tetrahedron with one face highlighted and (b) a face of the triakis tetrahedron.
    }
    \label{fig:triakisTetrahedron}
\end{minipage}
\par\medskip

The induced symmetry group is the order-\(2\) reflection group of the isosceles triangle \(\langle f \mid f^2 = e \rangle\).

\begin{theorem}
  \label{thm:TriakisTetrahedron}
  The number of ways of tiling the triakis tetrahedron up to full tetrahedral symmetry \((T_d)\) with \(\tid{}\) tiles, \(\tf{}\) of which are fixed under horizontal reflection, is given by the expression
  \begin{equation*}
    \frac{1}{24}
    \left(
      \tid{12} +
      8\tid4 +
      3\tid6 +
      6\tid5\tf2 +
      6\tid3
    \right).
  \end{equation*}
\end{theorem}
\begin{proof}
  We proceed by cases, following \Cref{tabl:TriakisTetrahedronConjugacy}.
  \begin{description}
    \item[Case 1.] If \(A \in C_{4}^T\), which is the conjugacy class consisting of reflections, then there are \(2\) faces that are fixed (reflected) under the action of \(A\), and the remaining \(10\) faces are partitioned into \(5\) parts of size \(|A| = 2\). Thus the number of face tilings fixed under \(A \in C_{4}^T\) is \(\tid{5}\tf{2}\).
    \item[Case 2.] For the remaining elements, \(A \in T_d \setminus C_{4}^T\), which are elements of conjugacy classes that do not appear in the table, the action of \(\langle A \rangle\) on the faces is free. Thus each remaining group element partitions the \(12\) faces into parts of size \(|A|\) and the number of face tilings fixed under \(A\) is \(t_{\id}^{12/|A|}\).
  \end{description}
\end{proof}

\begin{table}
  \[
  \begin{tabular}{P{1.7cm}P{1.6cm}lP{3.4cm}P{1.6cm}}
    conjugacy class  & cycle structure & non-maximal faces & vertex permutation generator & induced subgroup
    \\ \hline & & & & \\[-6pt]
    \multirow{2}{*}{\(\mathcal C^T_4\)} &
    \multirow{2}{*}{\((2^5,1^2)\)} &
    \(F_8 = (v'_2,v_4,v'_3)\) &
    \((v'_2\ v'_3)(v_4)\) &
    \(\langle f \rangle\)
    \\
    & &
    \(F_9 = (v'_2,v'_3,v_1)\) &
    \((v'_2\ v'_3)(v_1)\) &
    \(\langle f \rangle\)
  \end{tabular}
  \]
  \caption{The cycle structure of the permutations of the vertices of a triakis tetrahedron under each conjugacy class of \(T_d\).}
  \label{tabl:TriakisTetrahedronConjugacy}
\end{table}

\begin{corollary}
  \label{thm:TriakisTetrahedronRotations}
  The number of ways of tiling the triakis tetrahedron up to rotational tetrahedral symmetry \(T\) with \(\tid{}\) tiles is given by the expression
  \begin{equation*}
    \frac{1}{12}
    \left(
      \tid{12} +
      8\tid4 +
      3\tid6
    \right).
  \end{equation*}
\end{corollary}
\begin{oeis*}
  The number of \(n\)-colorings of the faces of the triakis tetrahedron up to the \(12\) symmetries of the rotational tetrahedral group \(T\) has been added to the OEIS as sequence \oeis{A396987} (\(0\)-indexed): \[
    0, 1, 368, 44523, 1399296, 20349375, 181411056, 1153471613, 5726691328, \ldots,
  \]
  and the number of colorings up to the \(24\) symmetries of the tetrahedral group \(T_d\) appears in the OEIS as sequence \oeis{A060530} (\(0\)-indexed): \[
    0, 1, 218, 22815, 703760, 10194250, 90775566, 576941778, 2863870080, \ldots.
  \]
\end{oeis*}

\subsection{Pyritohedron}
While the rest of the paper considers only Platonic solids, Catalan solids, bipyramids, and trapezohedra, here we provide one example outside of these families in order to motivate counting tilings of polyhedra up to \textit{subgroups} of their symmetry groups.

The pyritohedron, illustrated in \Cref{fig:pyritohedron}, is (a family of polyhedra) topologically equivalent to a regular dodecahedron, but unlike the regular dodecahedron whose symmetry group is given by the full icosahedral group \(I_h\) of order \(120\), the symmetry group of the pyritohedron is the pyritohedral symmetry group \(T_h \le I_h\) of order \(24\).

\par\medskip\noindent
\begin{minipage}{\textwidth}\captionsetup{type=figure}
  \centering
    \begin{subfigure}[b]{0.35\textwidth}
      \centering
      \includegraphics[scale=0.5]{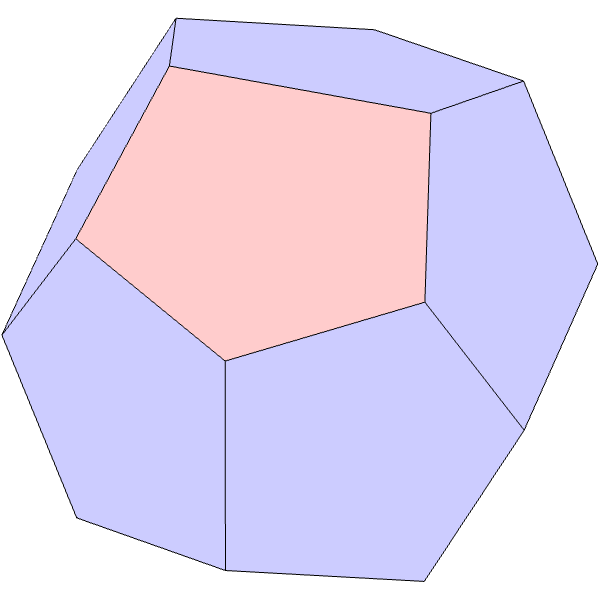}
      \caption{}\label{subfig:pyritohedron}
    \end{subfigure}
    \begin{subfigure}[b]{0.35\textwidth}
      \centering
      \includegraphics[scale=0.5]{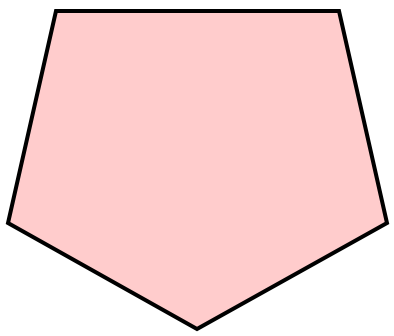}
      \caption{}\label{subfig:pyritohedronFace}
    \end{subfigure}
    \caption[Illustrations of the pyritohedron and its faces.]{
      Illustrations of (a) the pyritohedron with one face highlighted and (b) a face of the pyritohedron.
    }
    \label{fig:pyritohedron}
\end{minipage}
\par\medskip

Pyritohedral symmetry \(T_h\) is the subgroup of integer-valued matrices of the representation of the icosahedral group \(I_h\) described in \Cref{sec:icosahedral}, and can be generated by the matrices \[
  \begin{pmatrix}
    0 & 0 & 1 \\
    1 & 0 & 0 \\
    0 & 1 & 0
  \end{pmatrix} \quad\text{and}\quad
  \begin{pmatrix}
    1 & 0 & 0 \\
    0 & -1 & 0 \\
    0 & 0 & 1
  \end{pmatrix}.
\]

While pyritohedra are a family of polyhedra with one degree of freedom, we describe one example with integer coordinates, which is the pyritohedron that appears in the Weaire--Phelan structure \cite{Lautensack2008}. This pyritohedron has vertices which are generated by the orbits of the points \(v = (4,4,4)\) and \(v' = (0, 3, 6)\) under \(T_h\).

The pyritohedral symmetry group induces a reflectional symmetry \(\langle f \mid f^2 = e \rangle\) on the faces---moreover, this is the same symmetry group induced by \(T_h\) even in the case that the pyritohedron is the regular dodecahedron with regular pentagonal faces.

We present the following proposition without proof.
\begin{proposition}
  \label{thm:Pyritohedron}
  The number of ways of tiling the faces of the pyritohedron up to pyritohedral symmetry \((T_h)\) with \(\tid{}\) tiles, \(\tf{}\) of which are fixed under reflection, is given by the expression
  \begin{equation*}
    \frac{1}{24}
    \left(
      \tid{12} +
      3\tid6 +
      8\tid4 +
      3\tid4\tf4 +
      \tid6 +
      8\tid2
    \right).
  \end{equation*}
\end{proposition}
\begin{corollary}
  \label{thm:PyritohedronRotational}
  The number of ways of tiling the faces of the pyritohedron up to rotational tetrahedral symmetry \((T)\) with \(\tid{}\) tiles is given by the expression
  \begin{equation*}
    \frac{1}{12}
    \left(
      \tid{12} +
      3\tid6 +
      8\tid4
    \right).
  \end{equation*}
\end{corollary}
\begin{oeis*}
  The number of \(n\)-colorings of the faces of the pyritohedron up to the \(12\) symmetries of the rotational tetrahedral group \(T\) has been added to the OEIS as sequence \oeis{A396987} (\(0\)-indexed): \[
    0, 1, 368, 44523, 1399296, 20349375, 181411056, 1153471613, 5726691328, \ldots.
  \]
  and the number of colorings up to the \(24\) symmetries of the pyritohedral group \(T_h\) has been added to the OEIS as sequence \oeis{A396858} (\(0\)-indexed): \[
    0, 1, 220, 23115, 708016, 10224175, 90917436, 577461325, 2865453760, \ldots.
  \]
\end{oeis*}

\section{Polyhedra with octahedral symmetry}
\label{sec:octahedral}
There are two Platonic solids and six Catalan solids with octahedral symmetry. The Platonic solids are the octahedron, cube, and the Catalan solids are the rhombic dodecahedron, triakis octahedron, deltoidal icositetrahedron, tetrakis hexahedron, disdyakis dodecahedron, and pentagonal icositetrahedron.

The full octahedral group \(O_h\) is the order-\(48\) Coxeter group with generators \(R_0\), \(R_1\), and \(R_2\) and group presentation \[
  O_h = \langle
    R_0, R_1, R_2 \mid R_0^2 = R_1^2 = R_2^2 = (R_0R_1)^4 = (R_1R_2)^3 = (R_0R_2)^2 = 1
  \rangle.
\]
We use the following matrix representation:
\[
  R_0 = \begin{bmatrix}1 & 0 & 0 \\ 0 & 1 & 0 \\ 0 & 0 & -1 \end{bmatrix}\!\!,
  \ R_1 = \begin{bmatrix} 1&0&0 \\ 0&0&1 \\ 0&1&0 \end{bmatrix}\!\!,
  \text{ and}
  \ R_2 = \begin{bmatrix} 0&1&0 \\ 1&0&0 \\ 0&0&1 \end{bmatrix}\!\!.
\]

The rotational octahedral group \(O\) is the order-\(24\) subgroup of elements with positive determinant, \(
  O = \{A \in O_h \mid \det(A) = 1\}.
\)

The conjugacy classes of both \(O\) and \(O_h\) are described in Table \ref{tabl:OctahedralGroupConjugacyClasses}.

\begin{table}[ht]
\[
\begin{tabular}{lllll}
  name & description & representative & size & order \\ \hline
  \(C_1^O\) & Identity & \(\mathrm{I}\) & 1 & 1 \\
  \(C_2^O\) & Rotation of a vertex by \(90^\circ\) & \(R_0R_1\) & \(6\) & \(4\) \\
  \(C_3^O\) & Rotation of a vertex by \(180^\circ\) & \(R_0R_1R_0R_1\) & \(3\) & \(2\) \\
  \(C_4^O\) & Rotation of a face by \(120^\circ\) & \(R_1R_2\) & \(8\) & \(3\) \\
  \(C_5^O\) & Rotation of an edge by \(180^\circ\) & \(R_0R_2\) & \(6\) & \(2\) \\
  \hline
  \(C_6^O\) & Central inversion & \(R_0R_1R_0R_1R_2R_1R_0R_1R_2\) & 1 & 2 \\
  \(C_7^O\) & Rotation of a vertex by \(90^\circ\) and reflection & \(R_0R_1R_0R_1R_2\) & \(6\) & \(4\) \\
  \(C_8^O\) & Reflection over the equator & \(R_0\) & \(3\) & \(2\) \\
  \(C_9^O\) & Rotation of a face by \(60^\circ\) and reflection & \(R_0R_1R_2\) & \(8\) & \(6\) \\
  \(C_{10}^O\) & Reflection transposing opposite edges & \(R_1\) & \(6\) & \(2\) \\
\end{tabular}
\]
\caption{The conjugacy classes of the full octahedral group \(O_h\). The first five are also the conjugacy classes of the rotational octahedral group \(O\). The descriptions are given with respect to the group action on an octahedron.}
\label{tabl:OctahedralGroupConjugacyClasses}
\end{table}
\subsection{Octahedron}
The octahedron, illustrated in \Cref{fig:octahedron}, is a Platonic solid with \(8\) equilateral triangular faces. Its \(6\) vertices are generated by the orbit of \((1,0,0)\) under its isometry group, the full octahedral group \(O_h\) of order \(48\).

\par\medskip\noindent
\begin{minipage}{\textwidth}\captionsetup{type=figure}
  \centering
    \begin{subfigure}[b]{0.35\textwidth}
      \centering
      \includegraphics[scale=0.5]{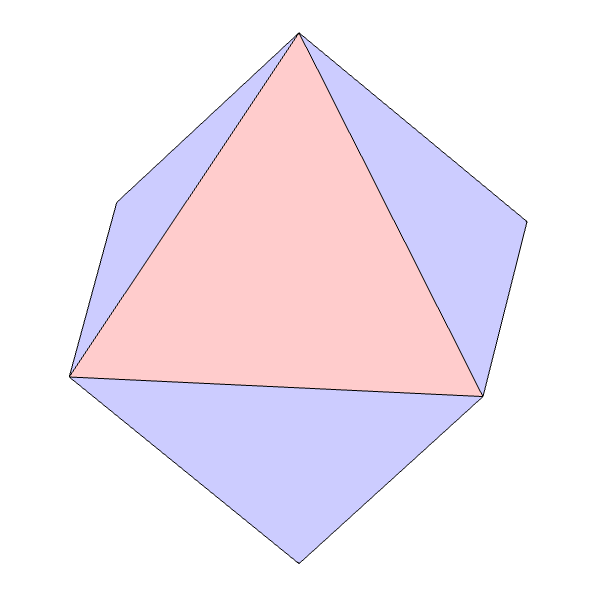}
      \caption{}
    \end{subfigure}
    \begin{subfigure}[b]{0.35\textwidth}
      \centering
      \includegraphics[scale=0.5]{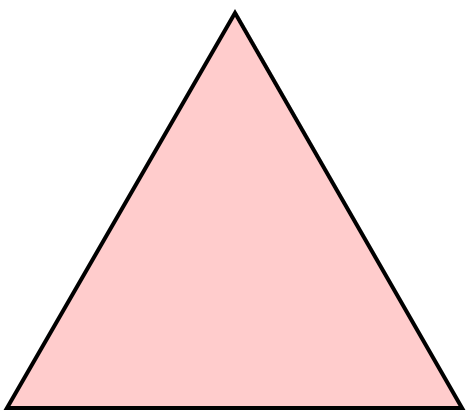}
      \caption{}
    \end{subfigure}
    \caption[Illustrations of the octahedron and its faces.]{
      Illustrations of (a) the octahedron with one face highlighted and (b) a face of the octahedron.
    }
    \label{fig:octahedron}
\end{minipage}
\par\medskip
The induced symmetry group on the faces of the octahedron is the dihedral group of the triangle \(D_3 = \langle r, f \mid r^3 = f^2 = (rf)^2 = e \rangle\).
\begin{theorem}
  The number of ways of tiling the octahedron up to full octahedral symmetry with \(\tid{}\) tiles, \(\tf{}\) of which are fixed under horizontal reflection and \(\tr{}\) of which are fixed under \(120^\circ\) rotation, is
  \begin{align*}
    &\frac{1}{48}
    \left(
      \tid8 +
      6\tid2 +
      3\tid4 +
      8\tid2t_r^2 +
      6\tid4 +
      \tid4 +
      6\tid2 +
      3\tid4 +
      8\tid{} t_r +
      6\tid2 t_f^4
    \right) \\
    &\qquad=\frac{1}{48} \left(
      \tid8 +
      13\tid4 +
      12\tid2 +
      8\tid2t_r^2 +
      8\tid{} t_r +
      6\tid2 t_f^4
    \right).
  \end{align*}
\end{theorem}
\begin{proof}
  We proceed by cases, following \Cref{tabl:OctahedronConjugacy}.
  \begin{description}
    \item[Case 1.] If \(A \in C_{4}^O\), which is the conjugacy class consisting of rotations by \(120^\circ\), then there are \(2\) faces that are fixed (rotated) under the action of \(A\), and the remaining \(6\) faces are partitioned into \(2\) parts of size \(|A| = 3\). Thus, there are \(\tid{2}\tr{2}\) face tilings fixed under \(A \in C_{4}^O\).
    \item[Case 2.] If \(A \in C_{9}^O\), which is the conjugacy class consisting of rotations by \(60^\circ\) followed by reflection, then there are \(2\) faces that are fixed (rotated) under the action of \(A^2\), and thus these two faces appear in the same orbit of size \(2\) under \(\langle A \rangle\). The remaining \(6\) faces are also all in the same orbit of size \(|A| = 6\). Thus, there are \(t_{\id}\tr{2}\) face tilings fixed under \(A \in C_{9}^O\).
    \item[Case 3.] If \(A \in C_{10}^O\), which is a conjugacy class consisting of reflections, then there are \(4\) faces that are fixed (reflected) under the action of \(A\). The remaining \(4\) faces are partitioned into \(2\) parts of size \(|A| = 2\). Thus, there are \(t_{\id}^{2}\tf{4}\) face tilings fixed under \(A \in C_{10}^O\).
    \item[Case 4.] For the remaining elements, \(A \in O_h \setminus \left(C_{4}^O \cup C_{9}^O \cup C_{10}^O\right)\), which are elements of conjugacy classes that do not appear in the table, the action of \(\langle A \rangle\) on the faces is free. Thus, each remaining group element partitions the \(8\) faces into parts of size \(|A|\) and there are \(t_{\id}^{8/|A|}\) face tilings fixed under \(A\).
  \end{description}
\end{proof}
\begin{table}[ht]
  \[
  \begin{tabular}{P{1.7cm}P{1.6cm}lP{3.4cm}P{1.6cm}}
    conjugacy class  & cycle structure & non-maximal faces & vertex permutation generator & induced subgroup
    \\ \hline & & & & \\[-6pt]
    \multirow{2}{*}{\(\mathcal C^O_4\)} & \multirow{2}{*}{\((3^2,1^2)\)} &
      \(F_2 = (v_1, v_3, v_2)\) &
      \((v_1\ v_2\ v_3)\) &
      \(\langle r \rangle\)
      \\
      & &
      \(F_8 = (v_4, v_6, v_5)\) &
      \((v_4\ v_6\ v_5)\) &
      \(\langle r \rangle\)
    \\[3pt] \hline & & & & \\[-9pt]
    \multirow{2}{*}{\(\mathcal C^O_9\)} &
    \multirow{2}{*}{\((6,2)\)} &
      \(F_4 = (v_1, v_5, v_3)\) &
      \((v_1\ v_3\ v_5)\) &
      \(\langle r \rangle\)
      \\
      & &
      \(F_6 = (v_2, v_6, v_4)\) &
      \((v_2\ v_6\ v_4)\) &
      \(\langle r \rangle\)
    \\[3pt] \hline & & & & \\[-9pt]
    \multirow{4}{*}{\(\mathcal C^O_{10}\)} &
    \multirow{4}{*}{\((2^2,1^4)\)} &
      \(F_2 = (v_1, v_3, v_2)\) &
      \((v_1)(v_2\ v_3)\) &
      \(\langle f \rangle\)
      \\
      & &
      \(F_3 = (v_1, v_4, v_5)\) &
      \((v_1)(v_4\ v_5)\) &
      \(\langle f \rangle\)
      \\
      & &
      \(F_5 = (v_2, v_3, v_6)\) &
      \((v_2\ v_3)(v_6)\) &
      \(\langle f \rangle\)
      \\
      & &
      \(F_8 = (v_4, v_6, v_5)\) &
      \((v_4\ v_5)(v_6)\) &
      \(\langle f \rangle\)
  \end{tabular}
  \]
  \caption{The cycle structure of the permutations of the vertices of an octahedron under each conjugacy class of \(O_h\).}
  \label{tabl:OctahedronConjugacy}
\end{table}
\begin{corollary}
  The number of ways of tiling the octahedron up to rotational octahedral symmetry with \(\tid{}\) tiles, \(\tr{}\) of which are fixed under \(120^\circ\) rotation, is
  \[
    \frac{1}{24}
    \left(
      \tid8 +
      6\tid2 +
      3\tid4 +
      8\tid2t_r^2 +
      6\tid4
    \right)
    =\frac{1}{24} \left(
      \tid8 +
      9\tid4 +
      6\tid2 +
      8\tid2t_r^2
    \right).
  \]
\end{corollary}
\begin{oeis*}
The number of \(n\)-colorings of the faces of the octahedron up to rotational octahedral symmetry \(O\) appears in the OEIS as sequence \oeis{A000543}: \[
    1, 23, 333, 2916, 16725, 70911, 241913, 701968, 1798281, 4173775, 8942021, \ldots,
  \]
  and the number of colorings up to the full octahedral group \(O_h\) appears in the OEIS as sequence \oeis{A128766}: \[
    1, 22, 267, 1996, 10375, 41406, 135877, 384112, 966141, 2212750, 4693711,  \ldots.
  \]
\end{oeis*}
\subsection{Cube}
The cube, illustrated in \Cref{fig:cube}, is a Platonic solid with \(6\) square faces. Its \(8\) vertices are generated by the orbit of \(v = (1,1,1)\) under the action of its isometry group, the full octahedral group \(O_h\) of order \(48\).

\par\medskip\noindent
\begin{minipage}{\textwidth}\captionsetup{type=figure}
  \centering
    \begin{subfigure}[b]{0.35\textwidth}
      \centering
      \includegraphics[scale=0.5]{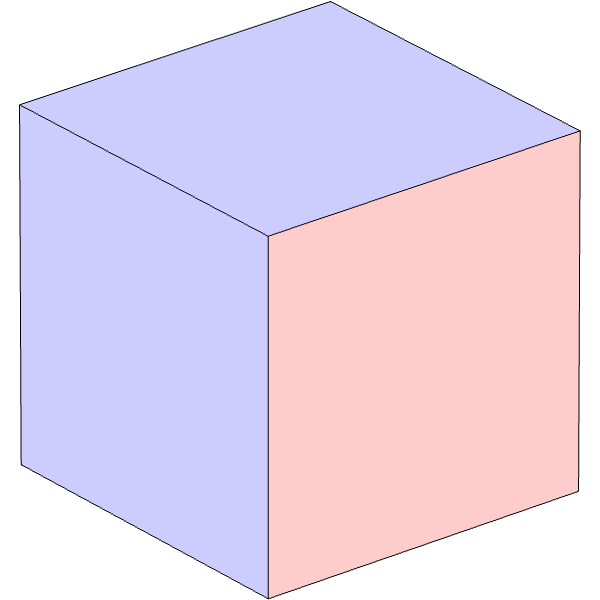}
      \caption{}
    \end{subfigure}
    \begin{subfigure}[b]{0.35\textwidth}
      \centering
      \includegraphics[scale=0.5]{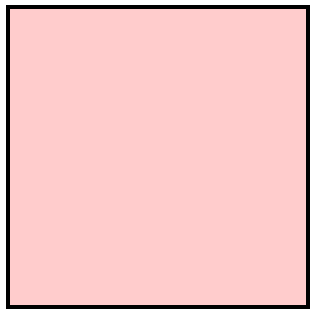}
      \caption{}
    \end{subfigure}
    \caption[Illustrations of the cube and its faces.]{
      Illustrations of (a) the cube with one face highlighted and (b) a face of the cube.
    }
    \label{fig:cube}
\end{minipage}
\par\medskip

The induced symmetry group on the faces of the cube is described by the dihedral group of the square \[
  D_4 = \langle r, f \mid r^4 = f^2 = (rf)^2 = e\rangle,
\] where \(r\) acts by \(90^\circ\) rotation, and \(f\) acts by a non-diagonal reflection across the square.

\begin{table}
  \[
  \begin{tabular}{P{1.7cm}P{1.6cm}lP{3.4cm}P{1.6cm}}
    conjugacy class  & cycle structure & non-maximal faces & vertex permutation generator & induced subgroup
    \\ \hline & & & & \\[-6pt]
    \multirow{2}{*}{\(\mathcal C^O_2\)} &
    \multirow{2}{*}{\((4,1^2)\)} &
      \(F_1 = (v_1, v_2, v_4, v_3)\) &
      \((v_1\ v_3\ v_4\ v_2)\) &
      \(\langle r \rangle\)
      \\
      & &
      \(F_6 = (v_5, v_7, v_8, v_6)\) &
      \((v_5\ v_7\ v_8\ v_6)\) &
      \(\langle r \rangle\)
    \\[3pt] \hline & & & & \\[-9pt]
    \multirow{2}{*}{\(\mathcal C^O_3\)} &
    \multirow{2}{*}{\((2^2,1^2)\)} &
      \(F_1 = (v_1, v_2, v_4, v_3)\) &
      \((v_1\ v_4)(v_2\ v_3)\) &
      \(\langle r^2 \rangle\)
      \\
      & &
      \(F_6 = (v_5, v_7, v_8, v_6)\) &
      \((v_5\ v_8)(v_6\ v_7)\) &
      \(\langle r^2 \rangle\)
    \\[3pt] \hline & & & & \\[-9pt]
    \multirow{2}{*}{\(\mathcal C^O_7\)} &
    \multirow{2}{*}{\((4,2)\)} &
      \(F_2 = (v_1, v_3, v_7, v_5)\) &
      \((v_1\ v_7)(v_3\ v_5)\) &
      \(\langle r^2 \rangle\)
      \\
      & &
      \(F_4 = (v_2, v_6, v_8, v_4)\) &
      \((v_2\ v_8)(v_4\ v_6)\) &
      \(\langle r^2 \rangle\)    \\[3pt] \hline & & & & \\[-9pt]
    \multirow{4}{*}{\(\mathcal C^O_8\)} &
    \multirow{4}{*}{\((2,1^4)\)} &
      \(F_1 = (v_1, v_2, v_4, v_3)\) &
      \((v_1\ v_2)(v_3\ v_4)\) &
      \(\langle f \rangle\)
      \\
      & &
      \(F_3 = (v_1, v_5, v_6, v_2)\) &
      \((v_1\ v_2)(v_5\ v_6)\) &
      \(\langle f \rangle\)
      \\
      & &
      \(F_5 = (v_3, v_4, v_8, v_7)\) &
      \((v_3\ v_4)(v_7\ v_8)\) &
      \(\langle f \rangle\)
      \\
      & &
      \(F_6 = (v_5, v_7, v_8, v_6)\) &
      \((v_5\ v_6)(v_7\ v_8)\) &
      \(\langle f \rangle\)
    \\[3pt] \hline & & & & \\[-9pt]
    \multirow{2}{*}{\(\mathcal C^O_{10}\)} &
    \multirow{2}{*}{\((2^2,1^2)\)} &
      \(F_1 = (v_1, v_2, v_4, v_3)\) &
      \((v_1)(v_2\ v_3)(v_4)\) &
      \(\langle rf \rangle\)
      \\
      & &
      \(F_6 = (v_5, v_7, v_8, v_6)\) &
      \((v_5)(v_6\ v_7)(v_8)\) &
      \(\langle rf \rangle\)
  \end{tabular}
  \]
  \caption{The cycle structure of the permutations of the vertices of a cube under each conjugacy class of \(O_h\).}
  \label{tabl:CubeConjugacy}
\end{table}
\begin{theorem}
  The number of ways of tiling the cube up to full octahedral symmetry \(O_h\) with \(\tid{}\) tiles of which
  \(\tr{}\) are fixed under \(90^\circ\) rotation,
  \(\trr{}\) are fixed under \(180^\circ\) rotation,
  \(\tf{}\) are fixed under horizontal reflection, and
  \(\trf{}\) are fixed under diagonal reflection
  is given by the expression
  \begin{align*}
    &\frac{1}{48}
    \left(
      \tid6 +
      6\tid{}\tr2 +
      3\tid2\trr2 +
      8\tid2 +
      6\tid3 +
      1\tid3 +
      6\tid{}\trr{} +
      3\tid{}\tf4 +
      8\tid{} +
      6\tid2\trf2
    \right) \\
    &\qquad=\frac{1}{48} \left(
      \tid6 +
      7\tid3 +
      8\tid2 +
      8\tid{} +
      6\tid{}\tr2 +
      6\tid{}\trr{} +
      3\tid2\trr2 +
      3\tid{}\tf4 +
      6\tid2\trf2
    \right).
  \end{align*}
  \label{thm:cube}
\end{theorem}
\begin{proof}
  We proceed by cases, following \Cref{tabl:CubeConjugacy}.
  \begin{description}
    \item[Case 1.] If \(A \in C_{2}^O\), which is the conjugacy class consisting of rotations by \(90^\circ\), then there are \(2\) faces that are fixed (rotated by \(90^\circ\)) under the action of \(A\), and the remaining \(4\) faces are partitioned into \(2\) parts of size \(|A| = 2\). Thus the number of face tilings fixed under \(A \in C_{2}^O\) is \(t_{\id}\tr{2}\).
    \item[Case 2.] If \(A \in C_{3}^O\), which is a conjugacy class consisting of rotations by \(180^\circ\), then there are \(2\) faces that are fixed (rotated by \(180^\circ\)) under the action of \(A\), and the remaining \(4\) faces are partitioned into \(2\) parts of size \(|A| = 2\). Thus the number of face tilings fixed under \(A \in C_{3}^O\) is \(t_{\id}^{2}\trr{2}\).
    \item[Case 3.] If \(A \in C_{7}^O\), which is a conjugacy class consisting of rotations by \(90^\circ\) followed by a reflection, then there are \(2\) faces that are fixed (rotated by \(180^\circ\)) under the action of \(A^2\), and thus these faces appear in the same orbit under \(\langle A \rangle\). The remaining \(4\) faces are all in the same orbit size \(|A| = 4\). Thus the number of face tilings fixed under \(A \in C_{7}^O\) is \(t_{\id}\trr{}\).
    \item[Case 4.] If \(A \in C_{8}^O\), which is a conjugacy class consisting of reflections, then there are \(4\) faces that are fixed (reflected horizontally/vertically) under the action of \(A\), and the remaining \(2\) faces appear in the same orbit of size \(|A| = 2\). Thus the number of face tilings fixed under \(A \in C_{8}^O\) is \(t_{\id}\tf{4}\).
    \item[Case 5.] If \(A \in C_{10}^O\), which is a conjugacy class consisting of reflections, then there are \(2\) faces that are fixed (reflected over the diagonal) under the action of \(A\), and the remaining \(4\) faces are partitioned into \(2\) parts of size \(|A| = 2\). Thus the number of face tilings fixed under \(A \in C_{10}^O\) is \(t_{\id}^{2}\trf{2}\).
    \item[Case 6.] For the remaining elements, \(A \in O_h \setminus \left(C_{2}^O \cup C_{3}^O \cup C_{7}^O \cup C_{8}^O \cup C_{10}^O\right)\), which are elements of conjugacy classes that do not appear in the table, the action of the cyclic subgroup \(\langle A \rangle\) on the faces is free. Thus each remaining group element partitions the \(24\) faces into parts of size \(|A|\) and the number of face tilings fixed under \(A\) is \(t_{\id}^{6/|A|}\).
  \end{description}
\end{proof}
\begin{corollary}
  The number of ways of tiling the cube up to rotation octahedral symmetry \(O\) with \(\tid{}\) tiles of which
  \(\tr{}\) are fixed under \(90^\circ\) rotation and
  \(\trr{}\) are fixed under \(180^\circ\) rotation
  is given by the expression
  \begin{align*}
    &\frac{1}{24}
    \left(
      \tid6 +
      6\tid{}\tr2 +
      3\tid2\trr2 +
      8\tid2 +
      6\tid3
    \right).
  \end{align*}
  \label{thm:cubeRotation}
\end{corollary}
\begin{oeis*}
  The number of \(n\)-colorings of the faces of the cube up to the \(24\) symmetries of the rotational octahedral group \(O\) appears in the OEIS as sequence \oeis{A047780} (\(0\)-indexed): \[
    0, 1, 10, 57, 240, 800, 2226, 5390, 11712, 23355, 43450, 76351, 127920, \ldots,
  \]
  and the number of colorings up to the \(48\) symmetries of the octahedral group \(O_h\) appears in the OEIS as sequence \oeis{A198833} (\(1\)-indexed): \[
    1, 10, 56, 220, 680, 1771, 4060, 8436, 16215, 29260, 50116, 82160, \ldots.
  \]
\end{oeis*}
\begin{example}
  The Escher-inspired tile design shown in Figure \ref{fig:escherTile_1} is not fixed by any symmetries of the square, so \(\tid{} = 8\) and \(\tr{}=\trr{}=\tf{}=\trf{}=0\), resulting in \(5548\) distinct tilings of the cube by these tiles (and rotations/reflections of these tiles) up to rotation and reflection of the cube.
\end{example}

\begin{example}
  David Reimann's May 2022 Museum of Mathematics talk ``Play Truchet: Fun with Tiling Patterns and Generalizations'' involved miscellaneous tilings on the faces of cubes. One ``variant pattern with 4 arcs per face'' \cite{Reimann} (shown in Figure \ref{fig:drArt}) can be used to create 159\,775 distinct cubes.
  \setlength{\fboxsep}{0pt}
  \begin{figure}[h!tbp]
    \centering
    \begin{minipage}[b]{0.5\textwidth}
      \begin{minipage}[b]{0.30\textwidth}
        \centering
        \fbox{\includegraphics[width=0.95\linewidth]{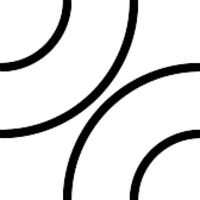}}
        \subcaption{}\label{fig:drArt_1}
      \end{minipage}
      \hfill
      \begin{minipage}[b]{0.30\textwidth}
        \centering
        \fbox{\includegraphics[width=0.95\linewidth]{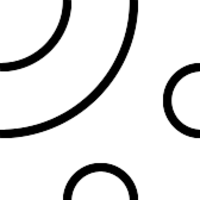}}
        \subcaption{}\label{fig:drArt_2}
      \end{minipage}
      \hfill
      \begin{minipage}[b]{0.30\textwidth}
        \centering
        \fbox{\includegraphics[width=0.95\linewidth]{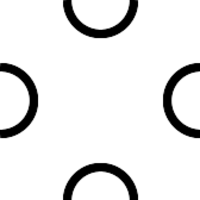}}
        \subcaption{}\label{fig:drArt_3}
      \end{minipage}
      \vspace{1em}
      \begin{minipage}[b]{0.30\textwidth}
        \centering
        \fbox{\includegraphics[width=0.95\linewidth]{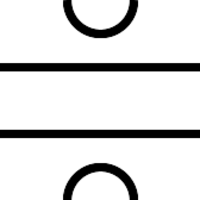}}
        \subcaption{}\label{fig:drArt_4}
      \end{minipage}
      \hfill
      \begin{minipage}[b]{0.30\textwidth}
        \centering
        \fbox{\includegraphics[width=0.95\linewidth]{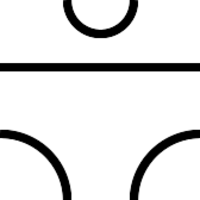}}
        \subcaption{}\label{fig:drArt_5}
      \end{minipage}
      \hfill
      \begin{minipage}[b]{0.30\textwidth}
        \centering
        \fbox{\includegraphics[width=0.95\linewidth]{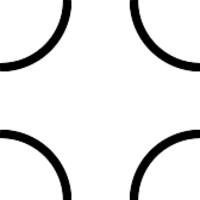}}
        \subcaption{}\label{fig:drArt_6}
      \end{minipage}
    \end{minipage}
    \qquad
    \begin{minipage}[b]{0.3\textwidth}
      \centering
      \includegraphics[width=\linewidth]{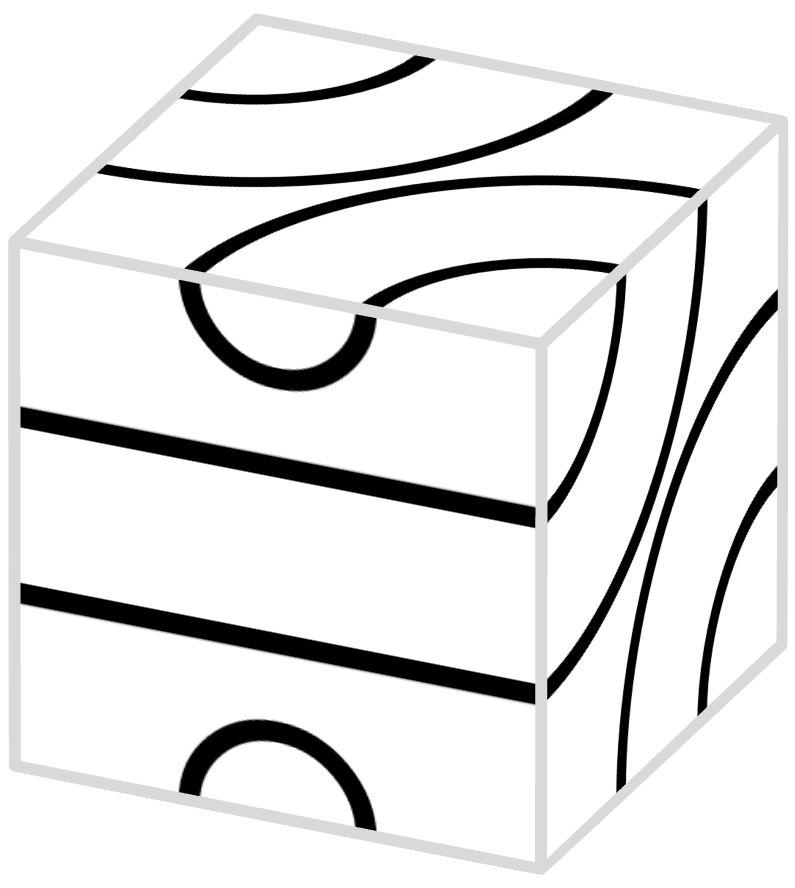}
      \subcaption{}\label{fig:drArt_7}
    \end{minipage}
    \caption{
      Six tile designs (a--f) from David Reimann \cite{Reimann}, and an illustration (g) of one of the 159775 different tilings of the cube up to rotation and reflection that are generated by this design, where \(\tid{}=14, \tr{}=2, \trr{}=6, \tf{}=6, \trf{}=6\).
    }
    \label{fig:drArt}
  \end{figure}
\end{example}

\begin{example}
  In our previous paper \cite[\S 6.6]{KageyKeehn}, we asked about tilings on the faces of a cube where each face has been partitioned into an \(n \times n\) grid.

  \begin{figure}[ht]
    \begin{subfigure}[b]{0.3\textwidth}
      \centering
      \includegraphics[width=\linewidth]{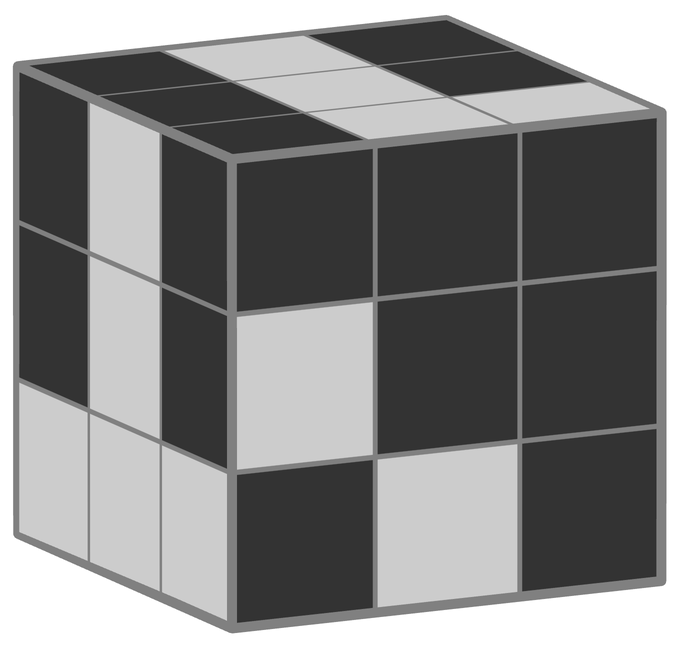}
      \subcaption{}\label{fig:3x3Cube_bw}
    \end{subfigure}
    \hfill
    \begin{subfigure}[b]{0.3\textwidth}
      \centering
      \includegraphics[width=\linewidth]{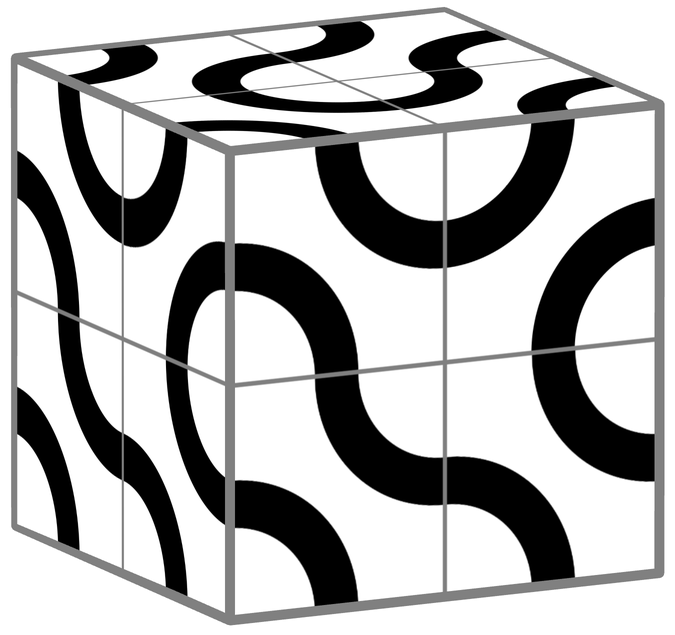}
      \subcaption{}\label{fig:2x2Cube_Truchet}
    \end{subfigure}
    \hfill
    \begin{subfigure}[b]{0.3\textwidth}
      \centering
      \includegraphics[width=\linewidth]{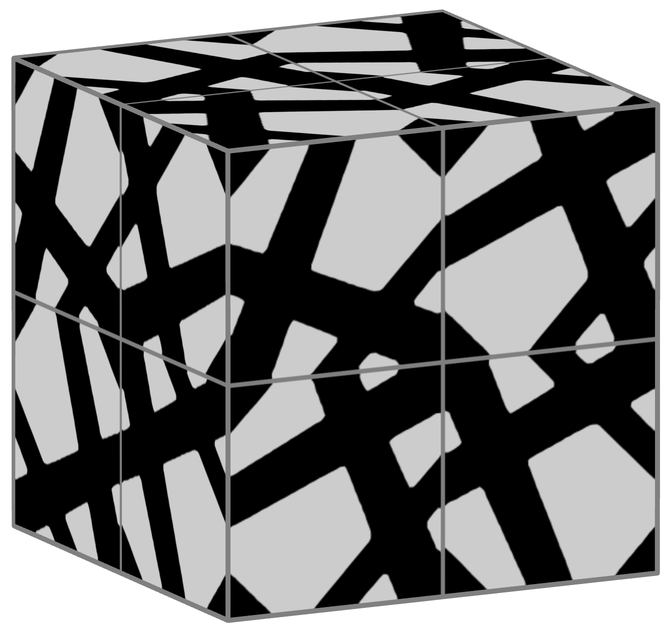}
      \subcaption{}\label{fig:2x2Cube_Escher}
    \end{subfigure}
    \\
    \begin{subfigure}[b]{0.3\textwidth}
      \centering
      \includegraphics[width=0.476\linewidth]{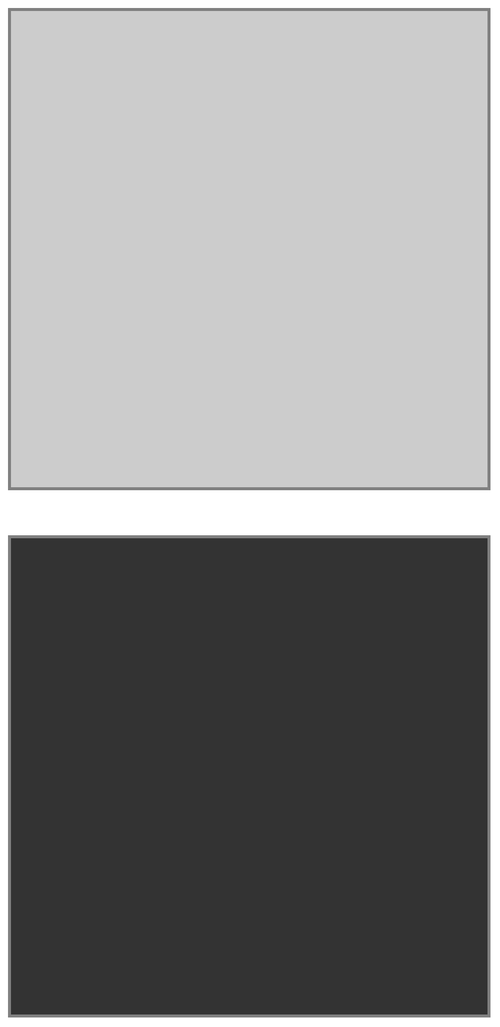}
      \subcaption{}\label{fig:3x3CubeFaces_bw}
    \end{subfigure}
    \hfill
    \begin{subfigure}[b]{0.3\textwidth}
      \centering
      \includegraphics[width=\linewidth]{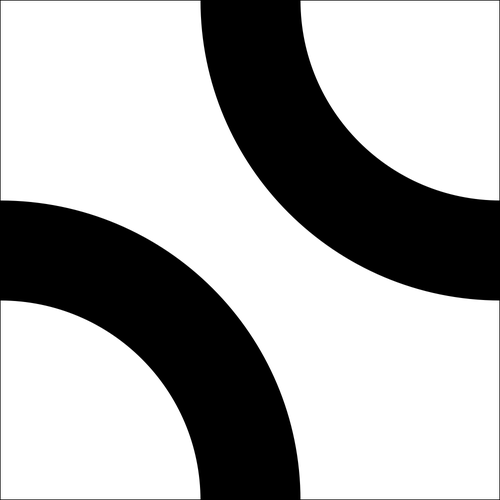}
      \subcaption{}\label{fig:3x3CubeFaces_Truchet}
    \end{subfigure}
    \hfill
    \begin{subfigure}[b]{0.3\textwidth}
      \centering
      \includegraphics[width=\linewidth]{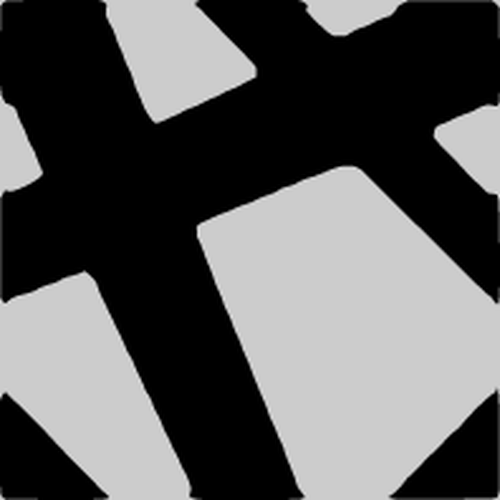}
      \subcaption{}\label{fig:2x2CubeFace_3}
    \end{subfigure}
    \caption{
      Three cubes with grids on each face and their corresponding tile designs:
      (a) a cube with \(3\times3\) grid faces and light-and-dark tiles,
      (b) a cube with \(2\times2\) grid faces and Truchet tiles, and
      (c) a cube with \(2\times2\) grid faces and Escher's tile design.
    }
    \label{fig:gridCubes}
  \end{figure}

  Suppose that we make \(n\times n\)  tiles out of \(\dot t_{\id}\) \(1\times1\) tiles, of which,
  \(\dot t_r\) are fixed under \(90^\circ\) rotation,
  \(\dot t_{r^2}\) are fixed under \(180^\circ\) rotation,
  \(\dot t_{f}\) are fixed under horizontal reflection, and
  \(\dot t_{rf}\) are fixed under diagonal reflection. Then by Theorem 24 in our previous paper \cite{KageyKeehn}, the number of \(n \times n\) tiles that can be formed is \begin{alignat*}{1}
    \hat t_{\id} &= \dot t_{\id}^{n^2}
    \\
    \hat t_{r} &= \begin{cases}
      \dot t_{\id}^{n^2/4} & n \text{ is even} \\
      \dot t_{\id}^{(n^2-1)/4}\dot t_{r} & n \text{ is odd}
    \end{cases}
    \\
    \hat t_{r^2} &= \begin{cases}
      \dot t_{\id}^{n^2/2} & n \text{ is even} \\
      \dot t_{\id}^{(n^2-1)/2}\dot t_{r^2} & n \text{ is odd}
    \end{cases}
    \\
    \hat t_{f} &= \begin{cases}
      \dot t_{\id}^{n^2/2} & n \text{ is even} \\
      \dot t_{\id}^{(n^2-n)/2}\dot t_{f}^n & n \text{ is odd}
    \end{cases}
    \\
    \hat t_{rf} &= \dot t_{\id}^{(n^2-n)/2}\dot t_{rf}^n,
  \end{alignat*}
  which recovers \(\hat t_g = \dot t_g\) when \(n=1\).

  Thus the number of ways to tile a cube with \(n \times n\) grids on each face up to the full octahedral group \(O_h\) is \[
    \begin{cases}
       \begin{alignedat}{4}
        \displaystyle\frac{1}{48}\big(\dot t_{\id}^{6n^2} &+
        6\dot t_{\id}^{3n^2/2} &&+
        3\dot t_{\id}^{3n^2} &&+
        8\dot t_{\id}^{2n^2} &&+
        6\dot t_{\id}^{3n^2} +
        \\
        \dot t_{\id}^{3n^2} &+
        6\dot t_{\id}^{3n^2/2} &&+
        3\dot t_{\id}^{3n^2} &&+
        8\dot t_{\id}^{n^2} &&+
        6\dot t_{\id}^{3n^2-n}\dot t_{rf}^{2n}\big)
      \end{alignedat} & n \text{ is even}
    \\[20pt]
    \begin{alignedat}{4}
      \displaystyle\frac{1}{48}\big(\dot t_{\id}^{6n^2} &+
      6\dot t_{\id}^{(3n^2-1)/2}\dot t_{r}^2 &&+
      3\dot \dot t_{\id}^{3n^2-1}\dot t_{r^2}^2 &&+
      8\dot t_{\id}^{2n^2} &&+
      6\dot t_{\id}^{3n^2} +
      \\
      \dot t_{\id}^{3n^2} &+
      6\dot \dot t_{\id}^{(3n^2-1)/2}\dot t_{r^2} &&+
      3\dot \dot t_{\id}^{3n^2-2n}\dot t_{f}^{4n} &&+
      8\dot t_{\id}^{n^2} &&+
      6\dot t_{\id}^{(3n^2-n)}\dot t_{rf}^{2n}\big)
    \end{alignedat} & n \text{ is odd}
    \end{cases}
  \]
  For the cubes shown in Figure \ref{fig:gridCubes}, we have the following values:
  \begin{itemize}
    \item For the cube with \(3 \times 3\) grids on each face in Figure \ref{fig:3x3Cube_bw},
    \(\dot t_{\id}=\dot t_{r^2}=\dot t_{f}=\dot t_{rf}=2\), resulting in \(375\,300\,676\,436\,736\approx 3.75 \times 10^{14}\) distinct tilings by two colors of tiles up to rotation and reflection of the cube.
    \item For the cube with \(2\times2\) grids on each face in Figure \ref{fig:2x2Cube_Truchet},
    \(\dot t_{\id} = \dot t_{rf} = 2\) and \(\dot t_{f} = \dot t_{r^2} = 0\), resulting in \(352\,744\) distinct tilings by Truchet tiles up to rotation and reflection of the cube.
    \item For the cube with \(2\times2\) grids on each face in Figure \ref{fig:2x2Cube_Escher},
    \(\dot t_{\id} = 8\) and \(\dot t_{f} = \dot t_{r^2} = \dot t_{rf} = 0\), resulting in \(98\,382\,635\,078\,398\,662\,656 \approx 9.8 \times 10^{19}\) distinct tilings by Escher's tile design up to rotation and reflection of the cube.
  \end{itemize}
\end{example}

\subsection{Rhombic dodecahedron}
The rhombic dodecahedron, illustrated in \Cref{fig:rhombicDodecahedron}, is a Catalan solid with \(12\) rhombic faces and is the polyhedral dual of the cuboctahedron. Its \(14\) vertices are generated by the orbits of \(v = (1,0,0)\) and \(v' = \frac{1}{2}(1,1,1)\) under the action of its isometry group, the full octahedral group \(O_h\) of order \(48\), and the orbits have size \(6\) and \(8\) respectively.

\par\medskip\noindent
\begin{minipage}{\textwidth}\captionsetup{type=figure}
  \centering
    \begin{subfigure}[b]{0.35\textwidth}
      \centering
      \includegraphics[scale=0.5]{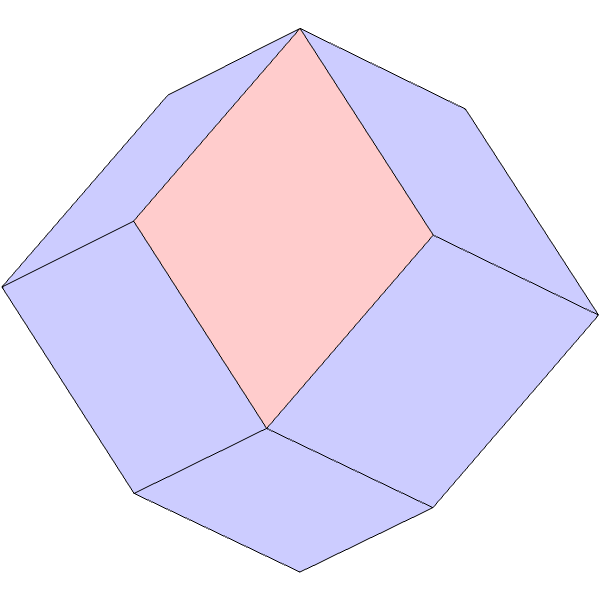}
      \caption{}
    \end{subfigure}
    \begin{subfigure}[b]{0.4\textwidth}
      \centering
      \includegraphics[scale=0.5]{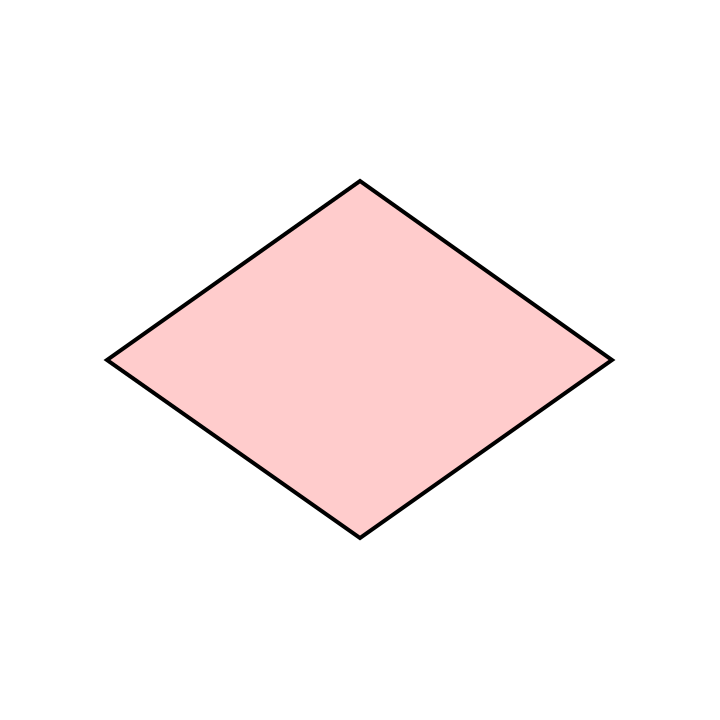}
      \caption{}
    \end{subfigure}
    \caption[Illustrations of the rhombic dodecahedron and its faces.]{
      Illustrations of (a) the rhombic dodecahedron with one face highlighted and (b) a face of the rhombic dodecahedron.
    }
    \label{fig:rhombicDodecahedron}
\end{minipage}
\par\medskip

The induced symmetry group on the faces of the rhombic dodecahedron is the dihedral group of the rhombus \[
  D_2 = \langle r, f \mid r^2 = f^2 = (rf)^2 = e \rangle,
\] where we follow the convention that \(rf\) swaps the \(v\) vertices (on the long diagonal) and \(f\) swaps the \(v'\) vertices (on the short diagonal).

\begin{theorem}
  The number of ways of tiling the faces of the rhombic dodecahedron up to full octahedral symmetry with \(\tid{}\) tiles of which
  \(\tr{}\) are fixed under \(180^\circ\) rotation,
  \(\tf{}\) are fixed under flipping across the short diagonal, and
  \(\trf{}\) are fixed under flipping across the long diagonal
  is given by the expression
  \begin{align*}
    &\frac{1}{48}
    \left(
      \tid{12} +
      6\tid3 +
      3\tid6 +
      8\tid4 +
      6\tid5\tr2 +
      \tid6 +
      6\tid3 +
      3\tid4\tf4 +
      8\tid2 +
      6\tid5\trf2
    \right) \\
    &\qquad=\frac{1}{48} \left(
      \tid{12} +
      4\tid6 +
      8\tid4 +
      12\tid3 +
      8\tid2 +
      6\tid5\tr2 +
      3\tid4\tf4 +
      6\tid5\trf2
    \right)
  \end{align*}
\end{theorem}
\begin{proof}
  We proceed by cases, following \Cref{tabl:RhombicDodecahedronConjugacy}.
  \begin{description}
    \item[Case 1.] If \(A \in C_{5}^O\), which is the conjugacy class consisting of rotations by \(180^\circ\), then there are \(2\) faces that are fixed (rotated) under the action of \(A\), and the remaining \(10\) faces are partitioned into \(5\) parts of size \(|A| = 2\). Thus the number of face tilings fixed under \(A \in C_{5}^O\) is \(t_{\id}^{5}\tr{2}\).
    \item[Case 2.] If \(a \in C_{8}^O\), which is a conjugacy class consisting of reflections, then there are \(4\) faces that are fixed (reflected over the short diagonal) under the action of \(A\), and the remaining \(8\) faces are partitioned into \(4\) parts of size \(|A| = 2\). Thus the number of face tilings fixed under \(A \in C_{8}^O\) is \(t_{\id}^{4}\tf{4}\).
    \item[Case 3.] If \(A \in C_{10}^O\), which is a conjugacy class consisting of reflections, then there are \(2\) faces that are fixed (reflected over the long diagonal) under the action of \(A\), and the remaining \(10\) faces are partitioned into \(5\) parts of size \(|A| = 2\). Thus the number of face tilings fixed under \(A \in C_{10}^O\) is \(t_{\id}^{5}\trf{2}\).
    \item[Case 4.] For the remaining elements, \(A \in O_h \setminus \left(C_{5}^O \cup C_{8}^O \cup C_{10}^O\right)\), which are elements of conjugacy classes that do not appear in the table, the action of the cyclic subgroup \(\langle A \rangle\) on the faces is free. Thus each partitions the \(24\) faces into parts of size \(|A|\) and the number of face tilings fixed under \(A\) is \(t_{\id}^{24/|A|}\).
  \end{description}
\end{proof}
\begin{table}[ht]
  \[
  \begin{tabular}{P{1.7cm}P{1.6cm}lP{3.4cm}P{1.6cm}}
    conjugacy class  & cycle structure & non-maximal faces & vertex permutation generator & induced subgroup
    \\ \hline & & & & \\[-9pt]
    \multirow{2}{*}{\(\mathcal C^O_5\)} & \multirow{2}{*}{\((2^5,1^2)\)} &
      \(F_1 = (v_1,v'_1,v_2,v'_2)\) &
      \((v_1\ v_2)(v'_1\ v'_2)\) &
      \(\langle r \rangle\)
      \\
      & &
      \(F_{12} = (v_5,v'_8,v_6,v'_7)\) &
      \((v_5\ v_6)(v'_7\ v'_8)\) &
      \(\langle r \rangle\)
    \\[3pt] \hline & & & & \\[-9pt]
    \multirow{4}{*}{\(\mathcal C^O_8\)} &
    \multirow{4}{*}{\((2^4,1^4)\)} &
      \(F_1 = (v_1,v'_1,v_2,v'_2)\) &
      \((v_1)(v'_1\ v'_2)(v_2)\) &
      \(\langle f \rangle\)
      \\
      & &
      \(F_4 = (v_1,v'_4,v_5,v'_3)\) &
      \((v_1)(v'_3\ v'_4)(v_5)\) &
      \(\langle f \rangle\)
      \\
      & &
      \(F_9 = (v_2,v'_5,v_6,v'_6)\) &
      \((v_2)(v'_5\ v'_6)(v_6)\) &
      \(\langle f \rangle\)
      \\
      & &
      \(F_{12} = (v_5,v'_8,v_6,v'_7)\) &
      \((v_5)(v'_7\ v'_8)(v_6)\) &
      \(\langle f \rangle\)    \\[3pt] \hline & & & & \\[-9pt]
    \multirow{2}{*}{\(\mathcal C^O_{10}\)} &
    \multirow{2}{*}{\((2^5,1^2)\)} &
      \(F_5 = (v'_1,v_3,v'_5,v_2)\) &
      \((v'_1)(v_2\ v_3)(v'_5)\) &
      \(\langle rf \rangle\)
      \\
      & &
      \(F_8 = (v'_4,v_4,v'_8,v_5)\) &
      \((v'_4)(v_4\ v_5)(v'_8)\) &
      \(\langle rf \rangle\)
  \end{tabular}
  \]
  \caption{The cycle structure of the permutations of the vertices of a rhombic dodecahedron under each conjugacy class of \(O_h\), where \(r\) is a \(180^\circ\) rotation, \(f\) is a reflection that swaps the \(v'\) vertices, and \(rf\) is a reflection that swaps the \(v\) vertices.}
  \label{tabl:RhombicDodecahedronConjugacy}
\end{table}
\begin{corollary}
  The number of ways of tiling the rhombic dodecahedron up to rotational octahedral symmetry \((O)\) with \(\tid{}\) tiles, \(\tr{}\) of which are fixed under \(180^\circ\) rotation, is
  \begin{align*}
    &\frac{1}{24}
    \left(
      \tid{12} +
      6\tid3 +
      3\tid6 +
      8\tid4 +
      6\tid5\tr2
    \right).
  \end{align*}
\end{corollary}
Note that after substituting \(\tr{}=\tf{}\) this is identical to \Cref{thm:TriakisTetrahedron}.

\begin{example}
    Starting in 1974, the artist Sol Lewitt created a series of artworks called ``Incomplete Open Cubes,'' which consist of the \(122\) proper subsets of edges of the cube that are connected and non-planar up to rotations (but not reflections) of the cube \cite{Baume2001}, an example of which is illustrated in \Cref{fig:SolLewittCube}.

    When tiling the rhombic dodecahedron by the tiles in \Cref{fig:SolLewittRD_face}, \(t_{\id} = t_r = 2\), because there are two tile designs which are both fixed under \(180^\circ\) rotations. Thus, there are \[
        \frac{1}{24} \left(2^{12} + 6 \cdot 2^3 + 3 \cdot 2^6 + 8 \cdot 2^3 + 6 \cdot 2^5 \cdot 2^2\right) = 218
    \] non-necessarily-connected subsets of edge of the cube. Thus, there are \(218-122=96\) rotationally distinct subsets of edges of the cube which either include all edges, no edges, are disconnected, or are planar.

    The combinatorics and history of LeWitt's ``Incomplete Open Cubes'' is animated in Paul Dancstep's video ``Exploration \& Epiphany.'' \cite{Dancstep2025}

    By instead using a blank tile design and a tile design with a line connecting the long diagonal of the rhombic faces, the tilings correspond to edge-subsets of the octahedron, some of which correspond to the \(185\) ``Variations of Incomplete Open Octahedra,'' by Mikael Vejdemo-Johansson \cite{VejdemoJohansson2025art}. Because there are also \(218\) edge-subsets, there are \(218 - 185 = 33\) edge subsets of the octahedron that are complete, empty, planar, or disconnected.
\end{example}

\begin{figure}
    \centering
    \begin{minipage}[b]{0.3\textwidth}
      \centering
      \includegraphics[width=\linewidth]{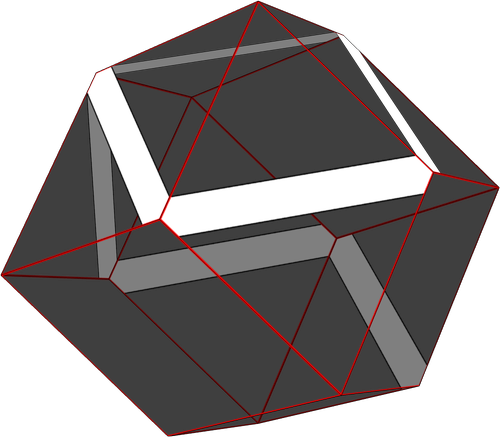}
      \subcaption{}\label{fig:SolLewittRD}
    \end{minipage}
    \hfill
    \begin{minipage}[b]{0.36\textwidth}
      \centering
      \includegraphics[width=0.48\linewidth]{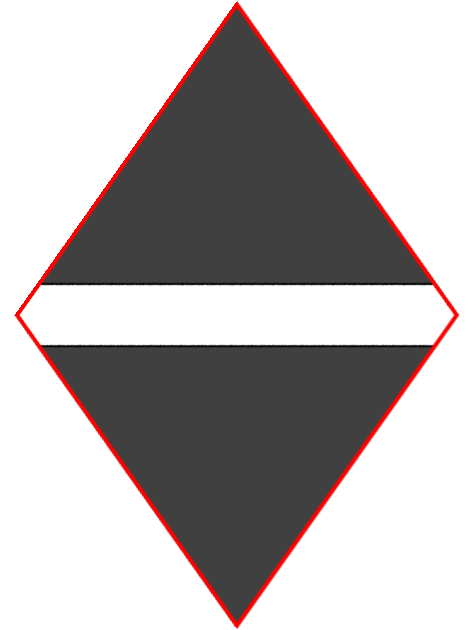}
      \hfill
      \includegraphics[width=0.48\linewidth]{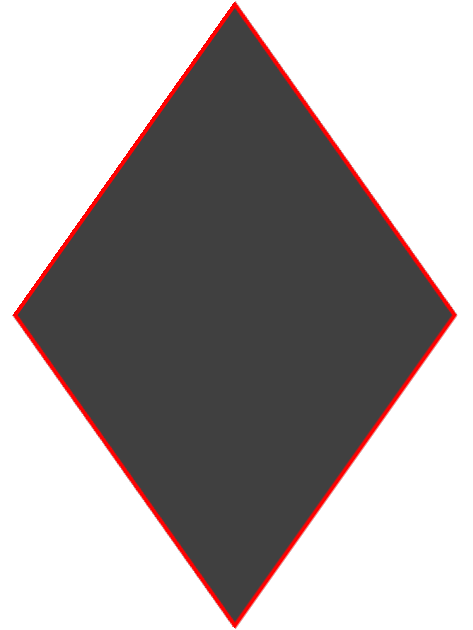}
      \subcaption{}\label{fig:SolLewittRD_face}
    \end{minipage}
    \hfill
    \begin{minipage}[b]{0.25\textwidth}
      \centering
      \includegraphics[width=\linewidth]{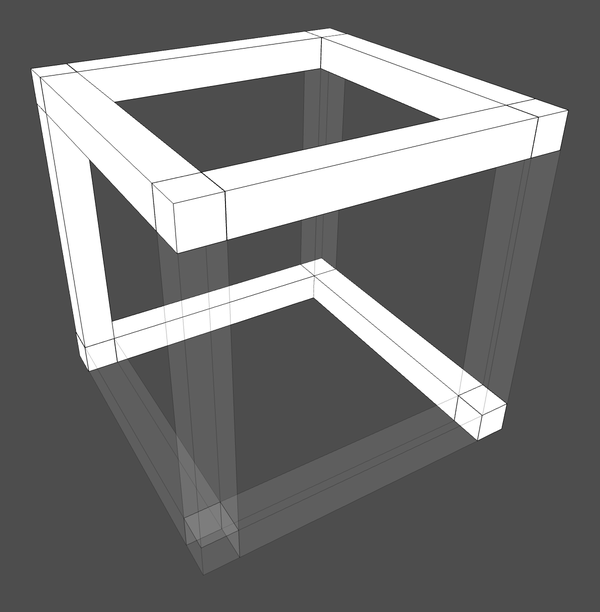}
      \subcaption{}\label{fig:SolLewittCube}
    \end{minipage}
    \caption{
      An example of a tiling of rhombic dodecahedron (a) by semi-transparent tiles in (b), and (c) a corresponding model which is reminiscent of a \(7\)-edge subset of the \(1\)-skeleton of a cube from Sol LeWitt's 1974 ``Incomplete Open Cubes''.
    }
    \label{fig:SolLewitt}
\end{figure}

\begin{oeis*}
  The number of \(n\)-colorings of the faces of the rhombic dodecahedron up to the \(24\) symmetries of the rotational octahedral group \(O_h\) appears in the OEIS as sequence \oeis{A060530} (\(0\)-indexed):
  \[
    0, 1, 218, 22815, 703760, 10194250, 90775566, 576941778, 2863870080, \ldots,
  \]
  and the number of colorings up to the \(48\) symmetries of the octahedral group \(O_h\) appears in the OEIS as sequence \oeis{A199406}: \[
   1, 144, 12111, 358120, 5131650, 45528756, 288936634, 1433251296, 5887880415, \ldots.
  \]
\end{oeis*}

\subsection{Triakis octahedron and deltoidal icositetrahedron}
The triakis octahedron, illustrated in \Cref{subfig:TriakisOctahedronPolyhedron}, is a Catalan solid with \(24\) isosceles triangular faces and is the polyhedral dual of the truncated cube. Its \(14\) vertices are generated by the orbits of \(v=(1,0,0)\) and \(v'=\alpha(1,1,1)\), where \(\alpha = \sqrt{2}-1\) under the action of its isometry group, the full octahedral group \(O_h\) of order \(48\), and the orbits have size \(6\) and \(8\) respectively.

The deltoidal icositetrahedron, illustrated in \Cref{subfig:DeltoidalIcositetrahedronPolyhedron}, is a Catalan solid with \(24\) deltoidal faces and is the polyhedral dual of the rhombicuboctahedron. Its \(26\) vertices are generated by the orbits of \(v=(1,0,0)\), \(v'=\alpha(1,1,1)\), and \(v''=\beta(1,1,0)\) where \(\alpha = \sqrt{2}/2\) and \(\beta = (2 \sqrt{2}+1)/7\), under the action of its isometry group, the full octahedral group \(O_h\) of order \(48\), and the orbits have size \(6\), \(8\), and \(12\) respectively.

\par\medskip\noindent
\begin{minipage}{\textwidth}\captionsetup{type=figure}
  \centering
  \begin{subfigure}[b]{0.35\textwidth}
    \centering
    \includegraphics[scale=0.4]{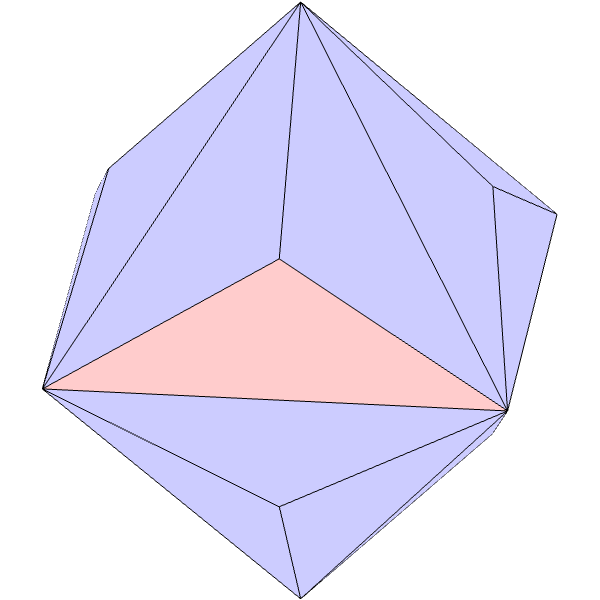}
    \caption{}\label{subfig:TriakisOctahedronPolyhedron}
  \end{subfigure}
  \begin{subfigure}[b]{0.3\textwidth}
    \centering
    \includegraphics[scale=0.4]{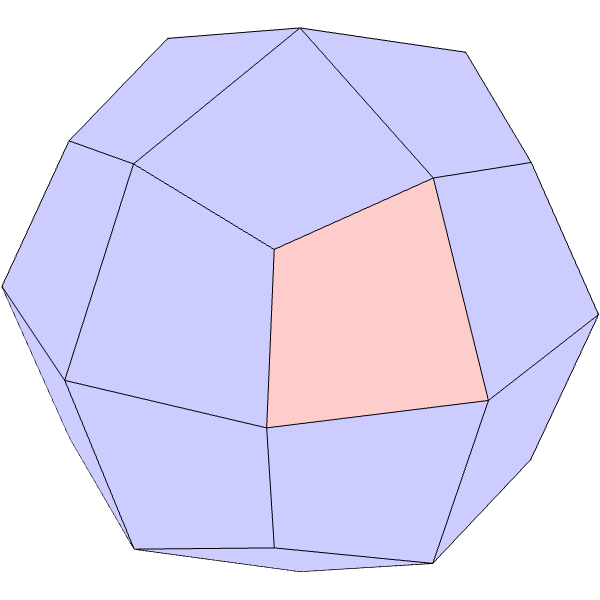}
    \caption{}\label{subfig:DeltoidalIcositetrahedronPolyhedron}
  \end{subfigure}
  \\
  \begin{subfigure}[b]{0.35\textwidth}
    \centering
    \includegraphics[scale=0.4]{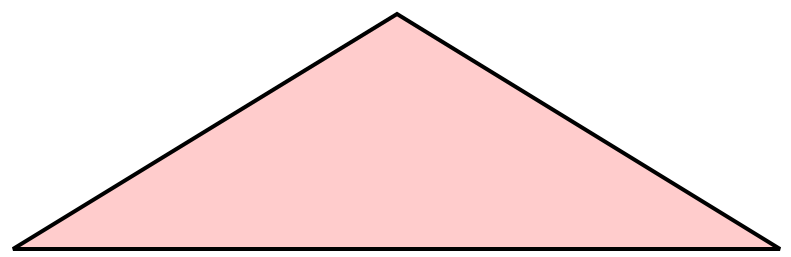}
    \caption{}\label{subfig:TriakisOctahedronFace}
  \end{subfigure}
  \begin{subfigure}[b]{0.3\textwidth}
    \centering
    \includegraphics[scale=0.4]{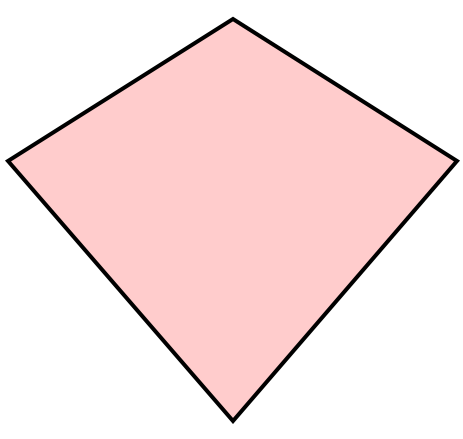}
    \caption{}
  \end{subfigure}
  \caption[Illustrations of the pentagonal hexecontahedron and its faces.]{
    Illustrations of
    (a) the triakis octahedron with a face highlighted,
    (b) the deltoidal icositetrahedron with a face highlighted,
    (c) a face of the triakis octahedron, and
    (d) a face of the deltoidal icositetrahedron.
  }
  \label{fig:TriakisOctahedron}
\end{minipage}
\par\medskip

The induced symmetry group on the faces of both the triakis octahedron and of the deltoidal icositetrahedron is the reflection group \(\langle f \mid f^2 = e \rangle\).

\begin{theorem}
  The number of ways of tiling the faces of the triakis octahedron (equivalently deltoidal icositetrahedron) up to full octahedral symmetry \(O_h\) with \(\tid{}\) tiles of which \(\tf{}\) are fixed under horizontal reflection, is given by the expression
  \begin{align*}
    &\frac{1}{48}
    \left(
      \tid{24} +
      6\tid6 +
      3\tid{12} +
      8\tid8 +
      6\tid{12} +
      \tid{12} +
      6\tid6 +
      3\tid{12} +
      8\tid4 +
      6\tid{10}\tf{4}
    \right) \\
    &\qquad=\frac{1}{48} \left(
      \tid{24} +
      13\tid{12} +
      8\tid8 +
      12\tid6 +
      8\tid4 +
      6\tid{10}\tf{4}
    \right).
  \end{align*}
\end{theorem}
\begin{proof}
  We proceed by cases, following \Cref{tabl:TriakisOctahedronConjugacy,tabl:DeltoidalIcositetrahedronConjugacy}.
  \begin{description}
    \item[Case 1.] If \(A \in C_{10}^O\), which is a conjugacy class consisting of reflections, then there are \(4\) faces that are fixed (reflected) under the action of \(A\), and the remaining \(20\) faces are partitioned into \(10\) parts of size \(|A| = 2\). Thus the number of face tilings fixed under \(A \in C_{10}^O\) is \(t_{\id}^{10}\tr{4}\).
    \item[Case 2.] For the remaining elements, \(A \in O_h \setminus C_{10}^O\), which are elements of conjugacy classes that do not appear in the table, the action of the cyclic subgroup \(\langle A \rangle\) on the faces is free. Thus each partitions the \(24\) faces into parts of size \(|A|\) and the number of face tilings fixed under \(A\) is \(t_{\id}^{24/|A|}\).
  \end{description}
\end{proof}
\begin{table}
  \[
  \begin{tabular}{P{1.7cm}P{1.6cm}lP{3.4cm}P{1.6cm}}
    conjugacy class  & cycle structure & non-maximal faces & vertex permutation generator & induced subgroup
    \\ \hline & & & & \\[-9pt]
    \multirow{4}{*}{\(\mathcal C^O_{10}\)} &
    \multirow{4}{*}{\((2^{10},1^4)\)} &
      \(F_9 = (v'_1,v_3,v_2)\) &
      \((v'_1)(v_2\ v_3)\) &
      \(\langle f \rangle\)
      \\
      & &
      \(F_{12} = (v'_4,v_4,v_5)\) &
      \((v'_4)(v_4\ v_5)\) &
      \(\langle f \rangle\)
      \\
      & &
      \(F_{13} = (v_2,v_3,v'_5)\) &
      \((v_2\ v_3)(v'_5)\) &
      \(\langle f \rangle\)
      \\
      & &
      \(F_{21} = (v_4,v'_8,v_5)\) &
      \((v_4\ v_5)(v'_8)\) &
      \(\langle f \rangle\)
    \end{tabular}
  \]
  \caption{The cycle structure of the permutations of the vertices of a triakis octahedron under each conjugacy class of \(O_h\).}
  \label{tabl:TriakisOctahedronConjugacy}
\end{table}
\begin{table}
  \[
  \begin{tabular}{P{1.7cm}P{1.6cm}lP{3.4cm}P{1.6cm}}
    conjugacy class  & cycle structure & non-maximal faces & vertex permutation generator & induced subgroup
    \\ \hline & & & & \\[-9pt]
    \multirow{4}{*}{\(\mathcal C^O_{10}\)} &
    \multirow{4}{*}{\((2^{10},1^4)\)} &
      \(F_2 = (v_1,v''_2,v'_1,v''_1)\) &
      \((v_1)(v''_1\ v''_2)(v'_1)\) &
      \(\langle f \rangle\)
      \\
      & &
      \(F_3 = (v_1,v''_3,v'_4,v''_4)\) &
      \((v_1)(v''_3\ v''_4)(v'_4)\) &
      \(\langle f \rangle\)
      \\
      & &
      \(F_{21} = (v'_5,v''_{10},v_6,v''_9)\) &
      \((v'_5)(v''_9\ v''_{10})(v_6)\) &
      \(\langle f \rangle\)
      \\
      & &
      \(F_{24} = (v'_8,v''_{11},v_6,v''_{12})\) &
      \((v'_8)(v''_{11}\ v''_{12})(v_6)\) &
      \(\langle f \rangle\)
    \end{tabular}
  \]
  \caption{The cycle structure of the permutations of the vertices of a deltoidal icositetrahedron under each conjugacy class of \(O_h\).
  }
  \label{tabl:DeltoidalIcositetrahedronConjugacy}
\end{table}
\begin{corollary}
  The number of ways of tiling the faces of the triakis octahedron (equivalently deltoidal icositetrahedron) up to rotational octahedral symmetry \(O\) with \(\tid{}\) is given by the expression
  \begin{align*}
    &\frac{1}{24}
    \left(
      \tid{24} +
      6\tid6 +
      3\tid{12} +
      8\tid8 +
      6\tid{12}
    \right).
  \end{align*}
\end{corollary}
\begin{oeis*}
  The number of \(n\)-colorings of the faces of the triakis octahedron (alternatively of the deltoidal icositetrahedron) up to the \(24\) symmetries of the rotational octahedral group \(O\) has been added to the OEIS as sequence \oeis{A378475} (\(0\)-indexed): \[
    0, 1, 700688, 11768099013, 11728130343936, 2483526957328125, 197432556580265616, \ldots,
  \]
  and number of colorings up to the \(48\) symmetries of the full octahedral group \(O_h\) appears in the OEIS as sequence \oeis{A274900} (\(1\)-indexed): \[
    1, 352744, 5884691769, 5864100125056, 1241764261950625, 98716288267057896, \ldots.
  \]
\end{oeis*}

\subsection{Tetrakis hexahedron}
The tetrakis hexahedron, illustrated in \Cref{fig:polyhedron_TetrakisHexahedron}, is a Catalan solid with \(24\) isosceles triangular faces and is the polyhedral dual of the truncated octahedron.
Its \(14\) vertices can be generated by the orbits of \(v=(1,0,0)\) and \(v'=\frac23(1,1,1)\) under the action of its isometry group, the full octahedral group \(O_h\) of order \(48\). The orbits have size \(6\) and \(8\) respectively.

\par\medskip\noindent
\begin{minipage}{\textwidth}\captionsetup{type=figure}
  \centering
    \begin{subfigure}[b]{0.35\textwidth}
      \centering
      \includegraphics[scale=0.5]{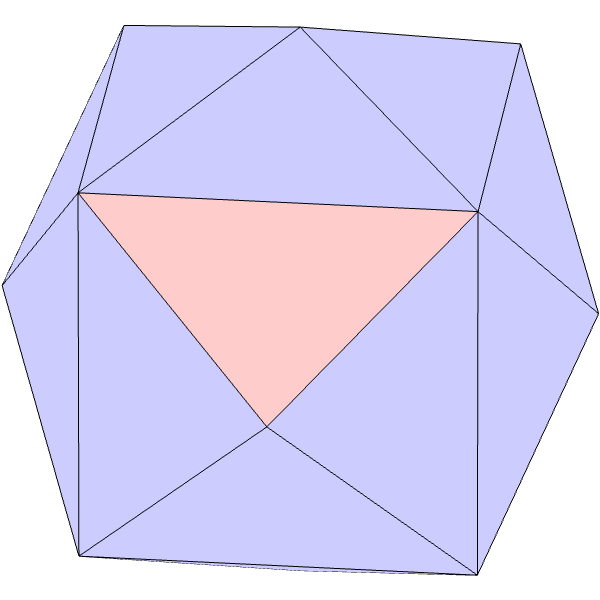}
      \caption{}
    \end{subfigure}
    \begin{subfigure}[b]{0.35\textwidth}
      \centering
      \includegraphics[scale=0.5]{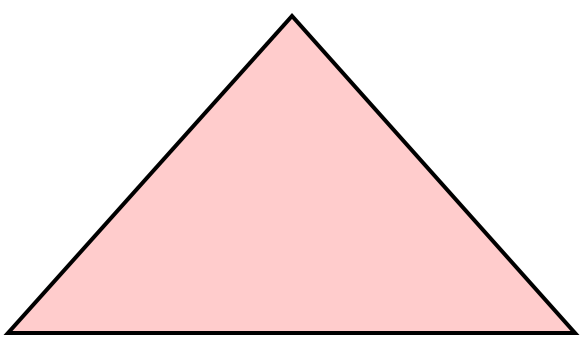}
      \caption{}
    \end{subfigure}
    \caption[Illustrations of the tetrakis hexahedron and its faces.]{
      Illustrations of (a) the tetrakis hexahedron with one face highlighted and (b) a face of the tetrakis hexahedron.
    }
    \label{fig:polyhedron_TetrakisHexahedron}
\end{minipage}
\par\medskip

The induced symmetry group of the faces of the tetrakis hexahedron is the reflection group \(\langle f \mid f^2 = e \rangle\).

\begin{theorem}
  The number of ways of tiling the faces of the tetrakis hexahedron up to full octahedral symmetry with \(\tid{}\) tiles of which
  \(\tf{}\) are fixed under flipping
  is given by the expression
  \begin{align*}
    &\frac{1}{48}
    \left(
      \tid{24} +
      6\tid6 +
      3\tid{12} +
      8\tid8 +
      6\tid{12} +
      \tid{12} +
      6\tid6 +
      3\tid8\tf8 +
      8\tid4 +
      6\tid{12}
    \right) \\
    &\qquad=\frac{1}{48} \left(
      \tid{24} +
      16\tid{12} +
      8\tid8 +
      12\tid6 +
      8\tid4 +
      3\tid8\tf8
    \right).
  \end{align*}
\end{theorem}
\begin{proof}
  We proceed by cases, following \Cref{tabl:TetrakisHexahedronConjugacy}.
  \begin{description}
    \item[Case 1.] If \(A \in C_8^O\), which is the conjugacy class consisting of reflections, then there are \(8\) faces that are fixed (reflected) under the action of \(A\), and the remaining \(16\) faces are partitioned into \(8\) parts of size \(|A| = 2\). Thus the number of face tilings fixed under \(A \in C_8^O\) is \(t_{\id}^{8}\tr{8}\).
    \item[Case 2.] For the remaining elements, \(A \in O_h \setminus C_8^O\), which are elements of conjugacy classes that do not appear in the table, the action of the cyclic subgroup \(\langle A \rangle\) on the faces is free. Thus each partitions the \(24\) faces into parts of size \(|A|\) and the number of face tilings fixed under \(A\) is \(t_{\id}^{24/|A|}\).
  \end{description}
\end{proof}

\begin{table}[ht]
  \[
  \begin{tabular}{P{1.7cm}P{1.6cm}lP{3.4cm}P{1.6cm}}
    conjugacy class  & cycle structure & non-maximal faces & vertex permutation generator & induced subgroup
    \\ \hline & & & & \\[-9pt]
    \multirow{8}{*}{\(\mathcal C^O_8\)} &
    \multirow{8}{*}{\((2^8,1^8)\)} &
      \(F_1 = (v_1, v'_1, v'_2)\) &
      \((v_1)(v'_1\ v'_2)\) &
      \(\langle f \rangle\)
      \\
      & &
      \(F_4 = (v_1, v'_4, v'_3)\) &
      \((v_1)(v'_3\ v'_4)\) &
      \(\langle f \rangle\)
      \\
      & &
      \(F_6 = (v'_1, v_2, v'_2)\) &
      \((v'_1\ v'_2)(v_2)\) &
      \(\langle f \rangle\)
      \\
      & &
      \(F_{12} = (v'_3, v'_4, v_5)\) &
      \((v'_3\ v'_4)(v_5)\) &
      \(\langle f \rangle\)
      \\
      & &
      \(F_{17} = (v_2, v'_5, v'_6)\) &
      \((v_2)(v'_5\ v'_6)\) &
      \(\langle f \rangle\)
      \\
      & &
      \(F_{20} = (v_5, v'_8, v'_7)\) &
      \((v_5)(v'_7\ v'_8)\) &
      \(\langle f \rangle\)
      \\
      & &
      \(F_{22} = (v'_5, v_6, v'_6)\) &
      \((v'_5\ v'_6)(v_6)\) &
      \(\langle f \rangle\)
      \\
      & &
      \(F_{24} = (v'_7, v'_8, v_6)\) &
      \((v'_7\ v'_8)(v_6)\) &
      \(\langle f \rangle\)
  \end{tabular}
  \]
  \caption{The cycle structure of the permutations of the vertices of a tetrakis hexahedron under each conjugacy class of \(O_h\).}
  \label{tabl:TetrakisHexahedronConjugacy}
\end{table}

\begin{corollary}
  The number of ways of tiling the faces of the tetrakis hexahedron
  up to rotational octahedral symmetry with \(\tid{}\) tiles
  is given by the expression
  \begin{equation*}
    \frac{1}{24}
    \left(
      \tid{24} +
      6\tid6 +
      3\tid{12} +
      8\tid8 +
      6\tid{12}
    \right) =
    \frac{1}{24} \left(
      \tid{24} +
      9\tid{12} +
      8\tid8 +
      6\tid6
    \right).
  \end{equation*}
\end{corollary}
Because the faces of the tetrakis hexahedron are the fundamental domain of full tetrahedral symmetry, we include the following theorem, which agrees with the previous.
\begin{theorem}
  The number of ways of tiling the faces of the tetrakis hexahedron up to full tetrahedral symmetry \(T_d\) with \(\tid{}\) tiles is given by the expression
  \begin{equation*}
    \frac{1}{24}
    \left(
      \tid{24} +
      3\tid{12} +
      8\tid8 +
      6\tid{12} +
      6\tid6
    \right) =
    \frac{1}{24} \left(
      \tid{24} +
      9\tid{12} +
      8\tid8 +
      6\tid6
    \right).
  \end{equation*}
\end{theorem}
\begin{proof}
  Because each of the \(24\) faces of the tetrakis hexahedron corresponds to the fundamental domain of the full tetrahedral group \(T_d\), the group action is free. Thus, for each \(A \in T_d\), the orbits of \(\langle A \rangle\) partition the faces into parts of size \(|A|\), and so there are \(\tid{24/|A|}\) tilings fixed by \(A\).
\end{proof}
\begin{oeis*}
  The number of \(n\)-colorings of the faces of the tetrakis hexahedron up to the \(24\) symmetries of the rotational octahedral group \(O\) (equivalently the \(24\) symmetries of the full tetrahedral group \(T_d\)) has been added to the OEIS as sequence \oeis{A378475} (\(0\)-indexed): \[
    0, 1, 700688, 11768099013, 11728130343936, 2483526957328125, \ldots,
  \]
  and the number of colorings up to the \(48\) symmetries of the full octahedral group \(O_h\) has been added to the OEIS as sequence \oeis{A378473} (\(0\)-indexed): \[
    0, 1, 355048, 5886817533, 5864336054656, 1241773051013125, \ldots.
  \]
\end{oeis*}

\subsection{Disdyakis dodecahedron}
The disdyakis dodecahedron, illustrated in \Cref{fig:DisdyakisDodecahedron}, is a Catalan solid with \(48\) scalene triangular faces and is the polyhedral dual of the truncated cuboctahedron. Its \(26\) vertices are generated by the orbits of \(v=(1,0,0)\) and \(v'=\alpha(1,1,1)\), and \(v''=\beta(1,1,0)\) under \(O\), where
\(\alpha = \left(3-\sqrt{2}\right)/3 \approx 0.528595\) and
\(\beta = \left(10-\sqrt{2}\right)/14 \approx 0.61327\), under the action of its isometry group, the full octahedral group \(O_h\) of order \(48\). The orbits have size \(6\), \(8\), and \(12\) respectively.

\par\medskip\noindent
\begin{minipage}{\textwidth}\captionsetup{type=figure}
  \centering
    \begin{subfigure}[b]{0.35\textwidth}
      \centering
      \includegraphics[scale=0.5]{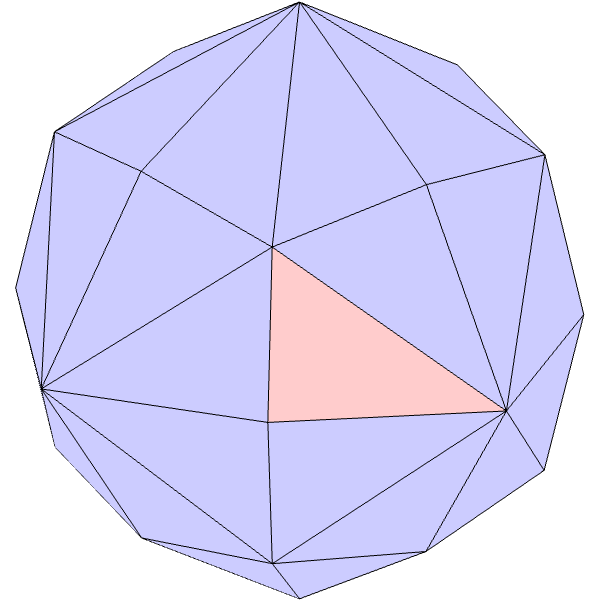}
      \caption{}
    \end{subfigure}
    \begin{subfigure}[b]{0.35\textwidth}
      \centering
      \includegraphics[scale=0.5]{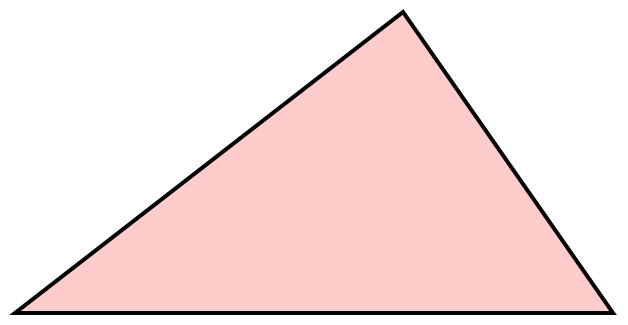}
      \caption{}
    \end{subfigure}
    \caption[Illustrations of the disdyakis dodecahedron and its faces.]{
      Illustrations of (a) the disdyakis dodecahedron with one face highlighted and (b) a face of the disdyakis dodecahedron.
    }
    \label{fig:DisdyakisDodecahedron}
\end{minipage}
\par\medskip

The induced symmetry group of the faces of the disdyakis dodecahedron is the trivial group.

\begin{theorem}
  The number of ways of tiling the faces of the disdyakis dodecahedron up to full octahedral symmetry \(O_h\) with \(\tid{}\) tiles
  is given by the expression
  \begin{align*}
    &\frac{1}{48}
    \left(
      \tid{48} +
      6\tid{12} +
      3\tid{24} +
      8\tid{16}+
      6\tid{24} +
      \tid{24} +
      6\tid{12} +
      3\tid{24} +
      8\tid8 +
      6\tid{24}
    \right) \\
    &\qquad=\frac{1}{48} \left(
      \tid{48} +
      19\tid{24} +
      8\tid{16} +
      12\tid{12} +
      8\tid8
    \right).
  \end{align*}
\end{theorem}
\begin{proof}
  Because each of the \(48\) faces of the disdyakis dodecahedron corresponds to the fundamental domain of the full octahedral group \(O_h\), the group action is free. Thus, for each \(A \in O_h\), the orbits of \(\langle A \rangle\) partition the faces into parts of size \(|A|\), and so there are \(\tid{48/|A|}\) tilings fixed by \(A\).
\end{proof}
\begin{corollary}
  The number of ways of tiling the faces of the disdyakis dodecahedron up to rotational octahedral symmetry \(O\) with \(\tid{}\) tiles
  is given by the expression
  \begin{align*}
    &\frac{1}{24}
    \left(
      \tid{48} +
      6\tid{12} +
      3\tid{24} +
      8\tid{16} +
      6\tid{24}
    \right).
  \end{align*}
\end{corollary}
\begin{oeis*}
  The number of \(n\)-colorings of the faces of the disdyakis dodecahedron up to the \(24\) symmetries of the rotational octahedral group \(O\) has been added to the OEIS as sequence \oeis{A396986} (\(0\)-indexed): \[
    0, 1, 11728130343936, 3323601794975613468921, 3301173438094452954283114496, \ldots,
  \]
  and the number of colorings up to the \(48\) symmetries of the full octahedral group \(O_h\) has been added to the OEIS as sequence \oeis{A378474} (\(0\)-indexed): \[
    0, 1, 5864068667776, 1661800897546646288751, 1650586719047285117763813376, \ldots.
  \]
\end{oeis*}

\subsection{Pentagonal icositetrahedron}
The pentagonal icositetrahedron, illustrated in \Cref{fig:PentagonalIcositetrahedron}, is a Catalan solid with \(24\) (non-regular) pentagonal faces and is the polyhedral dual of the snub cube. Its \(38\) vertices are generated by the orbits of \(v=(\alpha,0,0)\) and \(v'=(1,1,1)\), and \(v''=(1,\beta,\gamma)\), where
\(\alpha \approx 1.83929\) is the tribonacci constant, which is the real root of \(x^3-x^2-x-1\),
\(\beta = \alpha^2 - 2 \approx 1.38298\) is the real root of \(x^3+3 x^2-x-7\), and
\(\gamma = \alpha^{-2} \approx 0.295598\) is the real root of \(x^3+x^2+3 x-1\), under the action of its isometry group, the rotational octahedral group \(O_h\) of order \(24\). The orbits have size \(6\), \(8\), and \(24\) respectively.

\par\medskip\noindent
\begin{minipage}{\textwidth}\captionsetup{type=figure}
  \centering
    \begin{subfigure}[b]{0.35\textwidth}
      \centering
      \includegraphics[scale=0.5]{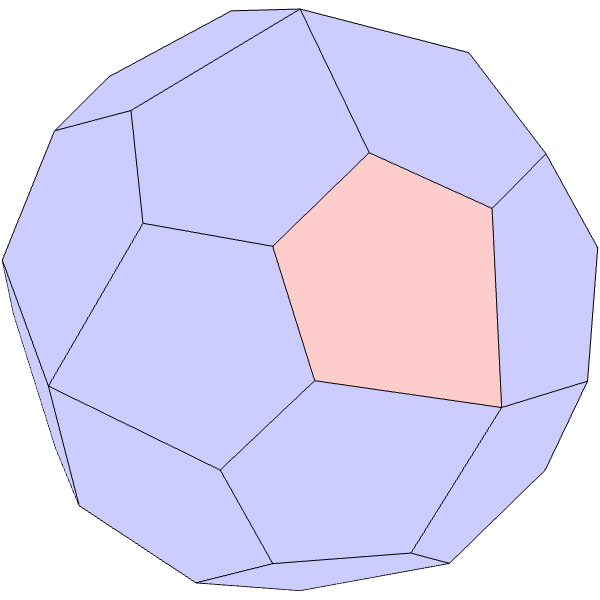}
      \caption{}
    \end{subfigure}
    \begin{subfigure}[b]{0.35\textwidth}
      \centering
      \includegraphics[scale=0.5]{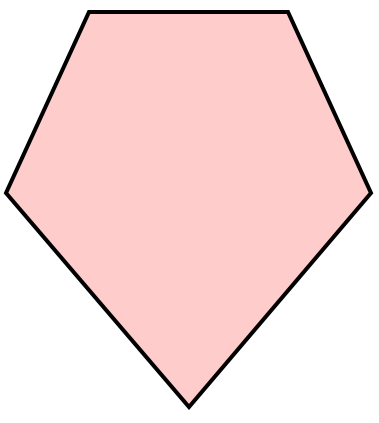}
      \caption{}
    \end{subfigure}
    \caption[Illustrations of the pentagonal icositetrahedron and its faces.]{
      Illustrations of (a) the pentagonal icositetrahedron with one face highlighted and (b) a face of the pentagonal icositetrahedron.
    }
    \label{fig:PentagonalIcositetrahedron}
\end{minipage}
\par\medskip

The induced symmetry group of the faces of the pentagonal icositetrahedron is the trivial group.

\begin{theorem}
  The number of ways of tiling the faces of the pentagonal icositetrahedron up to rotational octahedral symmetry with \(\tid{}\) tiles
  is given by the expression
  \begin{align*}
    &\frac{1}{24}
    \left(
      \tid{24} +
      6\tid{6} +
      3\tid{12} +
      8\tid{8}+
      6\tid{12}
    \right)
    =\frac{1}{24} \left(
      \tid{24} +
      9\tid{12} +
      8\tid{8} +
      6\tid{6}
    \right).
  \end{align*}
\end{theorem}
\begin{proof}
  Because each of the \(24\) faces of the pentagonal icositetrahedron corresponds to the fundamental domain of the rotational octahedral group \(O\), the group action is free. Thus, for each \(A \in O\), the orbits of \(\langle A \rangle\) partition the faces into parts of size \(|A|\), and so there are \(\tid{24/|A|}\) tilings fixed by \(A\).
\end{proof}
\begin{oeis*}
  The number of \(n\)-colorings of the faces of the pentagonal icositetrahedron up to the \(24\) rotational symmetries of the rotational octahedral group \(O\) has been added to the OEIS as sequence \oeis{A378475} (\(0\)-indexed): \[
    0, 1, 700688, 11768099013, 11728130343936, 2483526957328125, \ldots.
  \]
\end{oeis*}

\section{Polyhedra with icosahedral symmetry}
\label{sec:icosahedral}
The face-transitive polyhedra with icosahedral symmetry are the dodecahedron, icosahedron, rhombic triacontahedron, triakis icosahedron, deltoidal hexecontahedron, pentakis dodecahedron, disdyakis triacontahedron, and pentagonal hexecontahedron (rotational symmetry only).

The icosahedral group is the Coxeter group with generators \(R_0\), \(R_1\), and \(R_2\) and group presentation \[
  I_h = \langle
    R_0, R_1, R_2 \mid R_0^2 = R_1^2 = R_2^2 = (R_0R_1)^5 = (R_1R_2)^3 = (R_0R_2)^2 = 1
  \rangle,
\]
where \(\displaystyle \phi = \frac{1 + \sqrt{5}}{2}\) and
\[
  R_0 = \begin{bmatrix}-1 & 0 & 0 \\ 0 & 1 & 0 \\ 0 & 0 & 1 \end{bmatrix}\!\!,
  \ R_1 = \begin{bmatrix}(1-\phi)/2 & -\phi/2 & -1/2 \\ -\phi/2 & 1/2 & (1-\phi)/2 \\ -1/2 & (1-\phi)/2 & \phi/2 \end{bmatrix}\!\!,
  \text{ and}
  \ R_2 = \begin{bmatrix}1 & 0 & 0 \\ 0 & -1 & 0 \\ 0 & 0 & 1 \end{bmatrix}\!\!.
\]

The rotational icosahedral group \(I\) is the order \(60\) subgroup of elements with positive determinant, \[
  I = \{A \in I \mid \det(A) = 1\}.
\]

The conjugacy classes of both \(I\) and \(I_h\) are described in Table \ref{tabl:IcosahedralGroupConjugacyClasses}.

\begin{table}[ht]
\[
\begin{tabular}{lllll}
  name & description & representative & size & order \\ \hline
  \(C_1^I\) & Identity & \(\mathrm{I}\) & 1 & 1 \\
  \(C_2^I\) & Rotation by \(72^\circ\) & \(R_0R_1\) & \(12\) & \(5\) \\
  \(C_{2'}^I\) & Rotation by \(144^\circ\) & \((R_0R_1)^2\) & \(12\) & \(5\) \\
  \(C_3^I\) & Rotation by \(120^\circ\) & \(R_1R_2\) & \(20\) & \(3\) \\
  \(C_4^I\) & Rotation by \(180^\circ\) & \(R_0R_2\) & \(15\) & \(2\) \\
  \hline
  \(C_5^I\) & Central inversion & \((R_0 R_1)^2R_0 R_2 (R_1 R_0)^2 R_2 R_1 R_0 R_1 R_2\) & 1 & 1 \\
  \(C_6^I\) & Rotation by \(36^\circ\) and reflection & \(R_0R_1R_2\) & \(12\) & \(10\) \\
  \(C_{6'}^I\) & Rotation by \(108^\circ\) and reflection & \(R_0R_1R_0R_1R_2R_1R_0R_1R_2\) & \(12\) & \(10\) \\
  \(C_7^I\) & Rotation by \(60^\circ\) and reflection & \((R_0R_1)^2R_2\) & \(20\) & \(6\) \\
  \(C_8^I\) & Reflection & \(R_0\) & \(15\) & \(2\) \\
\end{tabular}
\]
\caption{The conjugacy classes of the full icosahedral group \(I_h\). The first five are also the conjugacy classes of the rotational icosahedral group \(I\).}
\label{tabl:IcosahedralGroupConjugacyClasses}
\end{table}

\subsection{Dodecahedron}
The dodecahedron, illustrated in \Cref{fig:dodecahedron}, is a Platonic solid with \(12\) regular pentagonal faces. Its \(20\) vertices are generated by the orbit of \(v = (1,1,1)\) under the action of its isometry group, the full icosahedral group \(I_h\) of order \(120\).

\par\medskip\noindent
\begin{minipage}{\textwidth}\captionsetup{type=figure}
  \centering
  \begin{subfigure}[b]{0.35\textwidth}
    \centering
    \includegraphics[scale=0.5]{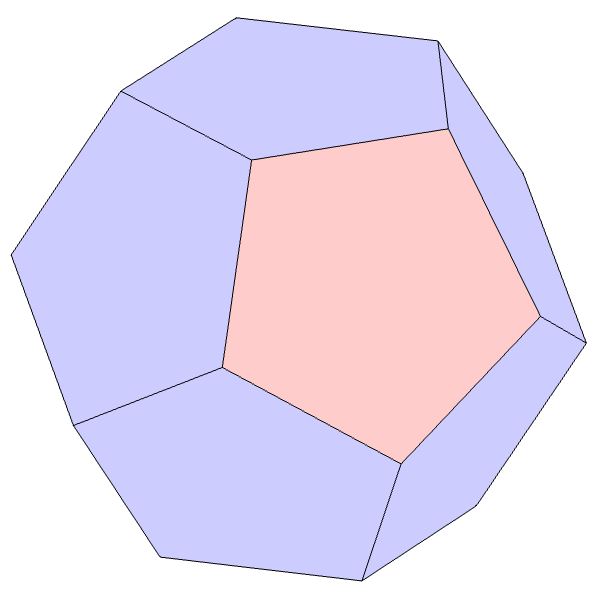}
    \caption{}
  \end{subfigure}
  \begin{subfigure}[b]{0.35\textwidth}
    \centering
    \includegraphics[scale=0.5]{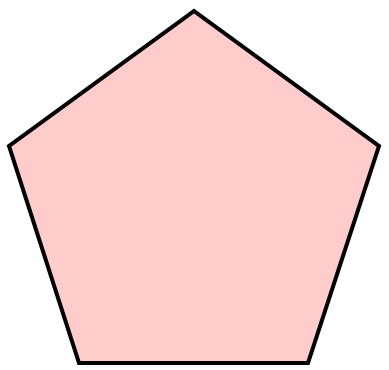}
    \caption{}
  \end{subfigure}
  \caption[Illustrations of the dodecahedron and its faces.]{
    Illustrations of (a) the dodecahedron with one face highlighted and (b) a face of the dodecahedron.
  }
  \label{fig:dodecahedron}
\end{minipage}
\par\medskip

The induced symmetry group on the faces of the dodecahedron is the dihedral group of the pentagon \[
  D_5 = \langle r, f \mid r^5 = f^2 = (rf)^2 = e \rangle,
\] where \(r\) corresponds to \(72^\circ\) rotations, and \(f\) corresponds to reflections.

\begin{theorem}
  \label{thm:dodecahedron}
  The number of ways of tiling the faces of the dodecahedron up to full icosahedral symmetry \(I_h\) with \(\tid{}\) tiles of which
  \(\tr{}\) are fixed under \(72^\circ\) rotation and
  \(\tf{}\) are fixed under flipping horizontally
  is given by the expression
  \begin{align*}
    &\frac{1}{120}
    \left(
      \tid{12} +
      24\tid2\tr2 +
      20\tid4 +
      15\tid6 +
      \tid6 +
      24\tr{}\tid{} +
      20\tid2 +
      15\tid4\tf4
    \right).
  \end{align*}
\end{theorem}
\begin{proof}
  We proceed by cases, following \Cref{tabl:DodecahedronConjugacy}.
  \begin{description}
    \item[Case 1.] If \(A \in C_2^I\), which is the conjugacy class consisting of rotations by \(72^\circ\), then there are \(2\) faces that are fixed (rotated) under the action of \(A\), and the remaining \(10\) faces are partitioned into \(2\) parts of size \(|A| = 5\). Thus the number of face tilings fixed under \(A \in C_4^I\) is \(t_{\id}^{2}\tr{2}\).
    \item[Case 2.] If \(A \in C_6^I\), which is the conjugacy class consisting of rotations by \(36^\circ\) followed by a reflection, then there are \(2\) faces that are fixed (rotated) under the action of \(A^2\), and thus appear of an orbit of size \(2\) under \(\langle A \rangle\). The orbits of the remaining \(10\) faces consist of a single orbit of size \(|A| = 10\). Thus the number of face tilings fixed under \(A \in C_7^I\) is \(t_{\id}\tr{}\).
    \item[Case 3.] If \(A \in C_8^I\), which is the conjugacy class consisting of reflections, then there are \(4\) faces that are fixed (reflected) under the action of \(A\), and the remaining \(8\) faces are partitioned into \(4\) parts of size \(|A| = 2\). Thus the number of face tilings fixed under \(A \in C_8^I\) is \(t_{\id}^{4}\tf{4}\).
    \item[Case 4.] For the remaining elements, \(A \in I \setminus (C_2^I \cup C_6^I \cup C_8^I)\), which are elements of conjugacy classes that do not appear in the table, the action of the cyclic subgroup \(\langle A \rangle\) on the faces is free. Thus each partitions the \(30\) faces into parts of size \(|A|\), and thus the number of face tilings fixed under \(A\) is \(t_{\id}^{30/|A|}\).
  \end{description}
\end{proof}
\begin{table}[ht]
  \[
  \begin{tabular}{P{1.7cm}P{1.6cm}lP{3.4cm}P{1.6cm}}
    conjugacy class  & cycle structure & non-maximal faces & vertex permutation generator & induced subgroup
    \\ \hline & & & & \\[-6pt]
    \multirow{2}{*}{\(\mathcal C^I_2\)} &
    \multirow{2}{*}{\((5^2,1^2)\)} &
      \(F_6 = (v_4,v_{10},v_{16},v_{14},v_8)\) &
      \((v_4\ v_{10}\ v_{16}\ v_{14}\ v_8)\) &
      \(\langle r \rangle\)
      \\
      & &
      \(F_7 = (v_5,v_{11},v_{17},v_{13},v_7)\) &
      \((v_5\ v_7\ v_{13}\ v_{17}\ v_{11})\) &
      \(\langle r \rangle\)
    \\[3pt] \hline & & & & \\[-9pt]
    \multirow{2}{*}{\(\mathcal C^I_6\)} &
    \multirow{2}{*}{\((10, 2)\)}
      & \(F_2 = (v_1,v_3,v_9,v_{10},v_4)\) &
      \((v_1\ v_3\ v_9\ v_{10}\ v_4)\) &
      \(\langle r \rangle\)
      \\
      & &
      \(F_{10} = (v_{11},v_{12},v_{18},v_{20},v_{17})\) &
      \((v_{11}\ v_{17}\ v_{20}\ v_{18}\ v_{12})\) &
      \(\langle r \rangle\)
    \\[3pt] \hline & & & & \\[-9pt]
    \multirow{4}{*}{\(\mathcal C^I_8\)} &
    \multirow{4}{*}{\((2^4,1^4)\)} &
      \(F_5 = (v_3, v_7, v_{13}, v_{15}, v_9)\) &
      \((v_3\ v_{15})(v_7\ v_{13})(v_9)\) &
      \(\langle f \rangle\)
      \\
      & &
      \(F_6 = (v_4, v_{10}, v_{16}, v_{14}, v_8)\) &
      \((v_4\ v_{16})(v_8\ v_{14})(v_{10})\) &
      \(\langle f \rangle\)
      \\
      & &
      \(F_7 = (v_5, v_{11}, v_{17}, v_{13}, v_7)\) &
      \((v_5\ v_{17})(v_7\ v_{13})(v_{11})\) &
      \(\langle f \rangle\)
      \\
      & &
      \(F_8 = (v_6, v_8, v_{14}, v_{18}, v_{12})\) &
      \((v_6\ v_{18})(v_8\ v_{14})(v_{12})\) &
      \(\langle f \rangle\)
  \end{tabular}
  \]
  \caption{The cycle structure of the permutations of the vertices of a dodecahedron under each conjugacy class of \(I_h\).}
  \label{tabl:DodecahedronConjugacy}
\end{table}

\begin{corollary}
  \label{thm:dodecahedronRotation}
  The number of ways of tiling the faces of the dodecahedron up to rotational icosahedral symmetry \(I\) with \(\tid{}\) tiles of which
  \(\tr{}\) are fixed under \(72^\circ\) rotation
  is given by the expression
  \begin{align*}
    &\frac{1}{60}
    \left(
      \tid{12} +
      24\tid2\tr2 +
      20\tid4 +
      15\tid6
    \right).
  \end{align*}
\end{corollary}
\begin{example}
  We will count the number of configurations of Matt Zucker's 3D printed dodecahedral ``Squiggle Orbs'' \cite{Zucker2025}, illustrated in Figure \ref{fig:squig}. There are two essentially different pentagonal tiles that one can choose as a face, one with rotational and reflectional symmetry, and the other with one reflectional symmetry but no rotational symmetry.
  Because the second tile has five distinct rotations, we have
  \(t_\mathrm{id} = 6\),
  \(t_r = 1\), and
  \(t_f = 2\). Therefore, there are \[
    \frac{6^{12} + 24\cdot6^2\cdot1^2+20\cdot6^4+15\cdot6^6 + 6^6 + 24\cdot1\cdot6 + 20\cdot6^2 + 15\cdot6^4\cdot2^4}{120} = 18\,148\,896
  \] ways of tiling the dodecahedron with these tiles up to rotation and reflection.
\end{example}
\begin{figure}
  \centering
  \begin{subfigure}[b]{0.2\textwidth}
    \begin{minipage}{\linewidth}
      \centering
      \includegraphics[width=\textwidth]{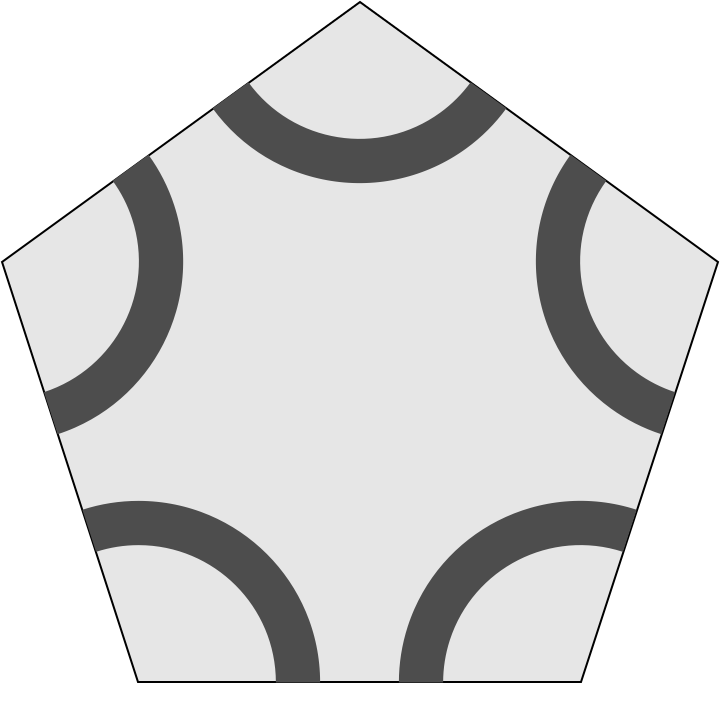}
      \includegraphics[width=\textwidth]{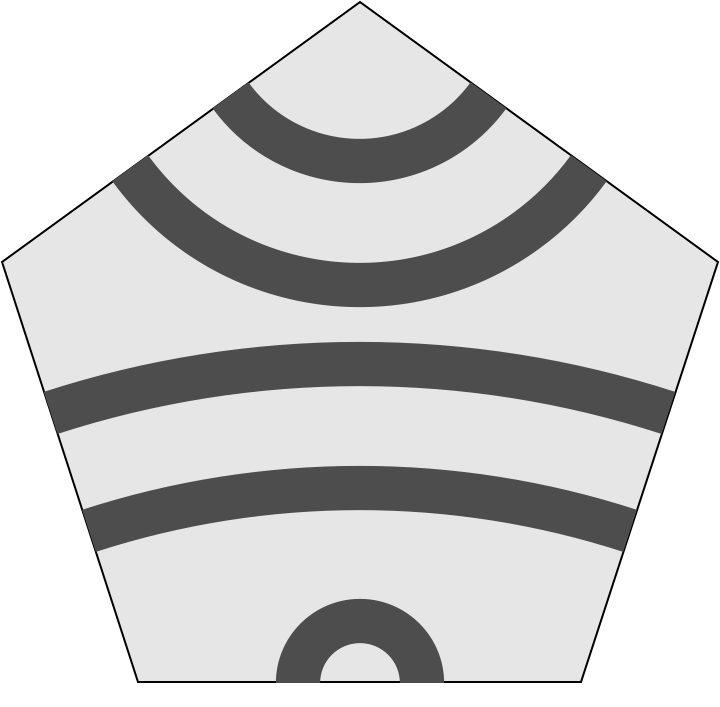}
    \end{minipage}
    \caption{}
  \end{subfigure}
  \hfill
  \begin{subfigure}[b]{0.41\textwidth}
    \centering
    \includegraphics[width=\textwidth]{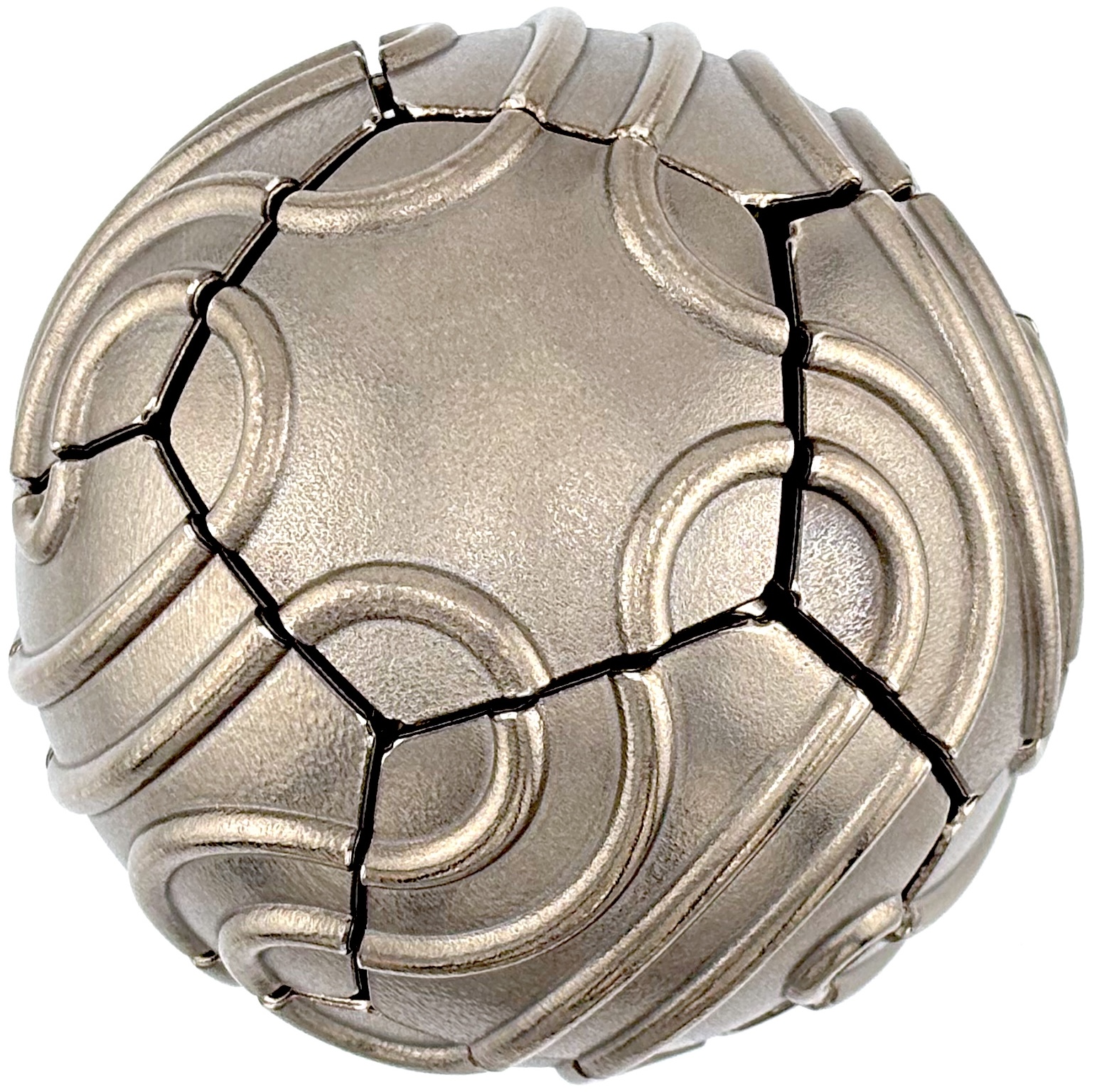}
    \caption{}
  \end{subfigure}
  \hfill
  \begin{subfigure}[b]{0.35\textwidth}
    \centering
        \includegraphics[width=\textwidth]{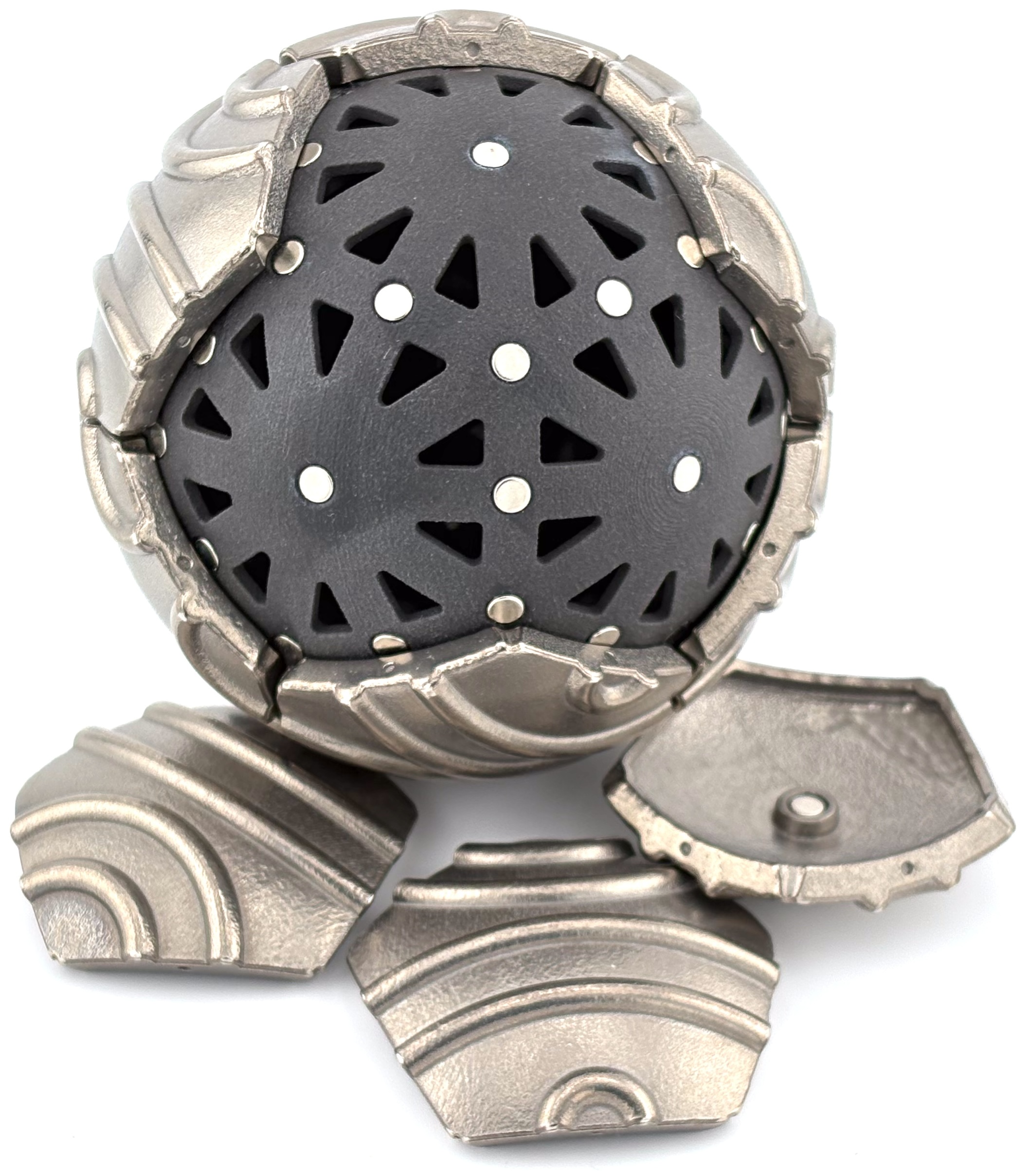}
    \caption{}
  \end{subfigure}
  \caption{In (a) the two essentially different tiles in Matt Zucker's configurable dodecahedral ``squiggle orb,'' and in (b) a photograph showing a dodecahedral squiggle orb, and in (c) a photograph of showing the squiggle orb with some faces removed. (Photos courtesy of Matt Zucker.)}
  \label{fig:squig}
\end{figure}
\begin{oeis*}
  The number of \(n\)-colorings of the faces of the dodecahedron up to the \(60\) symmetries of the rotational icosahedral group \(I\) appears in the OEIS as sequence \oeis{A000545}: \[
    1, 96, 9099, 280832, 4073375, 36292320, 230719293, 1145393152, 4707296613,
    \ldots,
  \]
  and the number of \(n\)-colorings up to the \(120\) symmetries of the icosahedral group \(I_h\) appears in the OEIS as sequence \oeis{A252705}: \[
    1, 82, 5379, 148648, 2085655, 18356514, 116081245, 574795936, 2359033605,
    \ldots.
  \]
\end{oeis*}

\subsection{Icosahedron}
The icosahedron, illustrated in \Cref{fig:icosahedron}, is a Platonic solid with \(20\) equilateral triangular faces. Its \(12\) vertices are generated by the orbits of \((0,1,\phi)\) under the action of its isometry group, the full icosahedral group \(I_h\) of order \(120\).

\par\medskip\noindent
\begin{minipage}{\textwidth}\captionsetup{type=figure}
  \centering
  \begin{subfigure}[b]{0.35\textwidth}
    \centering
    \includegraphics[scale=0.5]{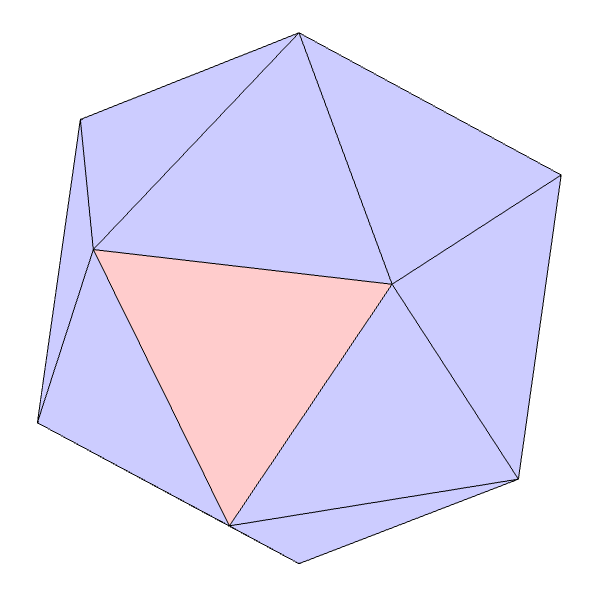}
    \caption{}
  \end{subfigure}
  \begin{subfigure}[b]{0.35\textwidth}
    \centering
    \includegraphics[scale=0.5]{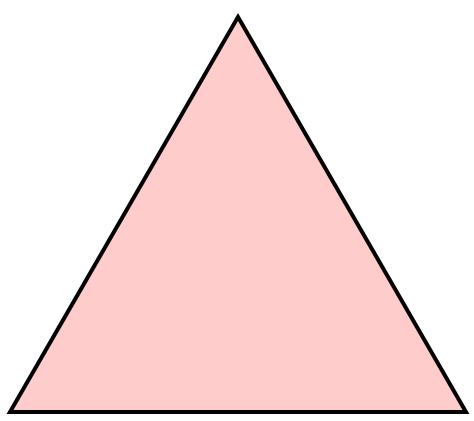}
    \caption{}
  \end{subfigure}
  \caption[Illustrations of the icosahedron and its faces.]{
    Illustrations of (a) the icosahedron with one face highlighted and (b) a face of the icosahedron.
  }
  \label{fig:icosahedron}
\end{minipage}
\par\medskip

The induced symmetry group on the faces of the icosahedron is the dihedral group of the triangle \[
  D_3 = \langle r, f \mid r^3 = f^2 = (rf)^2 = e \rangle,
\] where \(r\) acts by a \(120^\circ\) rotation, and \(f\) acts by reflection.

\begin{theorem}
  \label{thm:icosahedron}
  The number of ways of tiling the faces of the icosahedron up to full icosahedral symmetry \(I_h\) with \(\tid{}\) tiles of which
  \(\tr{}\) are fixed under \(120^\circ\) rotation and
  \(\tf{}\) are fixed under flipping horizontally
  is given by the expression
  \begin{align*}
    &\frac{1}{120}
    \left(
      \tid{20} +
      24\tid4 +
      20\tid6\tr2 +
      15\tid{10} +
      \tid{10} +
      24\tid{2} +
      20\tid3\tr{} +
      15\tid8\tf4
    \right).
  \end{align*}
\end{theorem}
\begin{proof}
  We proceed by cases, following \Cref{tabl:IcosahedronConjugacy}.
  \begin{description}
    \item[Case 1.] If \(A \in C_3^I\), which is the conjugacy class consisting of rotations by \(120^\circ\), then there are \(2\) faces that are fixed (rotated) under the action of \(A\), and the remaining \(18\) faces are partitioned into \(6\) parts of size \(|A| = 3\). Thus the number of face tilings fixed under \(A \in C_4^I\) is \(t_{\id}^{6}\tr{2}\).
    \item[Case 2.] If \(A \in C_7^I\), which is the conjugacy class consisting of rotations by \(60^\circ\) followed by a reflection, then there are \(2\) faces that are fixed (rotated) under the action of \(A^2\), and thus appear of an orbit of size \(2\) under \(\langle A \rangle\). The orbits of the remaining \(18\) faces consist of \(3\) parts of size \(|A| = 6\). Thus the number of face tilings fixed under \(A \in C_7^I\) is \(t_{\id}^{3}\tr{}\).
    \item[Case 3.] If \(A \in C_8^I\), which is the conjugacy class consisting of reflections, then there are \(4\) faces that are fixed (reflected) under the action of \(A\), and the remaining \(16\) faces are partitioned into \(8\) parts of size \(|A| = 2\). Thus the number of face tilings fixed under \(A \in C_8^I\) is \(t_{\id}^{8}\tf{4}\).
    \item[Case 4.] For the remaining elements, \(A \in I \setminus (C_3^I \cup C_7^I \cup C_8^I)\), which are elements of conjugacy classes that do not appear in the table, the action of the cyclic subgroup \(\langle A \rangle\) on the faces is free. Thus each partitions the \(30\) faces into parts of size \(|A|\), and thus the number of face tilings fixed under \(A\) is \(t_{\id}^{30/|A|}\).
  \end{description}
\end{proof}
\begin{table}[ht]
\[
  \begin{tabular}{P{1.7cm}P{1.6cm}lP{3.4cm}P{1.6cm}}
    conjugacy class  & cycle structure & non-maximal faces & vertex permutation generator & induced subgroup
    \\ \hline & & & & \\[-6pt]
    \multirow{2}{*}{\(\mathcal C^I_3\)} &
    \multirow{2}{*}{\((3^6,1^2)\)} &
    \(F_7 = (v_2,v_6,v_8)\) &
    \((v_2\ v_6\ v_8)\) &
    \(\langle r \rangle\)
    \\
    & &
    \(F_{13} = (v_5,v_7,v_{11})\) &
    \((v_5\ v_{11}\ v_7)\) &
    \(\langle r \rangle\)
    \\[3pt] \hline & & & & \\[-9pt]
    \multirow{2}{*}{\(\mathcal C^I_7\)} &
    \multirow{2}{*}{\((6^3,2)\)} &
    \(F_2 = (v_1,v_3,v_2)\) &
    \((v_1\ v_3\ v_2)\) &
    \(\langle r \rangle\)
    \\
    & &
    \(F_{20} = (v_{10},v_{12},v_{11})\) &
    \((v_{10}\ v_{11}\ v_{12})\) &
    \(\langle r \rangle\)
    \\[3pt] \hline & & & & \\[-9pt]
    \multirow{4}{*}{\(\mathcal C^I_8\)} &
    \multirow{4}{*}{\((2^8,1^4)\)} &
    \(F_9 = (v_3, v_5, v_9)\) &
    \((v_3\ v_9)(v_5)\) &
    \(\langle f \rangle\)
    \\
    & &
    \(F_{10} = (v_3, v_9, v_6)\) &
    \((v_3\ v_9)(v_6)\) &
    \(\langle f \rangle\)
    \\
    & &
    \(F_{11} = (v_4, v_8, v_{10})\) &
    \((v_4\ v_{10})(v_8)\) &
    \(\langle f \rangle\)
    \\
    & &
    \(F_{12} = (v_4, v_{10}, v_7)\) &
    \((v_4\ v_{10})(v_7)\) &
    \(\langle f \rangle\)
  \end{tabular}
\]
  \caption{The cycle structure of the permutations of the faces of an icosahedron under the full icosahedral group \(I_h\), along with the permutations of the corresponding vertices.}
  \label{tabl:IcosahedronConjugacy}
\end{table}

\begin{corollary}
  \label{thm:icosahedronRotation}
  The number of ways of tiling the faces of the icosahedron up to rotational icosahedral symmetry \(I\) with \(\tid{}\) tiles of which
  \(\tr{}\) are fixed under \(120^\circ\) rotation
  is given by the expression
  \begin{align*}
    &\frac{1}{120}
    \left(
      \tid{20} +
      24\tid4 +
      20\tid6\tr2 +
      15\tid{10}
    \right).
  \end{align*}
\end{corollary}
\begin{oeis*}
  The number of \(n\)-colorings of the faces of the icosahedron up to the \(60\) symmetries of the rotational icosahedral group \(I\) appears in the OEIS as sequence \oeis{A054472} (\(0\)-indexed): \[
    0, 1, 17824, 58130055, 18325477888, 1589459765875, 60935989677984, \ldots,
  \]
  and the number of \(n\)-colorings up to the \(120\) symmetries of the icosahedral group \(I_h\) appears in the OEIS as sequence \oeis{A252704} (\(1\)-indexed): \[
    1, 9436, 29131965, 9164844880, 794760482005, 30468267440892, \ldots.
  \]
\end{oeis*}
\subsection{Rhombic triacontahedron}
The rhombic triacontahedron, illustrated in \Cref{fig:rhombicTriacontahedron}, is a Catalan solid with \(30\) rhombic faces and is the polyhedral dual of the icosidodecahedron.
Its \(32\) vertices are generated by the orbits of \(v = (1,1,1)\) and \(v' = (0,1,\phi)\) under the action of its isometry group, the full icosahedral group \(I_h\) of order \(120\). The orbits have size \(20\) and \(12\) respectively.

\par\medskip\noindent
\begin{minipage}{\textwidth}\captionsetup{type=figure}
  \centering
    \begin{subfigure}[b]{0.35\textwidth}
      \centering
      \includegraphics[scale=0.5]{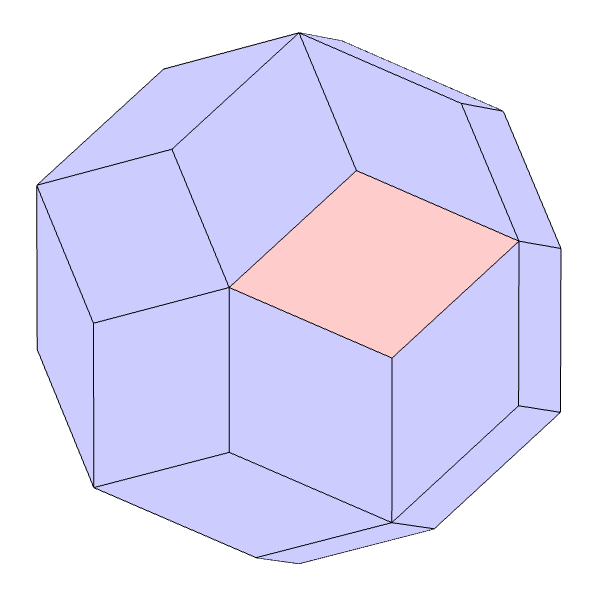}
      \caption{}\label{fig:rhombicTriacontahedronA}
    \end{subfigure}
    \begin{subfigure}[b]{0.35\textwidth}
      \centering
      \includegraphics[scale=0.5]{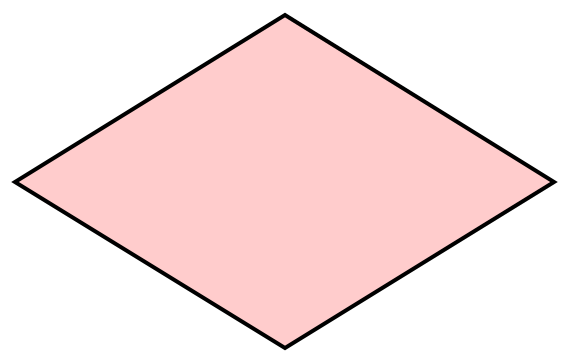}
      \caption{}\label{fig:rhombicTriacontahedronB}
    \end{subfigure}
    \caption[Illustrations of the rhombic triacontahedron and its faces.]{
      Illustrations of (a) the rhombic triacontahedron with one face highlighted and (b) a face of the rhombic triacontahedron.
    }
    \label{fig:rhombicTriacontahedron}
\end{minipage}
\par\medskip

The induced symmetry group on the faces of the rhombic triacontahedron is the dihedral group of the rhombus \[
  D_2 = \langle r, f \mid r^2 = f^2 = (rf)^2 = e \rangle,
\] where we follow the convention that \(rf\) swaps the \(v\) vertices (on the short diagonal) and \(f\) swaps the \(v'\) vertices (on the long diagonal).

\begin{theorem}
  \label{thm:rhombicTriacontahedron}
  The number of ways of tiling the faces of the rhombic triacontahedron up to full icosahedral symmetry \(I_h\) with \(\tid{}\) tiles of which
  \(\tr{}\) are fixed under \(180^\circ\) rotation,
  \(\tf{}\) are fixed under flipping across the short diagonal, and
  \(\trf{}\) are fixed under flipping across the long diagonal
  is given by the expression
  \begin{equation}
    \frac{1}{120}
    \left(
      \tid{30} +
      24\tid6 +
      20\tid{10} +
      15\tid{14}\tr2+
      \tid{15} +
      24\tid3 +
      20\tid5 +
      15\tid{13}\tf2\trf2
    \right).
  \end{equation}
\end{theorem}
\begin{proof}
  We proceed by cases, following \Cref{tabl:RhombicTriacontahedronConjugacy}.
  \begin{description}
    \item[Case 1.] If \(A \in C_4^I\), which is the conjugacy class consisting of rotations by \(180^\circ\), then there are \(2\) faces that are fixed (rotated) under the action of \(A\), and the remaining \(28\) faces are partitioned into \(14\) parts of size \(|A| = 2\). Thus the number of face tilings fixed under \(A \in C_4^I\) is \(t_{\id}^{14}\tr{2}\).
    \item[Case 2.] If \(A \in C_8^I\), which is the conjugacy class consisting of reflections, then there are \(2\) faces that are fixed (reflected) under the action of \(A\), and the remaining \(28\) faces are partitioned into \(14\) parts of size \(|A| = 2\). Thus the number of face tilings fixed under \(A \in C_8^I\) is \(t_{\id}^{14}\tf{2}\).
    \item[Case 3.] For the remaining elements, \(A \in I \setminus (C_4^I \cup C_8^I)\), which are elements of conjugacy classes that do not appear in the table, the action of the cyclic subgroup \(\langle A \rangle\) on the faces is free. Thus each partitions the \(30\) faces into parts of size \(|A|\), and thus the number of face tilings fixed under \(A\) is \(t_{\id}^{30/|A|}\).
  \end{description}
\end{proof}

\begin{table}
  \[
  \begin{tabular}{P{1.7cm}P{1.6cm}lP{3.4cm}P{1.6cm}}
    conjugacy class  & cycle structure & non-maximal faces & vertex permutation generator & induced subgroup
    \\ \hline & & & & \\[-6pt]
    \multirow{2}{*}{\(\mathcal C^I_4\)} &
    \multirow{2}{*}{\((2^{14},1^2)\)} &
    \(F_{16} = (v_7, v'_7, v_{13}, v'_5)\) &
    \((v_7\ v_{13})(v'_5\ v'_7)\) &
    \(\langle r \rangle\)
    \\
    & &
    \(F_{17} = (v_8, v'_6, v_{14}, v'_8)\) &
    \((v_8\ v_{14})(v'_6\ v'_8)\) &
    \(\langle r \rangle\)
    \\[3pt] \hline & & & & \\[-9pt]
    \multirow{4}{*}{\(\mathcal C^I_8\)} &
    \multirow{4}{*}{\((2^{13},1^4)\)} &
    \(F_{11} = (v'_3, v_9, v'_9, v_{10})\) &
    \((v'_3\ v'_9)(v_9)(v_{10})\) &
    \(\langle f \rangle\)
    \\
    & &
    \(F_{15} = (v'_4, v_{12}, v'_{10}, v_{11})\) &
    \((v'_4\ v'_{10})(v_{11})(v_{12})\) &
    \(\langle f \rangle\)
    \\
    & &
    \(F_{16} = (v_7, v'_7, v_{13}, v'_5)\) &
    \((v_7\ v_{13})(v'_5)(v'_7)\) &
    \(\langle rf \rangle\)
    \\
    & &
    \(F_{17} = (v_8, v'_6, v_{14}, v'_8)\) &
    \((v_8\ v_{14})(v'_6)(v'_8)\) &
    \(\langle rf \rangle\)
  \end{tabular}
  \]
  \caption{The cycle structure of the permutations of the faces of a rhombic triacontahedron under the full icosahedral group \(I_h\), along with the permutations of the corresponding vertices.}
  \label{tabl:RhombicTriacontahedronConjugacy}
\end{table}

\begin{corollary}
  The number of ways of tiling the faces of the rhombic triacontahedron up to full icosahedral symmetry \(I_h\) with \(\tid{}\) tiles of which
  \(\tr{}\) are fixed under \(180^\circ\) rotation
  is given by the expression
  \begin{equation}
    \frac{1}{60}\left(
      \tid{30} +
      24\tid6 +
      20\tid{10} +
      15\tid{14}\tr2
    \right).
  \end{equation}
\end{corollary}

\begin{oeis*}
  The number of \(n\)-colorings up to the \(60\) rotational symmetries of the icosahedral group \(I\) appears in the OEIS as sequence \oeis{A282670} (\(0\)-indexed): \[
    0, 1, 17912448, 3431529649899, 19215359484207104, 15522042948408209375,
		\ldots,
  \]
  and the number of \(n\)-colorings of the faces of the rhombic triacontahedron up to the \(120\) symmetries of the icosahedral group \(I_h\) appears in the OEIS as sequence \oeis{A337963}: \[
    1, 8972888, 1715781087090, 9607681898535232, 7761021569825850025,
		\ldots.
  \]
\end{oeis*}
The faces of the rhombic triacontahedron are in bijection with the faces of the icosahedron and dodecahedron, and so by the analogous process to the construction shown in \Cref{fig:SolLewitt}, there are \(\oeis{A282670}(2) = 17\,912\,448\) subsets of the edges of the icosahedron and dodecahedron up to rotation and reflection. Vejdemo-Johansson computes the number of connected, non-planar, proper edge subsets of the dodecahedron is \(\oeis{A393693}(4) = 2\,423\,206\) and of the icosahedron is \(\oeis{A393693}(5) = 16\,096\,166\) \cite{VejdemoJohansson2026}.
\begin{example}
  Dave Barber \cite{Barber2009} proposed a variant of Tantrix Rock called \textit{Lambda Rock} on the faces of a rhombic triacontahedron, a simplified version of which is shown in Figure \ref{fig:LambdaRock}, which also shows tiles such that \(\tid{} = 6\), \(t_r = t_f = t_{rf} = 2\). Thus, there are \(1842282664408676415816\) ways of tiling the rhombic dodecahedron with these tiles up to rotation and reflection of the rhombic triacontahedron, while disallowing rotation and reflection of the faces.

  \begin{figure}
    \begin{subfigure}[b]{0.19\linewidth}
      \centering
      \includegraphics[width=\linewidth]{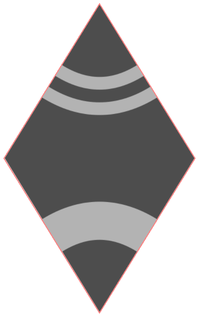}
      \caption{}\label{fig:LambdaRock1}
    \end{subfigure}
    \hfill
    \begin{subfigure}[b]{0.19\linewidth}
      \centering
      \includegraphics[width=\linewidth]{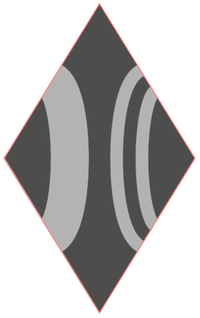}
      \caption{}\label{fig:LambdaRock2}
    \end{subfigure}
    \hfill
    \begin{subfigure}[b]{0.19\linewidth}
      \centering
      \includegraphics[width=\linewidth]{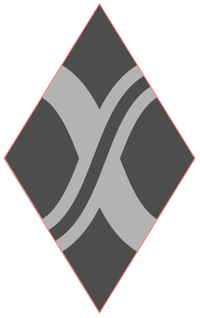}
      \caption{}\label{fig:LambdaRock3}
    \end{subfigure}
    \hfill
    \begin{subfigure}[b]{0.39\linewidth}
      \centering
      \includegraphics[width=\linewidth]{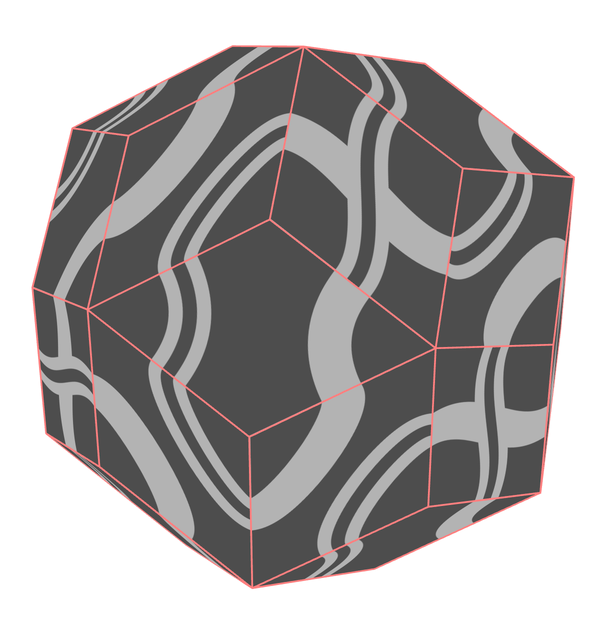}
      \caption{}\label{fig:LambdaRockRhombicTriacontahedron}
    \end{subfigure}
    \caption{The three essentially different tile designs (a) to (c) used in a simplified version of Dave Barber's Lambda Rock puzzle, along with (d) a rhombic triacontahedron tiled with random choices of these tiles.}
    \label{fig:LambdaRock}
  \end{figure}
\end{example}

\subsection{Triakis icosahedron, deltoidal hexecontahedron, and pentakis dodecahedron}

The triakis icosahedron, illustrated in \Cref{subfig:polyhedron_TriakisIcosahedron}, is a Catalan solid with \(60\) isosceles triangular faces and is the polyhedral dual of the truncated dodecahedron.
Its \(32\) vertices are generated by the orbits of \(v = (1,1,1)\) and \(v' = c(0,1,\phi)\) where \(\displaystyle c = (7 \sqrt{5}-5)/10 \approx 1.06525\) under the action of its isometry group, the full icosahedral group \(I_h\) of order \(120\). The orbits have size \(20\) and \(12\) respectively.

\par\medskip\noindent
\begin{minipage}{\textwidth}\captionsetup{type=figure}
  \centering
  \begin{subfigure}[b]{0.35\textwidth}
    \centering
    \includegraphics[scale=0.4]{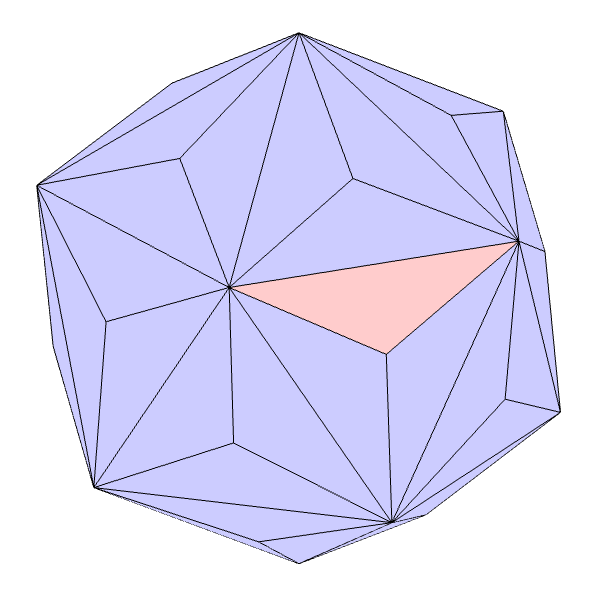}
    \caption{}\label{subfig:polyhedron_TriakisIcosahedron}
  \end{subfigure}
  \begin{subfigure}[b]{0.3\textwidth}
    \centering
    \includegraphics[scale=0.4]{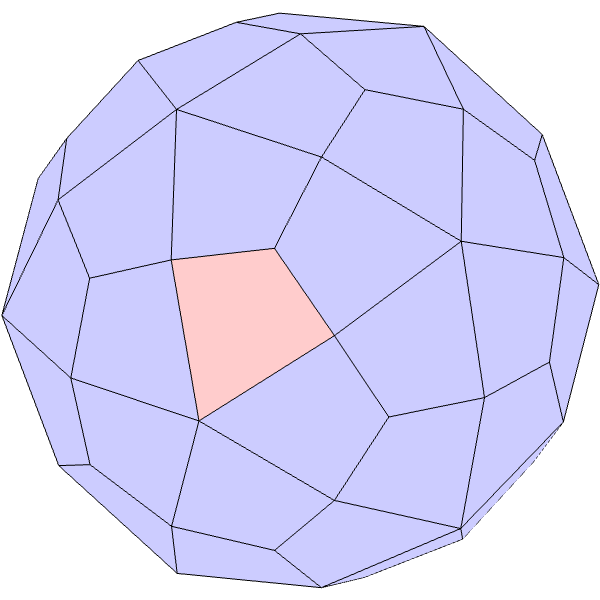}
    \caption{}\label{subfig:polyhedron_DeltoidalHexecontahedron}
  \end{subfigure}
  \begin{subfigure}[b]{0.3\textwidth}
    \centering
    \includegraphics[scale=0.4]{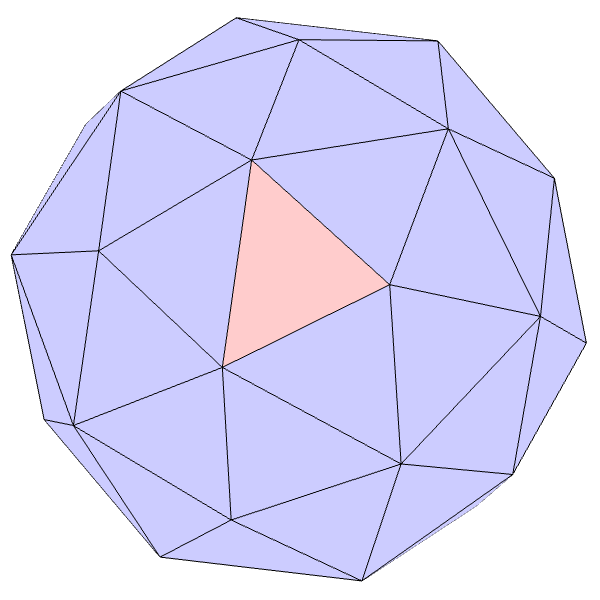}
    \caption{}\label{subfig:polyhedron_PentakisDodecahedron}
  \end{subfigure}
  \\
  \begin{subfigure}[b]{0.35\textwidth}
    \centering
    \includegraphics[scale=0.4]{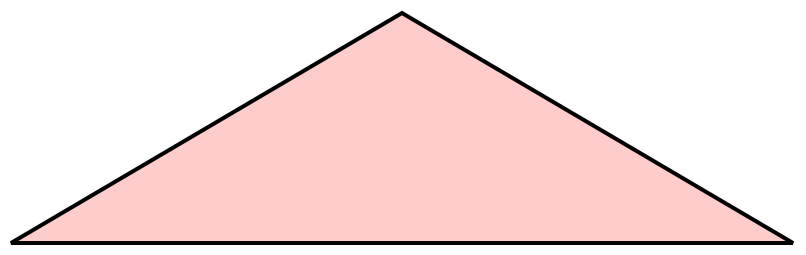}
    \caption{}
  \end{subfigure}
  \begin{subfigure}[b]{0.3\textwidth}
    \centering
    \includegraphics[scale=0.4]{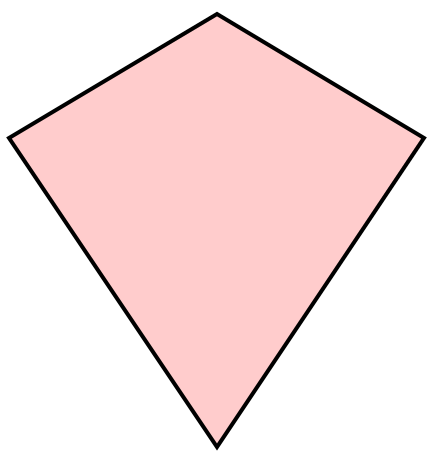}
    \caption{}
  \end{subfigure}
  \begin{subfigure}[b]{0.3\textwidth}
    \centering
    \includegraphics[scale=0.4]{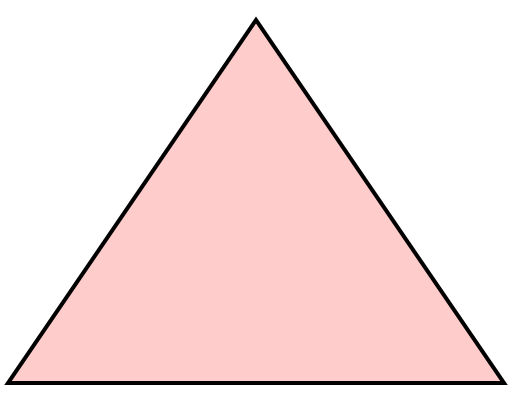}
    \caption{}
  \end{subfigure}
  \caption[Illustrations of the pentagonal hexecontahedron and its faces.]{
    Illustrations of
    (a) the triakis icosahedron with a face highlighted,
    (b) the deltoidal hexecontahedron with a face highlighted,
    (c) the pentakis dodecahedron with a face highlighted,
    (d) a face of the triakis icosahedron,
    (e) a face of the deltoidal hexecontahedron, and
    (f) a face of the pentakis dodecahedron.
  }
  \label{fig:triakisIcosahedron}
\end{minipage}
\par\medskip

The deltoidal hexecontahedron, illustrated in \Cref{subfig:polyhedron_DeltoidalHexecontahedron}, is a Catalan solid with \(60\) deltoidal faces and is the polyhedral dual of the rhombicosidodecahedron.
Its \(62\) vertices are generated by the orbits of \(v = (1,0,0)\), \(v' = c_1(1,1,1)\), and \(v'' = c_2(0,1,\phi)\) where \(c_1 = \left(\sqrt{5}+4\right)/11\) and \(c_2 = \phi/3\), under the action of its isometry group, the full icosahedral group \(I_h\) of order \(120\). The orbits have size \(30\), \(20\), and \(12\) respectively.

The pentakis dodecahedron, illustrated in \Cref{subfig:polyhedron_PentakisDodecahedron}, is a Catalan solid with \(60\) isosceles triangular faces and is the polyhedral dual of the truncated icosahedron.
Its \(32\) vertices are generated by the orbits of \(v = (1,1,1)\) and \(v' = c(0,1,\phi)\) where \(\displaystyle c = (3\sqrt{5}+27)/38 \approx 0.887058\), under the action of its isometry group, the full icosahedral group \(I_h\) of order \(120\). The orbits have size \(20\) and \(12\) respectively.

The induced symmetry group on the faces of the triakis icosahedron, of the deltoidal hexecontahedron, and of the pentakis dodecahedron is the reflection group \(\langle f \mid f^2 = e \rangle\).

\begin{theorem}
  \label{thm:triakisIcosahedron}
  The number of ways of tiling the faces of the triakis icosahedron, pentakis dodecahedron, and deltoidal hexecontahedron up to full icosahedral symmetry \(I_h\) with \(\tid{}\) tiles, where \(\tf{}\) of them have reflectional symmetry, is given by the expression
  \begin{align*}
    &\frac{1}{120}
    \left(
      \tid{60} +
      24\tid{12} +
      20\tid{20} +
      15\tid{30}+
      \tid{30} +
      24\tid6 +
      20\tid{10} +
      15\tid{28}\tf4
    \right) \\
    &\qquad=\frac{1}{120} \left(
      \tid{60} +
      16\tid{30} +
      20\tid{20} +
      24\tid{12} +
      20\tid{10} +
      24\tid6 +
      15\tid{28}\tf4
    \right).
  \end{align*}
\end{theorem}
\begin{proof}
  We proceed by cases, following \Cref{tabl:TriakisIcosahedronConjugacy,tabl:DeltoidalHexecontahedronConjugacy,tabl:PentakisDodecahedronConjugacy}.
  \begin{description}
    \item[Case 1.] If \(A \in C_8^I\), which is the conjugacy class consisting of the fifteen reflections of the polyhedron, then there are \(4\) faces that are fixed (reflected) under the action of \(A\), and the remaining \(56\) faces are partitioned into \(28\) parts of size \(|A| = 2\). Thus the number of face tilings fixed under \(A \in C_8^I\) is \(t_{\id}^{28}\tf{4}\).
    \item[Case 2.] For the remaining elements, \(A \in I \setminus C_8^I\), which are elements of conjugacy classes that do not appear in the tables, the action of the cyclic subgroup \(\langle A \rangle\) on the faces is free. Thus each partitions the \(60\) faces into parts of size \(|A|\), and thus the number of face tilings fixed under \(A\) is \(t_{\id}^{60/|A|}\).
  \end{description}
\end{proof}

\begin{table}[ht]
  \[
  \begin{tabular}{P{1.7cm}P{1.6cm}lP{3.4cm}P{1.6cm}}
    conjugacy class  & cycle structure & non-maximal faces & vertex permutation generator & induced subgroup
    \\[3pt] \hline & & & & \\[-9pt]
    \multirow{4}{*}{\(\mathcal C^I_8\)} &
    \multirow{4}{*}{\((2^{28},1^4)\)} &
    \(F_{20} = (v'_3,v_9,v'_9)\) &
    \((v'_3\ v'_9)(v_9)\) &
    \(\langle f \rangle\)
    \\
    & &
    \(F_{24} = (v'_3,v'_9,v_{10})\) &
    \((v'_3\ v'_9)(v_{10})\) &
    \(\langle f \rangle\)
    \\
    & &
    \(F_{29} = (v'_4,v_{12},v'_{10})\) &
    \((v'_4\ v'_{10})(v_{12})\) &
    \(\langle f \rangle\)
    \\
    & &
    \(F_{30} = (v'_4,v'_{10},v_{11})\) &
    \((v'_4\ v'_{10})(v_{11})\) &
    \(\langle f \rangle\)
  \end{tabular}
  \]
  \caption{
    The cycle structure of the permutations of the vertices of a triakis icosahedron under each conjugacy class of \(I_h\).
  }
  \label{tabl:TriakisIcosahedronConjugacy}
\end{table}

\begin{table}[ht]
  \[
  \begin{tabular}{P{1.7cm}P{1.6cm}lP{3.4cm}P{1.6cm}}
    conjugacy class  & cycle structure & non-maximal faces & vertex permutation generator & induced subgroup
    \\[3pt] \hline & & & & \\[-9pt]
    \multirow{4}{*}{\(\mathcal C^I_8\)} &
    \multirow{4}{*}{\((2^{28},1^4)\)} &
    \(F_{29} = (v''_{10},v'_5,v''_{18},v_9)\) &
    \((v''_{10}\ v''_{18})(v'_5)(v_9)\) &
    \(\langle f \rangle\)
    \\
    & &
    \(F_{30} = (v''_{11},v_{10},v''_{19},v'_6)\) &
    \((v''_{11}\ v''_{19})(v_{10})(v'_6)\) &
    \(\langle f \rangle\)
    \\
    & &
    \(F_{31} = (v''_{12},v_{11},v''_{20},v'_7)\) &
    \((v''_{12}\ v''_{20})(v_{11})(v'_7)\) &
    \(\langle f \rangle\)
    \\
    & &
    \(F_{32} = (v''_{13},v'_8,v''_{21},v_{12})\) &
    \((v''_{13}\ v''_{21})(v'_8)(v_{12})\) &
    \(\langle f \rangle\)
  \end{tabular}
  \]
  \caption{
    The cycle structure of the permutations of the vertices of a deltoidal hexecontahedron under each conjugacy class of \(I_h\).
  }
  \label{tabl:DeltoidalHexecontahedronConjugacy}
\end{table}

\begin{table}[ht]
  \[
  \begin{tabular}{P{1.7cm}P{1.6cm}lP{3.4cm}P{1.6cm}}
    conjugacy class  & cycle structure & non-maximal faces & vertex permutation generator & induced subgroup
    \\[3pt] \hline & & & & \\[-9pt]
    \multirow{4}{*}{\(\mathcal C^I_8\)} &
    \multirow{4}{*}{\((2^{28},1^4)\)} &
    \(F_{29} = (v_7,v'_7,v_{13})\) &
    \((v_7\ v_{13})(v'_7)\) &
    \(\langle f \rangle\)
    \\
    & &
    \(F_{30} = (v_7,v_{13},v'_5)\) &
    \((v_7\ v_{13})(v'_5)\) &
    \(\langle f \rangle\)
    \\
    & &
    \(F_{31} = (v_8,v'_6,v_{14})\) &
    \((v_8\ v_{14})(v'_6)\) &
    \(\langle f \rangle\)
    \\
    & &
    \(F_{32} = (v_8,v_{14},v'_8)\) &
    \((v_8\ v_{14})(v'_8)\) &
    \(\langle f \rangle\)
  \end{tabular}
  \]
  \caption{
    The cycle structure of the permutations of the vertices of a pentakis dodecahedron under each conjugacy class of \(I_h\).
  }
  \label{tabl:PentakisDodecahedronConjugacy}
\end{table}

\begin{corollary}
  \label{thm:triakisIcosahedronRotation}
  The number of ways of tiling the faces of the triakis icosahedron, pentakis dodecahedron, and deltoidal hexecontahedron up to rotational icosahedral symmetry \(I\) with \(\tid{}\) tiles is given by the expression
  \begin{align*}
    &\frac{1}{60}
    \left(
      \tid{60} +
      24\tid{12} +
      20\tid{20} +
      15\tid{30}
    \right).
  \end{align*}
\end{corollary}
\begin{oeis*}
  The number of \(n\)-colorings of the faces of the triakis icosahedron (equivalently pentakis dodecahedron and deltoidal hexecontahedron) up to the \(60\) symmetries of the icosahedral group \(I\) has been added to the OEIS as sequence \oeis{A378478} (\(0\)-indexed): \[
    0, 1, 19215358678900736, 706519304586988199183738259, \ldots,
  \]
  and the number of colorings up to the \(120\) symmetries of the icosahedral group \(I_h\) has been added to the OEIS as sequence \oeis{A378476} (\(0\)-indexed): \[
    0, 1, 9607679885269312, 353259652293727442874919719, \ldots.
  \]
\end{oeis*}

\subsection{Disdyakis triacontahedron}
The disdyakis triacontahedron, illustrated in Figure \ref{fig:disdyakisTriacontahedron}, is a Catalan solid with \(120\) scalene triangular faces that is the polyhedral dual of the truncated icosidodecahedron.
Its \(62\) vertices are generated by the orbits of \(v = (1,0,0)\), \(v' = c_1(1,1,1)\), and \(v'' = c_2(0,1,\phi)\) where
\(c_1 = \left(4-\sqrt{5}\right)/3 \approx 0.587977\) and
\(c_2 = \left(3 \sqrt{5}-1\right)/10 \approx 0.57082\), under the action of its isometry group, the full icosahedral group \(I_h\) of order \(120\).The orbits have size \(30\), \(20\) and \(12\) respectively.

\par\medskip\noindent
\begin{minipage}{\textwidth}\captionsetup{type=figure}
  \centering
    \begin{subfigure}[b]{0.35\textwidth}
      \centering
      \includegraphics[scale=0.5]{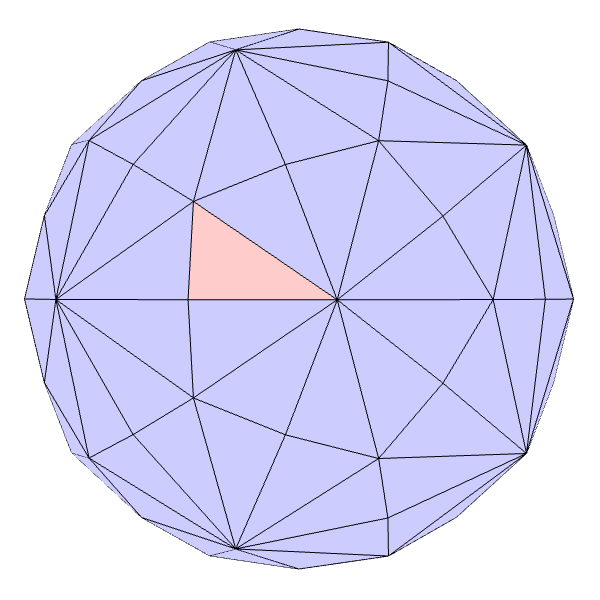}
      \caption{}
    \end{subfigure}
    \begin{subfigure}[b]{0.35\textwidth}
      \centering
      \includegraphics[scale=0.5]{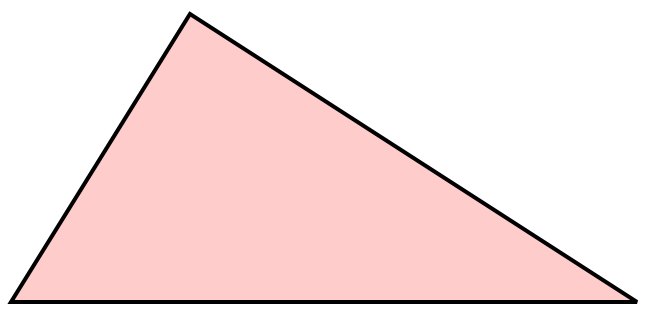}
      \caption{}
    \end{subfigure}
    \caption[Illustrations of the disdyakis triacontahedron and its faces.]{
      Illustrations of (a) the disdyakis triacontahedron with one face highlighted and (b) a face of the disdyakis triacontahedron.
    }
    \label{fig:disdyakisTriacontahedron}
\end{minipage}
\par\medskip

The induced symmetry group of the faces of the disdyakis triacontahedron is the trivial group.

\begin{theorem}
  \label{thm:disdyakisTriacontahedron}
  The number of ways of tiling the faces of the disdyakis triacontahedron up to full icosahedral symmetry \(I_h\) with \(\tid{}\) tiles is given by the expression
  \begin{align*}
    &\frac{1}{120}
    \left(
      \tid{120} +
      24\tid{24} +
      20\tid{40} +
      15\tid{60} +
      \tid{60} +
      24\tid{24} +
      20\tid{40} +
      15\tid{60}
    \right) \\
    &\qquad=\frac{1}{120} \left(
      \tid{120} +
      31\tid{60} +
      40\tid{40} +
      48\tid{24}
    \right).
  \end{align*}
\end{theorem}
\begin{proof}
  Because each of the \(120\) faces of the disdyakis triacontahedron corresponds to the fundamental domain of the full icosahedral group \(I_h\), the group action is free. Thus, for each \(A \in I_h\), the orbits of \(\langle A \rangle\) partition the faces into parts of size \(|A|\).
\end{proof}

\begin{corollary}
  The number of ways of tiling the faces of the disdyakis triacontahedron up to rotational icosahedral symmetry \(I\) with \(\tid{}\) tiles is given by the expression
  \[
    \frac{1}{60}
    \left(
      \tid{120} +
      24\tid{24} +
      20\tid{40} +
      15\tid{60} +
    \right).
  \]
\end{corollary}

\begin{oeis*}
  The number of \(n\)-colorings of the faces of the disdyakis triacontahedron up to the \(60\) symmetries of the icosahedral group \(I\) can be derived from OEIS sequence \oeis{A378478} as \(A378478(n^2)\) (\(0\)-indexed): \[
    0, 1, 22153799929748598169960860333637632, \ldots,
  \]
  and the number of colorings up to the \(120\) symmetries of the icosahedral group \(I_h\) has been added to the OEIS as sequence \oeis{A378477} (\(0\)-indexed): \[
    0, 1, 11076899964874299238703297447907328, \ldots.
  \]
\end{oeis*}
\begin{example}
  Because the faces of the disdyakis triacontahedron represent the fundamental domains of icosahedral symmetry \(I_h\), we can recover \Cref{thm:disdyakisTriacontahedron} from \Cref{thm:dodecahedron,thm:icosahedron,thm:rhombicTriacontahedron,thm:triakisIcosahedron}, as illustrated in \Cref{fig:disdyakisTriacontahedronPartitions}.

  We can partition the pentagonal faces of the dodecahedron into ten right triangles, each with \(t_{\id}\) possible tile designs so that, \[
    \hat t_{\id} = t_{\id}^{10}, \qquad
    \hat t_{r} = t_{\id}^2, \qquad
    \hat t_{f} = t_{\id}^5.
  \]

  We can partition the equilateral triangular faces of the icosahedron into six right triangles, each with \(t_{\id}\) possible tile designs so that, \[
    \hat t_{\id} = t_{\id}^6, \qquad
    \hat t_{r} = t_{\id}^2, \qquad
    \hat t_{f} = t_{\id}^3.
  \]

  We can partition the rhombic faces of the rhombic triacontahedron into four right triangles, each with \(t_{\id}\) possible tile designs so that, \[
    \hat t_{\id} = t_{\id}^4, \qquad
    \hat t_{r} = t_{\id}^2, \qquad
    \hat t_{f} = t_{\id}^2, \qquad
    \hat t_{rf} = t_{\id}^2.
  \]

  We can partition the obtuse triangular (or kite or acute triangular)  faces of the triakis icosahedron (respectively deltoidal hexecontahedron or pentakis dodecahedron) into two right triangles, each with \(t_{\id}\) possible tile designs so that, \[
    \hat t_{\id} = t_{\id}^2, \qquad
    \hat t_{f} = t_{\id}.
  \]

  Substituting these values into \Cref{thm:dodecahedron,thm:icosahedron,thm:rhombicTriacontahedron,thm:triakisIcosahedron} respectively recovers \Cref{thm:disdyakisTriacontahedron}.
\end{example}
\begin{figure}
  \centering
  \begin{minipage}{0.64\textwidth}
    \begin{subfigure}[b]{0.32\textwidth}
      \centering
      \includegraphics[scale=0.38]{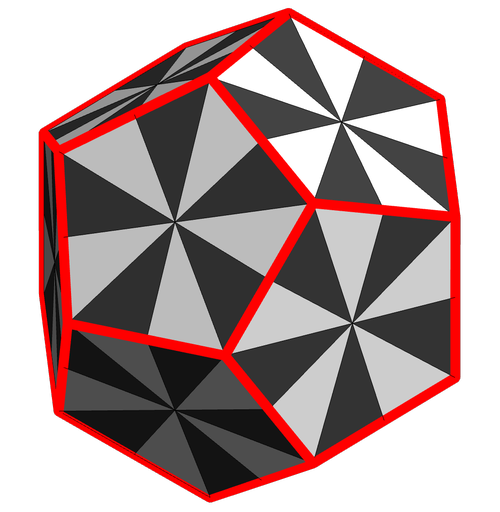}
      \caption{}
    \end{subfigure}
    \begin{subfigure}[b]{0.32\textwidth}
      \centering
      \includegraphics[scale=0.38]{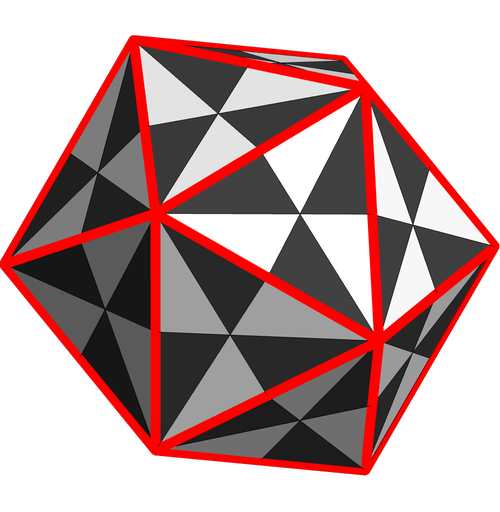}
      \caption{}
    \end{subfigure}
    \begin{subfigure}[b]{0.32\textwidth}
      \centering
      \includegraphics[scale=0.38]{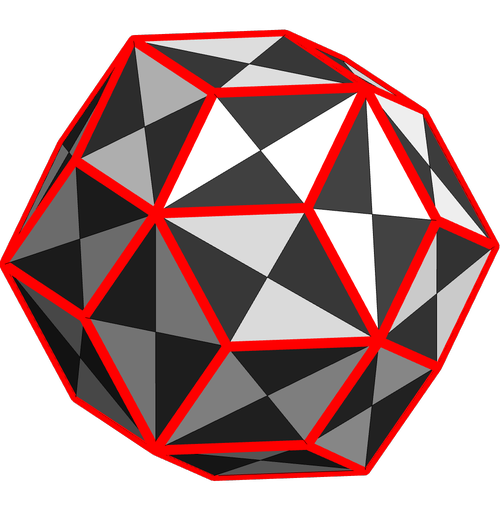}
      \caption{}
    \end{subfigure}
    \\
    \begin{subfigure}[b]{0.32\textwidth}
      \centering
      \includegraphics[scale=0.38]{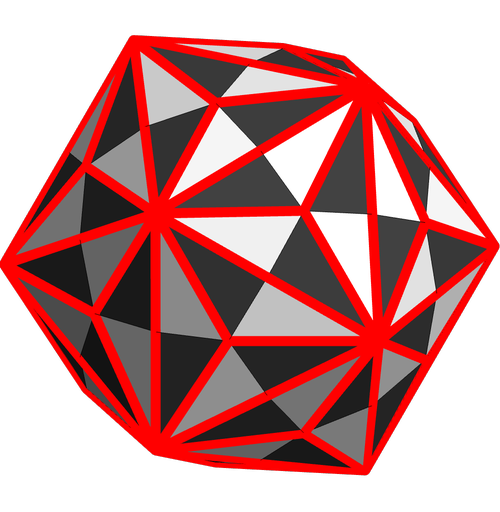}
      \caption{}
    \end{subfigure}
    \begin{subfigure}[b]{0.32\textwidth}
      \centering
      \includegraphics[scale=0.38]{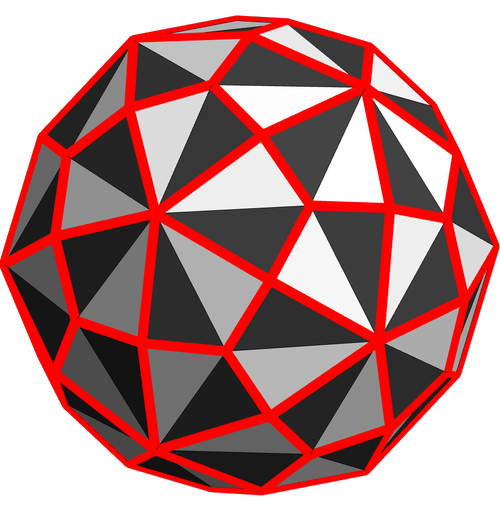}
      \caption{}
    \end{subfigure}
    \begin{subfigure}[b]{0.32\textwidth}
      \centering
      \includegraphics[scale=0.38]{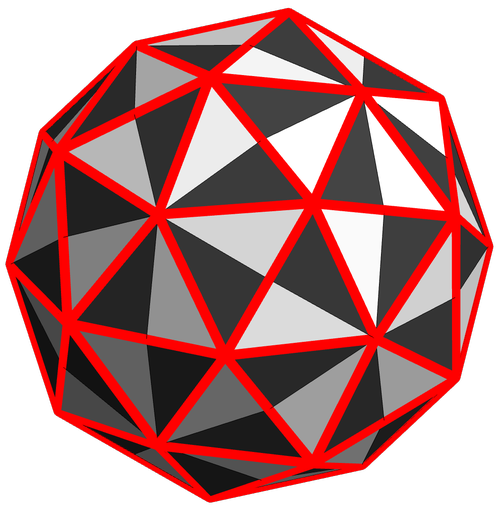}
      \caption{}
    \end{subfigure}
  \end{minipage}
  \begin{minipage}{0.35\textwidth}
    \begin{subfigure}[b]{\linewidth}
      \centering
      \includegraphics[width=\linewidth]{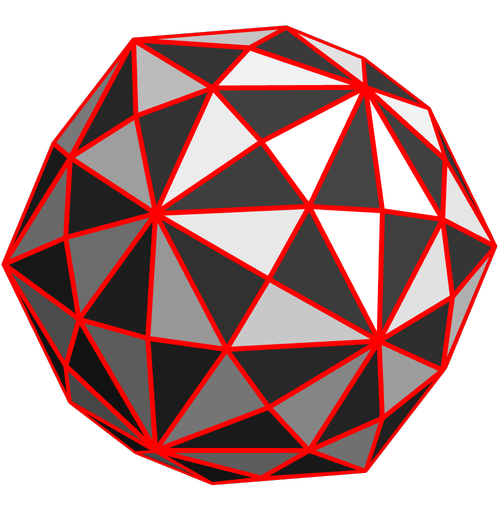}
      \caption{}
    \end{subfigure}
  \end{minipage}
  \caption{The faces of the (a) dodecahedron, (b) icosahedron, (c) rhombic triacontahedron, (d) triakis icosahedron, (e) deltoidal hexecontahedron, and (f) pentakis dodecahedron can be partitioned into right triangles that correspond to the faces of the (g) disdyakis triacontahedron.}
  \label{fig:disdyakisTriacontahedronPartitions}
\end{figure}

\subsection{Pentagonal hexecontahedron}
The pentagonal hexecontahedron, illustrated in \Cref{fig:pentagonalHexecontahedron}, is a Catalan solid with \(60\) (non-regular) pentagonal faces and is the polyhedral dual of the snub dodecahedron.
Its \(92\) vertices are generated by the orbits of
\(v=(1,1,1)\) and \(v'=\alpha(0,1,\phi)\), and \(v''=(\beta_1,\beta_2,\beta_3)\) under the rotational icosahedral group \(I\), where
\(\alpha \approx 0.954802\) is a root of \(31 x^6-37 x^5-42 x^4+47 x^3+2 x^2-x-1\),
\(\beta_1 \approx 1.71556\) is a root of \(x^6-4 x^4-x^3+4 x^2+2 x-1\),
\(\beta_2 \approx 0.178737\) is a root of \(x^6+x^5+12 x^4-14 x^3-18 x^2-2 x+1\), and
\(\beta_3 \approx 0.157803\) is a root of \(x^6+9 x^5+24 x^4+12 x^3-11 x^2-5 x+1\) under the action of its isometry group, the rotational icosahedral group \(I\) of order \(60\). The orbits have size \(20\), \(12\), and \(60\) respectively.

\par\medskip\noindent
\begin{minipage}{\textwidth}\captionsetup{type=figure}
  \centering
    \begin{subfigure}[b]{0.35\textwidth}
      \centering
      \includegraphics[scale=0.5]{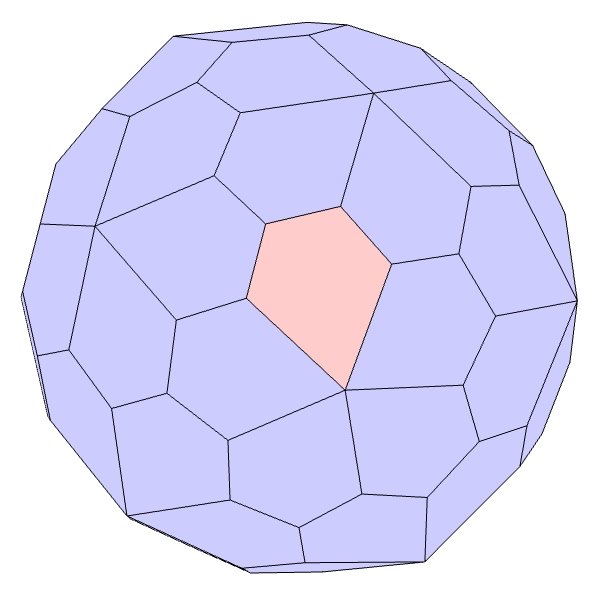}
      \caption{}
    \end{subfigure}
    \begin{subfigure}[b]{0.35\textwidth}
      \centering
      \includegraphics[scale=0.5]{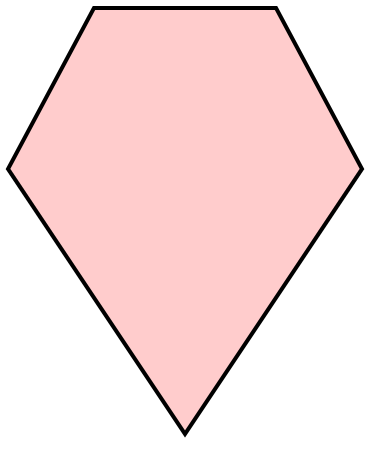}
      \caption{}
    \end{subfigure}
    \caption[Illustrations of the pentagonal hexecontahedron and its faces.]{
      Illustrations of (a) the pentagonal hexecontahedron with one face highlighted and (b) a face of the pentagonal hexecontahedron.
    }
    \label{fig:pentagonalHexecontahedron}
\end{minipage}
\par\medskip

Note that \(|v| = |v''| = \sqrt{3}\), and \(|v'| \approx 1.81614\) has algebraic degree \(12\).

The induced symmetry group of the faces of the pentagonal hexecontahedron is the trivial group.

\begin{theorem}
  \label{thm:pentagonalHexecontahedron}
  The number of ways of tiling the faces of the pentagonal hexecontahedron up to rotational icosahedral symmetry \(I\) with \(\tid{}\) tiles is given by the expression
  \begin{align*}
    &\frac{1}{60}
    \left(
      \tid{60} +
      24\tid{12} +
      20\tid{20} +
      15\tid{30}
    \right).
  \end{align*}
\end{theorem}
\begin{proof}
  Because each of the \(60\) faces of the pentagonal hexecontahedron corresponds to the fundamental domain of the rotational icosahedral group \(I\), the group action is free. Thus, for each \(A \in I\), the orbits of \(\langle A \rangle\) partition the faces into parts of size \(|A|\).
\end{proof}
\begin{oeis*}
  The number of \(n\)-colorings of the faces of the pentagonal hexecontahedron up to the \(60\) symmetries of the rotational icosahedral group \(h\) appears in the OEIS as sequence \oeis{A378478} (\(0\)-indexed): \[
    0, 1, 19215358678900736, 706519304586988199183738259, \ldots.
  \]
\end{oeis*}

\section{Polyhedra with dihedral symmetry}\label{sec:dihedral}
The bipyramids and trapezohedra consist of two infinite families of isohedral (face transitive) polyhedra. Here we count the tilings of these polyhedra via a bijection to tilings of cylindrical grids up to cyclic shifting and reflections of the grid \cite{Ethier,KageyKeehn}.
\subsection{Bipyramids}
Bipyramids are an infinite family of polyhedra, examples of which are shown in \Cref{fig:Bipyramids}.
The \(n\)-bipyramid is a face-transitive polyhedron with \(2n\) isosceles triangular faces and is the polyhedral dual of the \(n\)-prism.
The symmetry group of the \(n\)-bipyramid is the prismatic symmetry group \(D_{nh}\) of order \(4n\).

\par\medskip\noindent
\begin{minipage}{\textwidth}\captionsetup{type=figure}
  \centering
  \begin{subfigure}[b]{0.27\textwidth}
    \centering
    \includegraphics[scale=0.4]{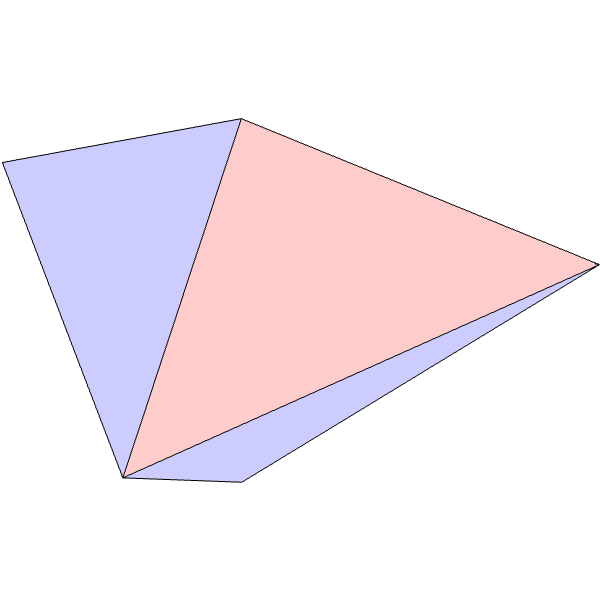}
    \caption{}\label{subfig:Bipyramid3}
  \end{subfigure}
  \begin{subfigure}[b]{0.3\textwidth}
    \centering
    \includegraphics[scale=0.4]{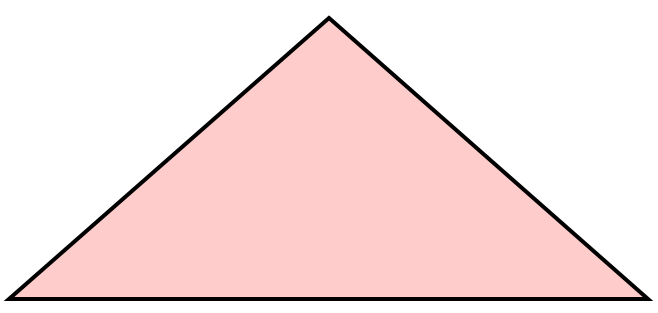}
    \caption{}\label{subfig:Bipyramid3Face}
  \end{subfigure}
    \begin{subfigure}[b]{0.16\textwidth}
    \centering
    \includegraphics[scale=0.4]{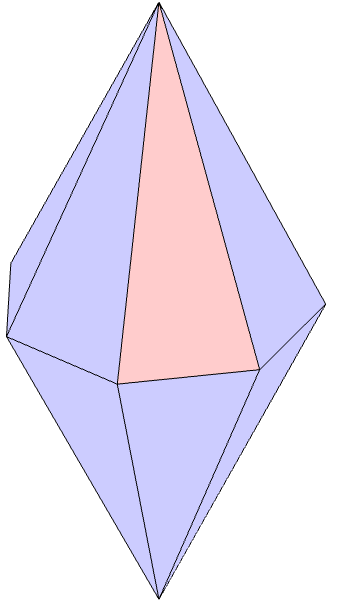}
    \caption{}\label{subfig:Bipyramid7}
  \end{subfigure}
  \begin{subfigure}[b]{0.14\textwidth}
    \centering
    \includegraphics[scale=0.4]{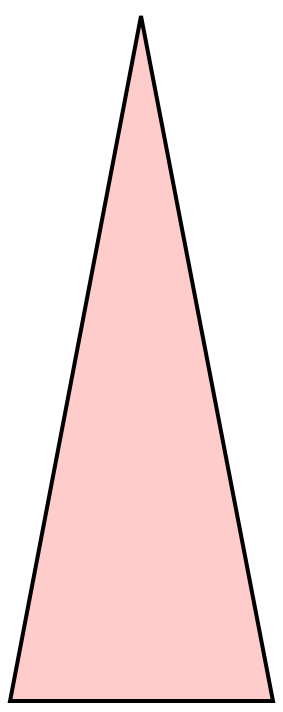}
    \caption{}\label{subfig:Bipyramid7Face}
  \end{subfigure}
  \caption[Illustrations of the trapezohedra and bipyramids and their faces.]{
    Illustrations of
    (a) the \(3\)-bipyramid with \(6\) faces,
    (b) a face of the \(3\)-bipyramid,
    (c) the \(7\)-bipyramid with \(14\) faces,and
    (d) a face of the \(7\)-bipyramid.
  }
  \label{fig:Bipyramids}
\end{minipage}
\par\medskip
We can describe the vertices of a bipyramid with poles at \((0,0,\pm z)\) and vertices along the equator at \(
    \left(
    \cos\!\left(\frac{2k}{n}\pi\right),
    \sin\!\left(\frac{2k}{n}\pi\right),
    0
    \right).
\)
When \(z = \cot(\pi/n)\), all dihedral angles of the \(n\)-bipyramid are equal.

\begin{theorem}
  For \(n \geq 3\), the number of ways of tiling the \(n\)-bipyramid up to \(D_{nh}\) symmetry with \(\tid{}\) isosceles triangular designs, \(t_f\) of which are fixed under reflection, is given by the following expressions. When \(n\) is even, \[
    \frac{1}{4n}\left(
      n\tid{n} +
      n\left(\frac{1}{2}\tid{n} + \frac{1}{2}\tid{n-2}\tf{4}\right) +
      \sum_{d|n}\left(
        \phi(d) \tid{2n/d} +
        \phi(d) \tid{2n/\lcm(d,2)}
      \right)
    \right),
  \] and when \(n\) is odd, \[
    \frac{1}{4n}\left(
      n\tid{n} +
      n\tid{n-1}\tf{2} +
      \sum_{d|n}\left(
        \phi(d) \tid{2n/d} +
        \phi(d) \tid{2n/\lcm(d,2)}
      \right)
    \right),
  \]
\end{theorem}
\begin{proof}
  The faces of the \(n\)-bipyramid are in bijection with squares of a \(n \times 2\) cylindrical grid, and the prismatic dihedral group \(D_{nh}\) acts on the cylindrical grid by cyclically shifting columns along with horizontal and vertical reflection. Therefore, the tilings of the bipyramid are in bijective correspondence with the tilings of the \(n \times 2\) cylindrical grid up to rotation and reflection, which is computed in Theorem 3.1.6 of our previous paper \cite{KageyKeehn}.
\end{proof}

\begin{oeis*}
  The number of \(c\)-colorings of the \(n\)-bipyramid up to prismatic symmetry (i.e., \(\tid{}=\tf{}=c\)) has been added to the OEIS as sequence \oeis{A395240}.
  \[
    \begin{array}{c|rrrrrrrr}
      & c=1 & c=2 & c=3 & c=4 & c=5 & c=6 & c=7 & c=8 \\ \hline
      n=3 & 1 & 13 & 92 & 430 & 1505 & 4291 & 10528 & 23052 \\
      n=4 & 1 & 34 & 549 & 4756 & 26725 & 111546 & 376369 & 1083664 \\
      n=5 & 1 & 78 & 3210 & 53764 & 493131 & 3037314 & 14158228 & 53762472 \\
      n=6 & 1 & 237 & 23337 & 709316 & 10229225 & 90932661 & 577499937 & 2865540112
    \end{array}
  \]
  When \(c = 2\), this recovers an enumeration by Ethier of the number of toroidal \(n \times 2\) binary arrays, up to rotation and reflection \cite{Ethier}, which appears as OEIS sequence \oeis{A222187}.
\end{oeis*}
\subsection{Trapezohedra}
Trapezohedra are an infinite family of polyhedra, examples of which are shown in \Cref{fig:Trapezohedra}. The \(n\)-trapezohedron is a face-transitive polyhedron with \(2n\) deltoidal faces, and is the polyhedral dual of the \(n\)-antiprism. The symmetry group of the \(n\)-trapezohedron is the antiprismatic symmetry group \(D_{nd}\) of order \(4n\).

\par\medskip\noindent
\begin{minipage}{\textwidth}\captionsetup{type=figure}
  \centering
  \begin{subfigure}[b]{0.22\textwidth}
    \centering
    \includegraphics[scale=0.4]{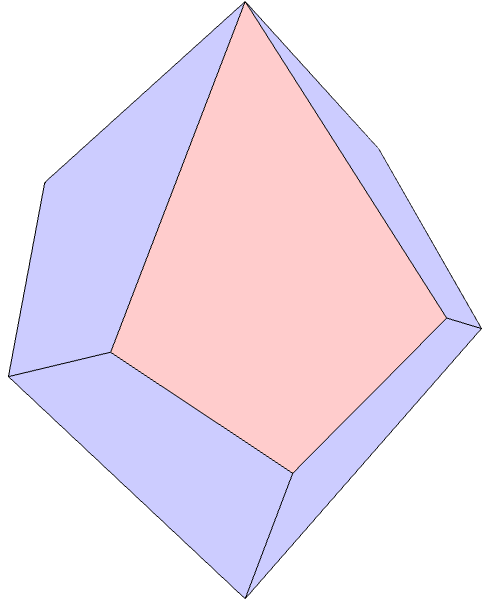}
    \caption{}\label{subfig:Trapezohedron4}
  \end{subfigure}
  \begin{subfigure}[b]{0.22\textwidth}
    \centering
    \includegraphics[scale=0.4]{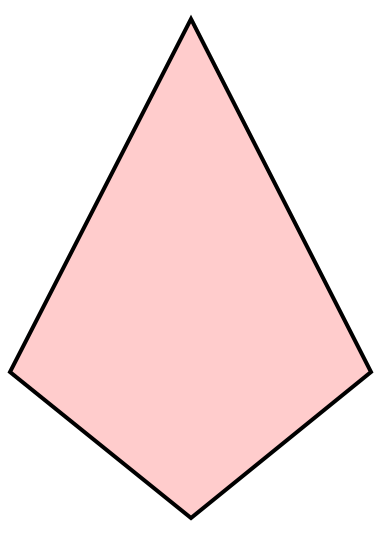}
    \caption{}\label{subfig:Trapezohedron4Face}
  \end{subfigure}
  \begin{subfigure}[b]{0.18\textwidth}
    \centering
    \includegraphics[scale=0.4]{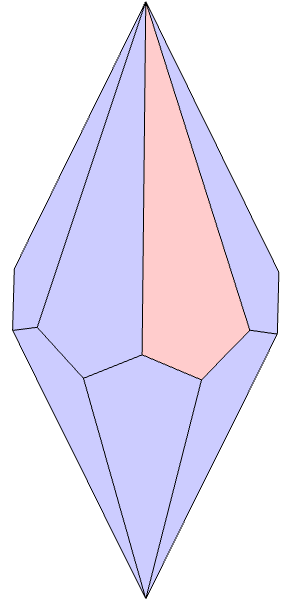}
    \caption{}\label{subfig:Trapezohedron7}
  \end{subfigure}
  \begin{subfigure}[b]{0.18\textwidth}
    \centering
    \includegraphics[scale=0.4]{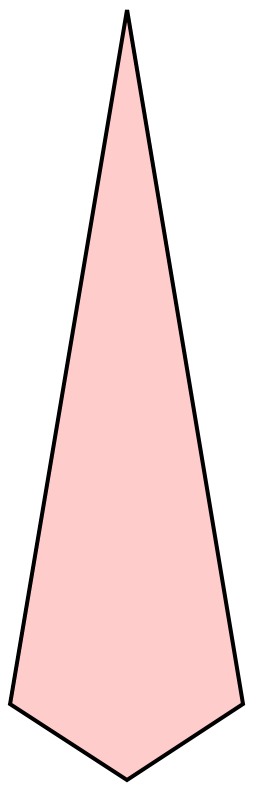}
    \caption{}\label{subfig:Trapezohedron7Face}
  \end{subfigure}
  \caption[Illustrations of the trapezohedra and bipyramids and their faces.]{
    Illustrations of
    (a) the \(4\)-trapezohedron with \(8\) faces,
    (c) a face of the \(4\)-trapezohedron,
    (b) the \(7\)-trapezohedron with \(14\) faces, and
    (d) a face of the \(7\)-trapezohedron.
  }
  \label{fig:Trapezohedra}
\end{minipage}
\par\medskip
We can describe the vertices of a \(n\)-trapezohedron with a north and south pole of \((0,0,z)\) and \((0,0,-z)\) respectively, and near-equator vertices at \[
    \left(
    \cos\!\left(\frac{2k}{n}\pi\right),
    \sin\!\left(\frac{2k}{n}\pi\right),
    z'
    \right)
    \quad\text{and}\quad
    \left(
    \cos\!\left(\frac{2k + 1}{n}\pi\right),
    \sin\!\left(\frac{2k + 1}{n}\pi\right),
    -z'
    \right)
\] for all \(0 \leq k < n\), where \(z\) and \(z'\) are chosen so that the convex hull forms an \(n\)-trapezohedron. To make all dihedral angles equal (as is the convention for Catalan solids), we can choose \[
  z = \frac{\cos(\pi/n)+1}{\sqrt{2\cos(\pi/n)-2\cos(2\pi/n)}}
  \quad \text{and} \quad
  z' = \frac{\sqrt{1-\cos(\pi/n)}}{\sqrt{4 \cos(\pi/n)+2}}.
\]

We can give a matrix representation of the antiprismatic group \[
  D_{nd} \cong D_{2n} = \langle R, F \mid R^{2n} = F^2 = (RF)^2 = \mathrm{Id}\rangle
\] by choosing \(F\) to be a mirror symmetry and \(R\) to be a rotoreflection: \[
  F = \begin{bmatrix}
    1 & 0 & 0 \\
    0 & -1 & 0 \\
    0 & 0 & 1
  \end{bmatrix}
  \quad \text{and} \quad
  R = \begin{bmatrix}
    \cos (\pi/n) & -\sin (\pi/n) & 0 \\
    \sin (\pi/n) & \cos (\pi/n) & 0 \\
    0 & 0 & -1 \\
  \end{bmatrix}.
\]

\begin{theorem}
  For \(n \geq 3\), the number of ways of tiling the \(n\)-trapezohedron up to \(D_{nd}\) (antiprismatic symmetry) with \(\tid{}\) kite tile designs, \(t_f\) of which are fixed under reflection, is given by the expression \[
    \frac{1}{4}\left(\tid{n} + \tid{n-1}\tf{2}\right)
    +
    \frac{1}{4n}\sum_{d|2n}\phi(d)\tid{2n/d}.
  \]
\end{theorem}
\begin{proof}
  The tilings of the \(n\)-trapezohedron under \(D_{nd} \cong D_{2n}\) have a natural bijection to the tilings \(2n \times 1\) cylinder under horizontal reflection and are computed in Theorem 3.1.6 of our previous paper \cite{KageyKeehn}.
\end{proof}
The bijection is illustrated in \Cref{fig:trapezohedronCylinderBijection}.

\begin{figure}
  \centering
  \begin{subfigure}[b]{0.49\textwidth}
    \centering
    \includegraphics[height=6cm]{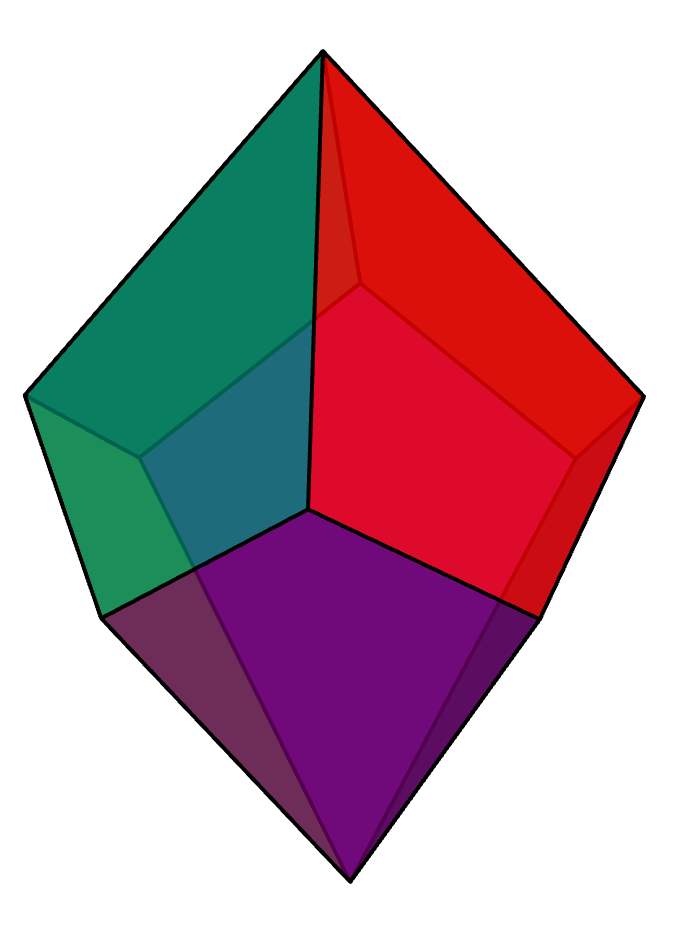}
    \caption{A view of the \(4\)-trapezohedron from the side}
  \end{subfigure}
  \begin{subfigure}[b]{0.49\textwidth}
    \centering
    \includegraphics[height=6cm]{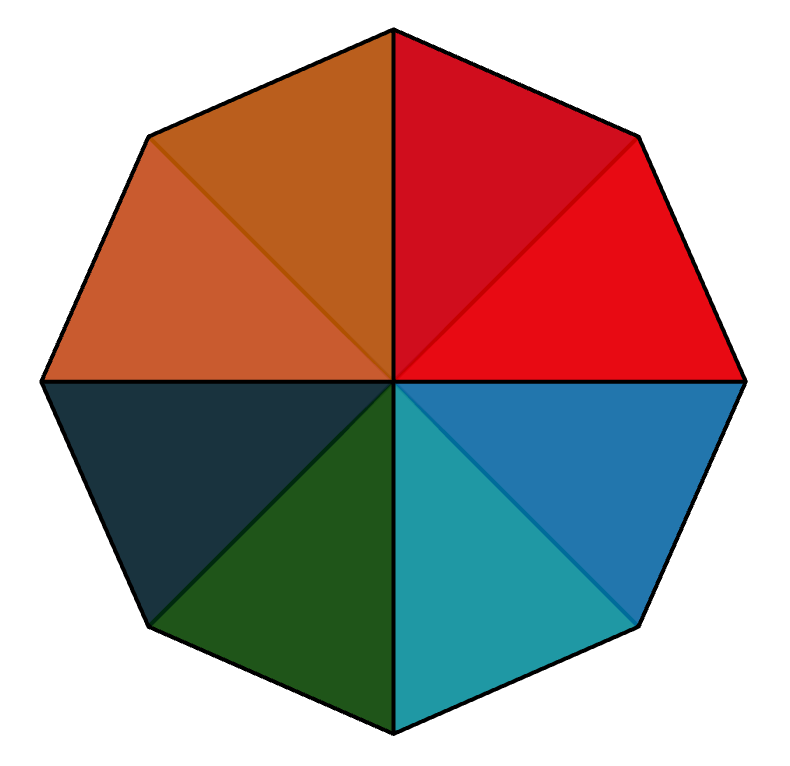}
    \caption{A view of the \(4\)-trapezohedron from the top}
  \end{subfigure}
  \caption{Two views of the \(4\)-trapezohedron illustrating that the transitive action of the antiprismatic symmetry group \(D_{4d}\) on the faces corresponds to the dihedral group of the octahedron \(D_8\) on the projected shadow.}
  \label{fig:trapezohedronCylinderBijection}
\end{figure}

\begin{oeis*}
  The number of \(c\)-colorings of the \(n\)-trapezohedron up to antiprismatic symmetry, (i.e., \(\tid{}=\tf{}=c\)) has been added to the OEIS as sequence \oeis{A396913}.
  \[
    \begin{array}{c|rrrrrrrr}
      & c=1 & c=2 & c=3 & c=4 & c=5 & c=6 & c=7 & c=8 \\ \hline
      n=3 & 1 & 13 & 92 & 430 & 1505 & 4291 & 10528 & 23052 \\
      n=4 & 1 & 30 & 498 & 4435 & 25395 & 107331 & 365260 & 1058058 \\
      n=5 & 1 & 78 & 3210 & 53764 & 493131 & 3037314 & 14158228 & 53762472 \\
      n=6 & 1 & 224 & 22913 & 704370 & 10196680 & 90782986 & 576960734 & 2863912668
    \end{array}
  \]
\end{oeis*}

\section{Further directions}
The strategy used here in the case of face-transitive polyhedra can be extended to polyhedra that are not face-transitive by specifying tile designs for each of the equivalence classes of faces up to congruence.

\subsection{Archimedean solids}
\label{subsec:Archimedean}
Archimedean solids have two or more types of faces, and so they require specifying the tile designs for each face type. Note that in the case of the truncated icosahedron (shown in \Cref{fig:jonPaulBall}) the rotational icosahedral group \(I\) induces \(C_5\) cyclic symmetry on the pentagonal faces, but only \(C_3\) cyclic symmetry on the hexagonal faces, because rotating a hexagonal face by \(60^\circ\) does not constitute an isometry of the polyhedron.
\begin{theorem}
  The number of ways of tiling the \(20\) hexagonal and \(12\) pentagonal faces of the truncated icosahedron up to rotational icosahedral symmetry \(I\), with \(t_{6,\id}\) hexagonal tile designs and \(t_{5,\id}\) pentagonal designs such that  \(t_{6,r}\) of the hexagonal designs are fixed under \(120^\circ\) rotation and \(t_{5,r}\) of the pentagonal designs are fixed under \(72^\circ\) rotation is given by the expression
  \begin{align*}
    &\frac{1}{60}
    \left(
      t_{6,\id}^{20}          t_{5,\id}^{12} +
      24t_{6,\id}^4           t_{5,\id}^2 t_{5,r}^2 +
      20t_{6,\id}^6 t_{6,r}^2 t_{5,\id}^4 +
      15t_{6,\id}^{10}        t_{5,\id}^6
    \right).
  \end{align*}
\end{theorem}
\begin{oeis*}
  When \(t_{6,\id} = t_{6,r} = t_{5,\id} = t_{5,r} = n\), this gives the number of \(n\)-colorings of the faces of the truncated icosahedron up to the \(60\) symmetries of the rotational icosahedral group \(I\) and has been added to the OEIS as sequence \oeis{A396861} (\(0\)-indexed): \[
    0, 1, 71600640, 30883680755649, 307445735641186304, \dots
  \]
\end{oeis*}
\begin{example}
The truncated icosahedron is the classic ``soccer ball'' shape, usually made with black pentagonal faces and white hexagonal faces. A soccer ball designer, Jon-Paul Wheatley, hand-sewed a truncated icosahedral ball \cite{JonPaulsBalls}, illustrated in Figure \ref{fig:jonPaul}, where the hexagonal faces were chosen to be hexagonal Truchet tiles and the pentagonal faces have dihedral symmetry. Up to rotation (but not reflection) there are 58\,127\,868 distinct soccer balls with these faces, since \(t_{6,\id}=3\), \(t_{6,r}=0\), and \(t_{5,\id} = t_{5,r} = 1\).
\end{example}
\begin{figure}
  \centering
  \begin{subfigure}[b]{0.3\textwidth}
    \includegraphics[width=\linewidth]{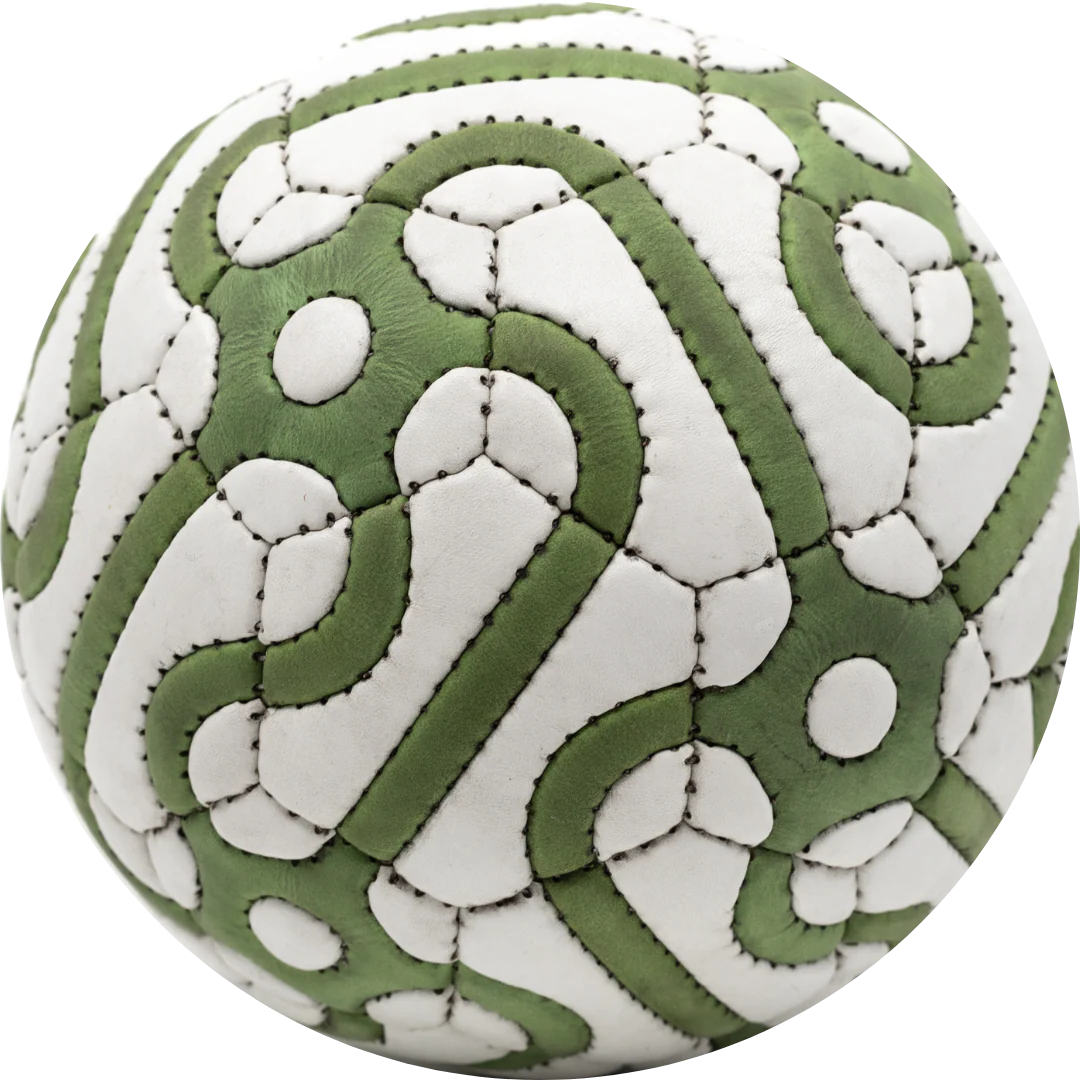}
    \caption{}
    \label{fig:jonPaulBall}
  \end{subfigure}
  \hfill
  \begin{subfigure}[b]{0.3\textwidth}
    \includegraphics[width=\linewidth]{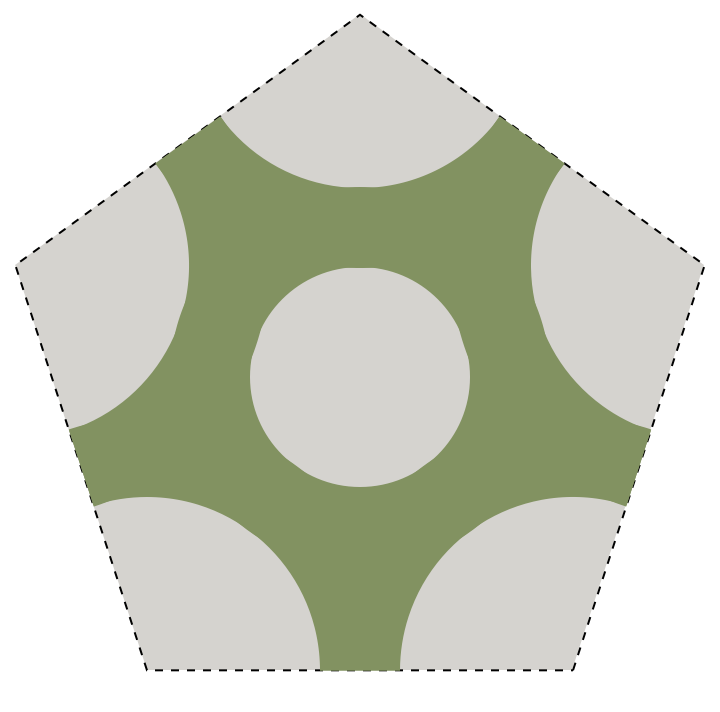}
    \caption{}
    \label{fig:jonPaulFace5}
  \end{subfigure}
  \hfill
  \begin{subfigure}[b]{0.3\textwidth}
    \includegraphics[width=\linewidth]{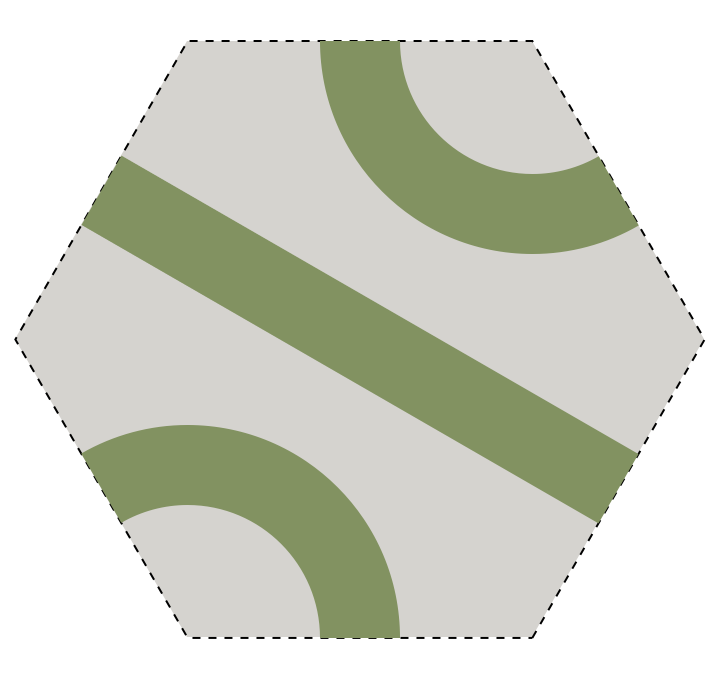}
    \caption{}
    \label{fig:jonPaulFace6}
  \end{subfigure}
  \caption{A soccer ball (a) designed and sewn by Jon-Paul Wheatley, featuring 12 pentagonal panels with full dihedral symmetry (b) and 20 panels that are hexagonal Truchet tiles (c). There are 58\,127\,868 distinct designs that can be made using these panels, up to the \(60\) symmetries of the rotational icosahedral group. (Photo courtesy of Jon-Paul Wheatley.)}
\label{fig:jonPaul}
\end{figure}
\begin{example}
  In 1997, Tantrix produced a puzzle, shown in \Cref{fig:TantrixRocks}, where players place tiles on the faces of a truncated octahedron.

  Ignoring the colors of the tile designs, there were two essentially different kinds of square tile designs, which appear in \Cref{fig:TantrixSquare1,fig:TantrixSquare2}, and three essentially different kinds of hexagonal tile designs, which appear in \Cref{fig:TantrixHexagon1,fig:TantrixHexagon2,fig:TantrixHexagon3}.

  \begin{alignat*}{3}
    t_{4,\id} &= 1 + 2 = 3 \qquad &&t_{4,r} = 1 \qquad
    &&t_{4,r^2} = 3
    \\
    t_{6,\id} &= 2 + 3 + 6 = 9 \qquad
    &&t_{6,r} = 2 + 0 + 0 = 2 \qquad &&
  \end{alignat*}

  We compute that there are
  \[
    \frac{1}{24} \left(
       t_{4,\id}^6            t_{6,\id}^8 +
      6t_{4,\id}t_{4,r}^2     t_{6,\id}^2 +
      3t_{4,\id}^2t_{4,r^2}^2 t_{6,\id}^4  +
      8t_{4,\id}^2            t_{6,\id}^2t_{6,r}^2 +
      6t_{4,\id}^3            t_{6,\id}^4
    \right) = 653\,827\,950.
  \] number of distinct ways of tiling the truncated octahedron with these tiles up to rotation but not reflection. (The formula also recovers OEIS sequence \oeis{A316093}, when \(t_{4,\id} = t_{4,r} = t_{4,r^2} = n\) and \(t_{4,\id} = t_{6,r} = k\).)
  \begin{figure}
    \begin{subfigure}[b]{0.08\linewidth}
      \centering
      \includegraphics[width=\linewidth]{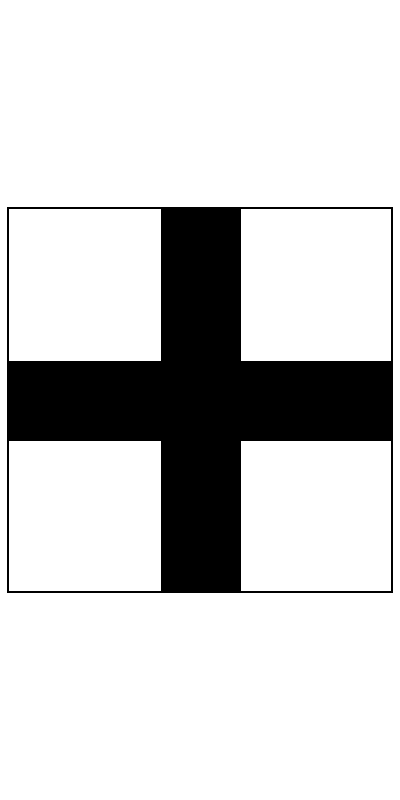}
      \caption{}\label{fig:TantrixSquare1}
    \end{subfigure}
    \hfill
    \begin{subfigure}
      [b]{0.08\linewidth}
      \centering
      \includegraphics[width=\linewidth]{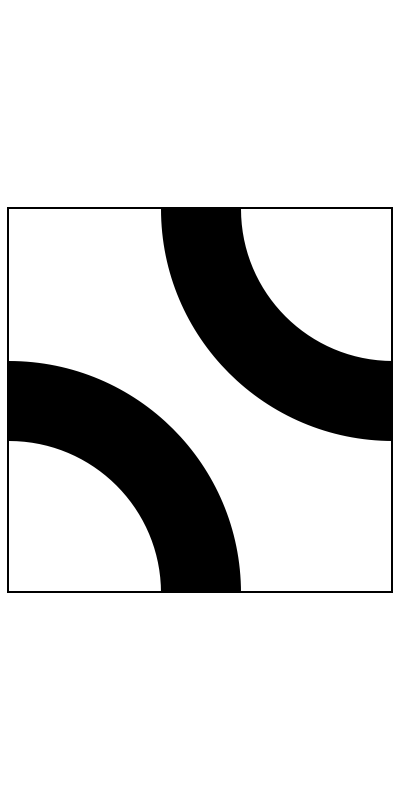}
      \caption{}\label{fig:TantrixSquare2}
    \end{subfigure}
    \hfill
    \begin{subfigure}[b]{0.14\linewidth}
      \centering
      \includegraphics[width=\linewidth]{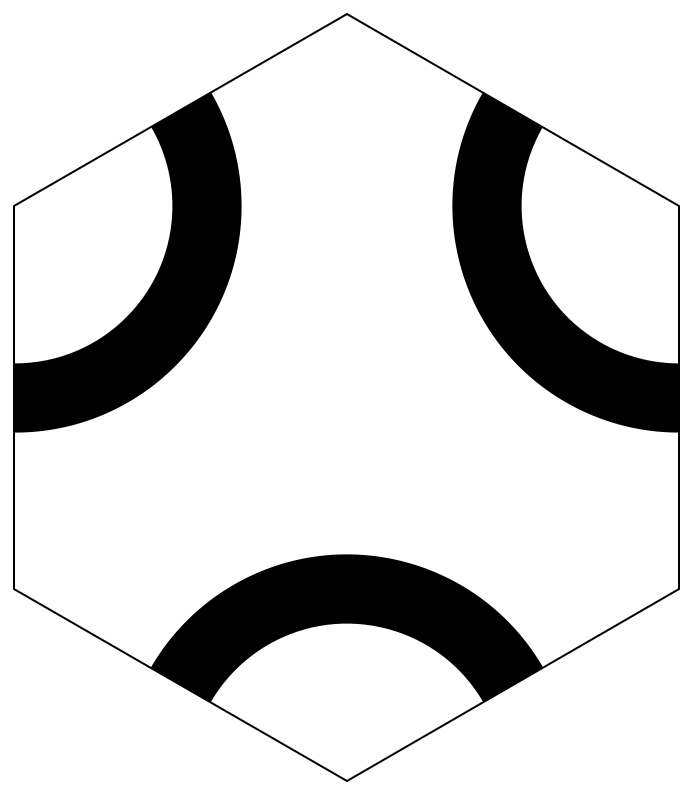}
      \caption{}\label{fig:TantrixHexagon1}
    \end{subfigure}
    \hfill
    \begin{subfigure}[b]{0.14\linewidth}
      \centering
      \includegraphics[width=\linewidth]{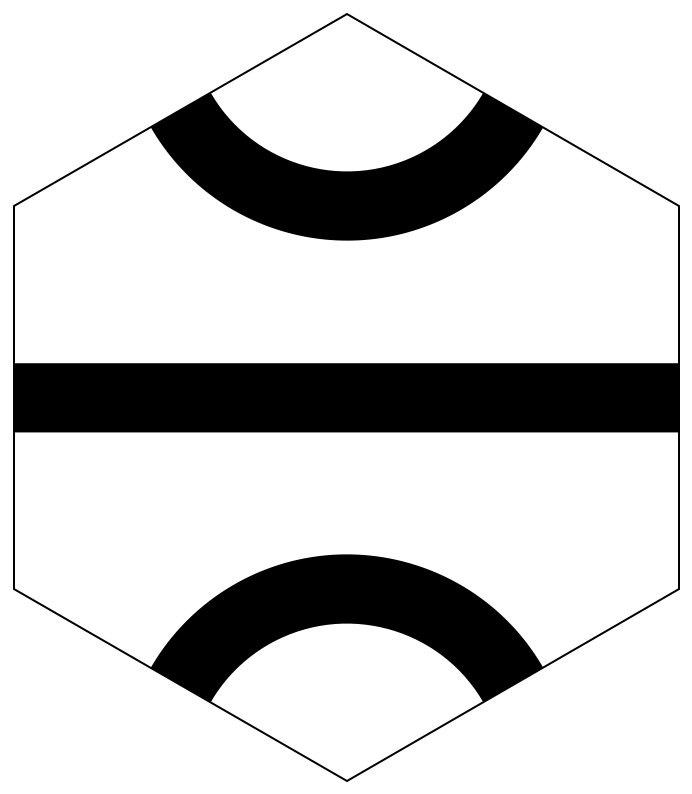}
      \caption{}\label{fig:TantrixHexagon2}
    \end{subfigure}
    \hfill
    \begin{subfigure}[b]{0.14\linewidth}
      \centering
      \includegraphics[width=\linewidth]{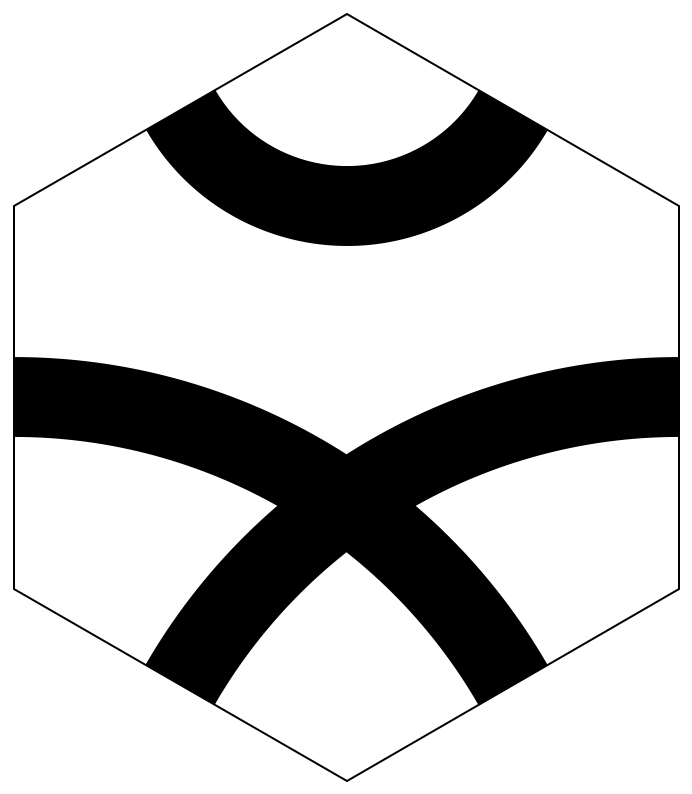}
      \caption{}\label{fig:TantrixHexagon3}
    \end{subfigure}
    \hfill
    \begin{subfigure}[b]{0.273\linewidth}
      \centering
      \includegraphics[width=\linewidth]{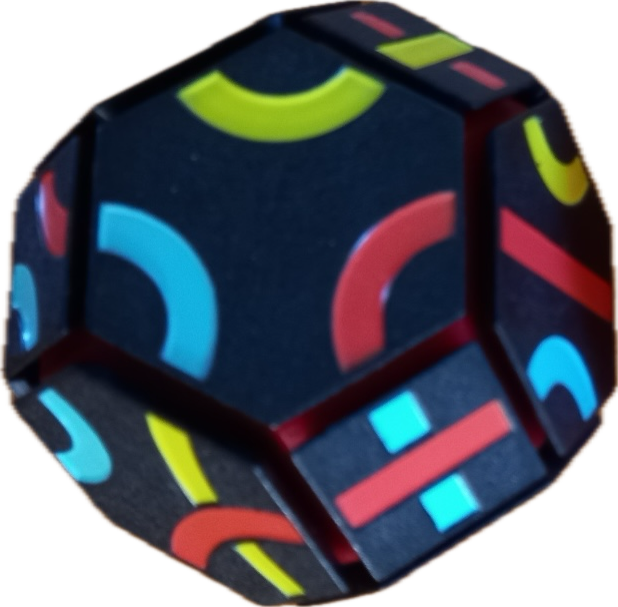}
      \caption{}\label{fig:TantrixRocks}
    \end{subfigure}
    \caption{The five essentially different tile designs (a) to (e) used in the Tantrix Rocks puzzle (f). (Photo courtesy of Mike McManaway.)}
    \label{fig:TantrixTiles}
  \end{figure}
\end{example}

\subsection{Higher dimensional tiles}
While we demonstrated how to count the number of tilings of cubes with \(n \times n\) grids on each face, we would be remiss not to remind the reader of Schattschneider's suggestion of counting tilings of \(n \times n \times n\) arrangements of cubes, where the ``tiles'' are \(3\)-dimensional structures internal to the cubes \cite{Schattschneider}. This construction has applications to chemistry, as described by Meekel et al.\ \cite{Meekel2023}.

\section*{Acknowledgements}
The authors would like to thank the Prison Mathematics Project for helping to facilitate our conversations, along with Matt Zucker, Jon-Paul Wheatley, and Mike McManaway for their examples and photographs.

{
  \setlength{\baselineskip}{13pt}
  \raggedright
  \bibliographystyle{plain}
  \bibliography{references_abbreviated.bib}
}
\end{document}